\documentclass[reqno,11pt,graphicx,]{amsart}
\usepackage[colorlinks,linkcolor=blue,anchorcolor=yellow,citecolor=red,urlcolor=green]{hyperref} 
\usepackage{amsmath,amscd,amsthm,amsfonts,appendix,tikz}
\usepackage{amssymb,mathrsfs,xcolor,bbm,dsfont,verbatim}
\usepackage{enumitem}
\usepackage[OT2,OT1]{fontenc}
\usepackage{caption}
\usepackage{longtable}
\usepackage[all]{xy}
\numberwithin{equation}{section}
\allowdisplaybreaks[4]
\theoremstyle{plain}
\newcommand{\N}{{\mathbb N}}
\newcommand{\Z}{{\mathbb Z}}

\newcommand{\R}{{\mathbb R}}

\newcommand{\I}{{\mathbb I}}
\newcommand{\Q}{{\mathbb Q}}

\newcommand{\B}{{\mathcal B}}

\newcommand{\A}{{\mathscr A}}

\newcommand{\p}{\mathbf P}

\newcommand{\tT}{{\mathbf t}}

\newcommand{\ma}{\mathbf{a}}
\newcommand{\mw}{x}
\newcommand{\mv}{\mathbf{v}}

\newtheorem{theorem}{Theorem}[section]
\newtheorem{proposition}[theorem]{Proposition}
\newtheorem{lemma}[theorem]{Lemma}
\newtheorem{corollary}[theorem]{Corollary}
\newtheorem{remark}[theorem]{Remark}
\newtheorem{definition}[theorem]{Definition}

\allowdisplaybreaks

\begin{document}
	\title{The spectral characteristics of the Sturm Hamiltonian with eventually periodic type}
	
	\author{Jie CAO}
	\address[J. Cao]{Chern Institute of Mathematics and LPMC, Nankai University, Tianjin 300071, P. R. China.}
	\email{caojie@nankai.edu.cn}
	
	\author{Zhenyu Yu$^*$}
	\address[Z. Yu]{College of Science,
		National University of Defense Technology, Changsha 410073, P. R. China.}
	\email{yuzy23@nudt.edu.cn}
	
	\thanks{* Corresponding author.}
	
	\begin{abstract}
		In this paper, we consider the spectral characteristics of the Sturm Hamiltonian with eventually periodic type frequencies under large coupling, and establish strict inequalities among the optimal Hölder exponent of the density of states measure, the dimension of the density of states measure, the dimension of the spectrum, and the transport exponent by analyzing the thermodynamical pressure function. Also, we provide the large coupling asymptotic properties of the four spectral characteristics.
	\end{abstract}
	
	\keywords{Sturm Hamiltonian, density of states measure, spectral characteristics, thermodynamical formalism.}
	\maketitle
	\section{Introduction}
	
	The {\it Sturm Hamiltonian} is a discrete Schr\"odinger operator defined on $\ell^2(\Z)$ with Sturmian potential:
	\begin{equation*}
		(H_{\alpha,\lambda,\theta}\psi)_n:=\psi_{n+1}+\psi_{n-1}+\lambda \chi_{[1-\alpha,1)}(n\alpha+\theta\pmod 1)\psi_n,
	\end{equation*}
	where $\alpha\in\I:=[0,1]\setminus\Q$ is the {\it frequency}, $\lambda >0$ is the {\it coupling constant} and $\theta\in [0,1)$ is the {\it phase} (where $\chi_A$ is the indicator function of the set $A$). It is well-known that the spectrum of $H_{\alpha,\lambda,\theta}$ is independent of $\theta$ (see \cite{BIST1989}).
	We denote the spectrum by $\Sigma_{\alpha,\lambda}$.
	
	Another spectral object  is the so-called {\it density of states measure (DOS)} $\mathcal{N}_{\alpha,\lambda}$ supported on $\Sigma_{\alpha,\lambda}$, which is defined by
	\begin{equation*}\label{def-dos}
		\int_{\Sigma_{\alpha,\lambda}} g(x)\,\mathrm{d} \mathcal{N}_{\alpha,\lambda}(x):=\int_{\mathbb T}\langle \delta_0,g(H_{\alpha,\lambda,\theta})\delta_0\rangle\,\mathrm{d}\theta, \ \ \ \forall g\in C(\Sigma_{\alpha,\lambda}).
	\end{equation*}
	
	Since the operator $H_{\alpha,\lambda,\theta}$ has purely singular continuous spectrum (see \cite{DKL2000}), the RAGE Theorem (see, e.g., \cite[Theorem XI.115]{RS1979}) implies that when studying the Schr\"odinger time evolution associated with this Schr\"odinger operator, i.e., $e^{-itH_{\alpha,\lambda,\theta}} \psi$ for some initial state $\psi \in \ell^2(\mathbb{Z})$ one should consider time-averaged quantities. For simplicity, we focus on initial states of the form $\delta_n$ with $n \in \mathbb{Z}$. Since spatial translation merely induces a phase adjustment, we may without loss of generality restrict our attention to the specific case $\psi = \delta_0$. The time-averaged spreading of $e^{-itH_{\alpha,\lambda,\theta}} \delta_0$ is typically characterized on a power-law scale, as illustrated in, for example, \cite{La1996,DT2010}. For $ p > 0$, we consider the $ p$-th moment of the position operator,
	\begin{equation*}
		\langle |X|^p\rangle(t):=\sum_{n\in{\mathbb{Z}}}|n|^p\langle
		e^{-itH_{\alpha,\lambda,\theta}}\delta_0,\delta_n\rangle|^2.
	\end{equation*}
	Then, define the upper and lower transport exponents $\beta
	^+(\lambda,p)$ and $\beta^-(\lambda,p)$ as:
	\begin{equation*}
		\beta^+(\lambda,p):=\limsup\limits_{T\to\infty}\dfrac{
			\log\frac{2}{T}\int_{0}^{\infty}e^{-\frac{2t}{T}}\langle|X|^p\rangle(t)\mathrm{d}t}{p\log T},\
		\beta
		^-(\lambda,p):=\liminf\limits_{T\to\infty}\dfrac{\log\frac{2}{T}\int_{0}^{\infty}e^{-\frac{2t}{T}}\langle|X|^p\rangle(t)\mathrm{d}t}{p\log T}.
	\end{equation*}
	The transport exponents $\beta^{\pm}(\lambda,p)$ belong to $[0,1]$
	and are non-decreasing in $p$ (see, e.g.,\cite{DT2010}), and hence the
	following limits exist uniformly in $\theta\in[0,1)$: 
	\begin{equation*}
		\mathcal{T}^{+}(\alpha,\lambda):=\lim\limits_{p%
			\to\infty}\beta^{+}(p,\lambda),\ \ \ \ 	\mathcal{T}^{-}(\alpha,\lambda):=\lim\limits_{p%
			\to\infty}\beta^{-}(p,\lambda). 
	\end{equation*} 
	Ballistic transport corresponds to transport exponents being equal to one, diffusive
	transport corresponds to the value $1/2$, and vanishing transport exponents correspond to dynamical localization. In all other cases, transport is called anomalous.
	
	Assume  $\mu$ is a finite Borel measure supported  on a  metric space $X$. We define the {\it  lower} and {\it upper } local dimensions of $\mu$ at $x\in X$ as
	\begin{equation}\label{def-loc-dim}
		\underline{d}_\mu(x):=\liminf_{r\to0}\frac{\log \mu(B(x,r))}{\log r}\ \ \ \text{ and }\ \ \ \overline{d}_\mu(x):=\limsup_{r\to0}\frac{\log \mu(B(x,r))}{\log r}.
	\end{equation}
	If  $\underline{d}_\mu(x)=\overline{d}_\mu(x)$, we say that the {\it local  dimension} of   $\mu$ at $x$ exists and denote it by $d_\mu(x)$. 
	The {\it Hausdorff dimension} of $\mu$ is defined as
	\begin{equation}\label{dim-meas}
		\dim_H\mu:=\sup\{s: \underline{d}_\mu(x)\ge s \text{ for  } \mu \text{ a.e. }x\in X\}.
	\end{equation}
	If there exists a constant $d$ such that $d_\mu(x)=d$ for $\mu$ a.e. $x\in X$, then necessarily $\dim_H\mu=d$, and $\mu$ is called {\it exact-dimensional} (see \cite[Chapter 10]{Fa1997} for more details).
	
	In this paper, we will consider the frequency of eventually periodic type. 
	For each $\mathbf{a}=a_1a_2\cdots a_k\in\N^k$, define the eventually periodic frequencies with sequence $\ma$ as
	\begin{equation*}
		\mathcal{EP}(\mathbf{a}):=\Big\{\alpha\in{\mathbb{I}}:\alpha=[b_1,\cdots,
		b_m,\, \overline{a_1,a_2,\cdots,a_k}],\, b_i\in{\mathbb{N}},\, 1\leq i\leq m,\,m\in\N
		\Big\},
	\end{equation*} 
	where the overline notation denotes infinite repetition of the periodic block. In particular, we define $\mathcal{P}(\ma)=\{\alpha\in{\mathbb{I}}:\alpha=[\overline{a_1,a_2,\cdots,a_k}]
	\}$ as the set of periodic frequencies. Let $\mathcal{EP}=\bigcup_{k=1}^{\infty}\bigcup_{\ma\in\N^k}\mathcal{EP}(\ma)$
	be the set of frequencies with eventually periodic type. 
	
	Denote the optimal H\"older exponent of the DOS $\mathcal{N}_{\alpha,\lambda}$ by 
	\begin{equation}\label{def-gamma}
		\gamma(\alpha,\lambda):=\inf\{\underline{d}_{\mathcal{N}_{\alpha,\lambda}}(x):x\in \Sigma_{\alpha,\lambda}\}.
	\end{equation}
	For any $\alpha\in\mathcal{EP}$, we will study the dimensional properties of $\Sigma_{\alpha,\lambda}$ and $\mathcal{N}_{\alpha,\lambda}$, and the optimal H\"older exponent of $\mathcal{N}_{\alpha,\lambda}$,
	as well as $\mathcal{T}^{\pm}(\alpha,\lambda)$ for large
	coupling constant $\lambda$. Write
	\begin{align*}
		D(\alpha,\lambda):=\dim_{H}\Sigma_{\alpha,\lambda}\ \ \text{and}\ \  d(\alpha,\lambda):=\dim_H\mathcal{N}_{\alpha,\lambda}.
	\end{align*}
	\subsection{Background and previous results}\
	
	The Sturm Hamiltonian, a classic model for one-dimensional quasicrystals, has been studied since the 1980s (see \cite{Da2007, Da2017}). Here, we summarize key findings about its spectral properties, focusing on the Fibonacci Hamiltonian-the most studied Sturmian model. This model uses the golden ratio $\alpha_1:=(\sqrt{5}-1)/2$ as its frequency. The model was first proposed to describe quasicrystals (see \cite{BIST1989,KKT1983,OPRSS1983}), and the first papers on the model in the mathematics literature belong to Casdagli \cite{Ca1986} and S\"ut\H{o} \cite{Su1987}. Its spectral structure is now fully understood  \cite{DGY2016}, with additional contributions from earlier studies \cite{Ca2009,DEGT2008,DG2009,DG2011,DG2012,DG2013,JL2000,Po2015}. We restate key results (Theorem A) relevant to our analysis, briefly introducing the trace map dynamics—a key tool for studying Fibonacci Hamiltonian.
	
	Define the {\it Fibonacci trace} map $\mathbf T: \R^3\to \R^3$ as $
	\mathbf T(x,y,z):=(2xy-z,x,y).$
	It is known that for any $\lambda>0$, the map $\mathbf T$ preserves the cubic surface
	\begin{equation*}
		S_\lambda:=\{(x,y,z)\in \R^3: x^2+y^2+z^2-2xyz-1=\lambda^2/4\}.
	\end{equation*}
	Write $\mathbf T_\lambda:=\mathbf T|_{S_\lambda}$ and let $\Lambda_\lambda$ be the set of points in $S_\lambda$ with bounded $\mathbf T_\lambda$-orbits. It is known that $\Lambda_\lambda$ is the non-wandering set of $\mathbf T_\lambda$ and is a locally maximal compact transitive hyperbolic set of $\mathbf T_\lambda$, see \cite{Ca2009,Ca1986,DG2009,Me2014}. 
	
	Let $\mu_{\lambda,\max}$ be the measure of maximal entropy of $\mathbf{T}_{\lambda}|_{\Lambda_{\lambda}}$ and $\mu_{\lambda}$ be the equilibrium measure of $\mathbf{T}_{\lambda}|_{\Lambda_\lambda}$ that corresponds to the potential $-D(\alpha_1,\lambda)\log\|D\mathbf{T}_{\lambda}|_{E^u}\|$.  Denote the set of periodic points of the map $\mathbf{T}_{\lambda}$ by $Per(\mathbf{T}_{\lambda})$.
	Assume that ${\rm Lyap}^u(p)$ is the unstable (positive) Lyapunov exponent of the periodic point $p$, and ${\rm Lyap}^u\mu_{\lambda}$(or ${\rm Lyap}^u\mu_{\lambda,\max}$) is the unstable Lyapunov
	exponent of $\mu_{\lambda}$ (respectively, $\mu_{\lambda,\max})$. 
	
	\smallskip
	\noindent {\bf Theorem A}(\cite{DGY2016})\  {\it
		Assume that $\lambda>0$. Then we have
		
		(i) The DOS $\mathcal{N}_{\alpha_1,\lambda}$ is exact-dimensional and $d(\alpha_1,\lambda)$ satisfies Young's formula:
		\begin{equation*}\label{dim-dos-golden}
			d(\alpha_1,\lambda)=\frac{-\log \alpha_1}{{\rm Lyap}^u\mu_{\lambda,\max}}.
		\end{equation*}
		
		(ii) The optimal H\"older exponent of $\mathcal{N}_{\alpha_1,\lambda}$ satisfies 
		\begin{equation*}\label{asym-golden-gamma}
			\gamma(\alpha_1,\lambda)=\frac{-\log\alpha_1}{\sup\limits_{p\in Per(\mathbf{T}_{\lambda})}{\rm Lyap}^u(p)}.
		\end{equation*}
		
		(iii) The spectrum $\Sigma_{\alpha_1,\lambda}$ satisfies
		$
		\dim_H \Sigma_{\alpha_1,\lambda}=\overline{\dim}_B \Sigma_{\alpha_1,\lambda}
		$
		and 
		\begin{equation*}\label{asym-golden-spectra}
			D(\alpha_1,\lambda)=\frac{h_{\mu_{\lambda}}(\mathbf{T}_{\lambda}|_{\Lambda_\lambda})}{{\rm Lyap}^u\mu_{\lambda}}.
		\end{equation*}
		
		(iv) $\mathcal{T}^-(\alpha_1,\lambda)$ and $\mathcal{T}^+(\alpha_1,\lambda)$ are equal and independent of $\theta\in[0,1)$. Moreover, 
		\begin{equation*}
			\mathcal{T}^{\pm}(\alpha_1,\lambda)=\frac{-\log\alpha_1}{\inf\limits_{p\in Per(\mathbf{T}_{\lambda})}{\rm Lyap}^u(p)}.
		\end{equation*}
		
		(v) The following inequalities hold:
		\begin{equation*}
			\gamma(\alpha_1,\lambda)<d(\alpha_1,\lambda)<D(\alpha_1,\lambda)<\mathcal{T}^{\pm}(\alpha_1,\lambda).
		\end{equation*}
		
		(vi) The following asymptotics hold:
		\begin{align*}
			&\lim\limits_{\lambda\to\infty}\gamma(\alpha_1,\lambda)\cdot\log\lambda=-\frac{3}{2}\log\alpha_1,\ \ \ \ \ \ \lim\limits_{\lambda\to\infty}d(\alpha_1,\lambda)\cdot\log \lambda=-\frac{5+\sqrt{5}}{4}\log \alpha_1,\\
			&\lim_{\lambda\to \infty} D(\alpha_1,\lambda)\cdot\log \lambda=\log (1+\sqrt{2}),\ \ \ \  \  \lim\limits_{\lambda\to\infty}\mathcal{T}^{\pm}(\alpha_1,\lambda)\cdot\log\lambda=-2\log\alpha_1.
		\end{align*}
	}
	
	There are several works that deal with sub-classes of the Sturm Hamiltonian. Parts (i)-(iii) of Theorem A hold for any constant type or periodic-type Sturmian Hamiltonian (small $\lambda$), a result that can be derived by combining findings from \cite{Ca2009} and \cite{DGY2016}, see also \cite{Me2014}. Girand \cite{Gi2014} considered eventually periodic type $\alpha$, he showed that
	$\mathcal{N}_{\alpha,\lambda}$ is exact-dimensional for small $\lambda$.
	For eventually constant type $\alpha$, Qu \cite{Qu2016} obtained results analogous to parts (i)-(iii) and (vi) of Theorem A for $\lambda > 20$ by applying thermodynamic and multifractal formalisms. We remark that the dynamical method is applicable in all the aforementioned works due to the special types of frequencies considered. We remark that, for all works mentioned  above, the dynamical method is applicable because of the special types of the frequencies. Very recently, Luna \cite{Lu2024} showed that $\lim_{\lambda\to0}D(\alpha,\lambda)=1$ for $\alpha$ with bounded type.
	
	We now proceed to discuss the spectral properties of the general Sturmian Hamiltonian. Bellissard et al. \cite{BIST1989} showed that $\Sigma_{\alpha,\lambda}$ is a Cantor set of Lebesgue measure zero.
	This motivates the study on the fractal dimensions of the spectrum. 
	Building on \cite{BIST1989}, Raymond \cite{Ra1997} showed that for $\lambda > 4$, the spectrum $ \Sigma_{\alpha,\lambda}$ admits a natural covering structure. Leveraging this structure, he established that all gaps of the spectrum predicted by gap labelling theory are open; see also \cite{BBBRT2024} for a new formulation of the results in \cite{Ra1997}. In a very recent work \cite{BBL2024}, Band, Beckus, and Loewy extended the aforementioned result to all $\lambda$ and resolved the ``dry ten Martini problem" for the Sturmian Hamiltonian.
	
	Fix $\alpha\in\I$  with continued fraction expansion $[a_1,a_2,\cdots]$. By using the subordinacy theory, Damanik, Killip and Lenz \cite{DKL2000} showed that, if   $\limsup_{k\rightarrow\infty}\frac{1}{k}\sum_{i=1}^k
	a_i<\infty$, then  $D(\alpha,\lambda)>0$.  
	Liu and Wen \cite{LW2004} refined Raymond's covering structure \cite{Ra1997}, enabling the computation of fractal dimensions. This development has spurred extensive investigations into the fractal dimensions of the spectrum of Sturmian Hamiltonians, as documented in \cite{CQ2023,DG2015,FLW2011,LQW2014,LW2004}. Define
	\begin{equation*}\label{K-ast}
		K_\ast(\alpha):=\liminf_{k\rightarrow\infty}
		\Big(\prod_{i=1}^k a_i\Big)^{1/k}\ \ \ \text{ and }\ \ \  K^\ast(\alpha):=
		\limsup_{k\rightarrow\infty}\Big(\prod_{i=1}^k a_i\Big)^{1/k}.
	\end{equation*}
	Fix $\lambda\geq24$. Then it is proven in \cite{LQW2014,LW2004} that
	$$
	\begin{cases}
		\dim_H \Sigma_{\alpha,\lambda} \in(0,1) & \text{ if   }\ \  K_\ast(\alpha)< \infty\\
		\dim_H \Sigma_{\alpha,\lambda} =1 & \text{ if   }\ \  K_\ast(\alpha)= \infty
	\end{cases}
	\ \text{ and }\
	\begin{cases}
		\overline{\dim}_B \Sigma_{\alpha,\lambda} \in(0,1) & \text{ if   }\ \  K^\ast(\alpha)< \infty\\
		\overline{\dim}_B \Sigma_{\alpha,\lambda} =1 & \text{ if   }\ \  K^\ast(\alpha)= \infty
	\end{cases}.
	$$
	Later, Damanik and Gorodetski \cite{DG2015} proved that for Lebesgue a.e. $\alpha\in\I$, both $\dim_H\Sigma_{\alpha,\lambda}$ and
	$\overline{\dim}_B\Sigma_{\alpha,\lambda}$ remain constant. In a very recent work \cite{CQ2023}, Cao and Qu found a set
	of full Lebesgue measure $\tilde{\I}\subset\I$ (independent of $\lambda$), such that for each $(\alpha,\lambda)\in\tilde{\I}\times[24,\infty)$,
	\begin{equation*}
		\dim_{H}\Sigma_{\alpha,\lambda}=\overline{\dim}_B\Sigma_{\alpha,\lambda}.
	\end{equation*}
	
	The properties of $\mathcal{N}_{\alpha,\lambda}$ are less studied for general frequencies. In \cite{Qu2018}, for any $\lambda>20$ and $\alpha$ with bounded continued fraction expansion, Qu constructed certain $\alpha$ such that $\mathcal{N}_{\alpha,\lambda}$ is not exact-dimensional. Jitomirskaya and Zhang \cite{JZ2022} also constructed Liouvillian frequency $\alpha$ such that for any $\lambda>0$, the related $\mathcal{N}_{\alpha,\lambda}$ is also not exact-dimensional. Recently, Cao and Qu \cite{CQ2023} showed that for any $(\alpha,\lambda)\in\tilde{\I}\times[24,\infty)$, the DOS $\mathcal{N}_{\alpha,\lambda}$ is exact-dimensional and its Hausdorff dimension satisfies Young's formula.
	\subsection{Main results}\

	In this paper, we establish the following principal result: For all $\lambda>20$, the spectral characteristics of the Sturm Hamiltonian with eventually periodic type frequencies exhibit striking parallels with those of the Fibonacci Hamiltonian (see Theorem A).
	\begin{theorem}\label{main-result}
		Let $\lambda>20$. For any $k\in\N$ and $\ma\in\N^k$, there exists a $C^1$ function $\p_{\ma}:\R\to\R$ such that for all $\alpha\in\mathcal{EP}(\ma)$, the following hold:
		\begin{enumerate}[label=(\roman*)]
			\item The DOS $\mathcal{N}_{\alpha,\lambda}$ is exact-dimensional and $d(\alpha,\lambda)$ satisfies
			\begin{equation*}\label{1}
				d(\alpha,\lambda)=
				-\frac{\p_{\ma}(0)}{\p'_{\ma}(0)}.
			\end{equation*}
			\item The optimal H\"older exponent of $\mathcal{N}_{\alpha,\lambda}$ satisfies 
			\begin{equation}\label{2} \gamma(\alpha,\lambda)=\inf\{\underline{d}_{\mathcal{N}_{\alpha,\lambda}}(x):x\in\Sigma_{\alpha,\lambda}\}=-\frac{\p_{\ma}(0)}{\lim\limits_{s\to-\infty}\p'_{\ma}(s)}.
			\end{equation}
			\item
			The spectrum $\Sigma_{\alpha,\lambda}$ satisfies
			$
			\dim_H \Sigma_{\alpha,\lambda}=\overline{\dim}_B \Sigma_{\alpha,\lambda}
			$, and there exists $0<s_{\ma}<D(\alpha,\lambda)$ such that
			\begin{equation}\label{3}
				D(\alpha,\lambda)
				=-\frac{\p_{\ma}(0)}{\p'_{\ma}(s_{\ma})}.
			\end{equation}
			\item $\mathcal{T}^-(\alpha,\lambda)$ and $\mathcal{T}^+(\alpha,\lambda)$ are equal and independent of $\theta\in\mathbb{T}$. Moreover, 
			\begin{equation}\label{4}
				\mathcal{T}^{\pm}(\alpha,\lambda)=\sup\{\overline{d}_{\mathcal{N}_{\alpha,\lambda}}(x):x\in\Sigma_{\alpha,\lambda}\}=-\frac{\p_{\ma}(0)}{\lim\limits_{s\to\infty}\p'_{\ma}(s)}.
			\end{equation}
			\item Furthermore, if the coupling constant $\lambda\geq240$,  the following inequalities hold:
			\begin{equation}\label{5}
				\gamma(\alpha,\lambda)<d(\alpha,\lambda)<D(\alpha,\lambda)<\mathcal{T}^{\pm}(\alpha,\lambda).
			\end{equation}
			\item There exist four constants $0<\rho_{\gamma}(\ma)\leq\rho_{d}(\ma)\leq\rho_{D}(\ma)\leq\rho_{\mathcal{T}}(\ma)$ such that 
			\begin{align}
				&\lim_{\lambda\to\infty}\gamma(\alpha,\lambda)\cdot\log\lambda=\rho_{\gamma}(\ma),\ \ \ \ \ \ \lim_{\lambda\to\infty}d(\alpha,\lambda)\cdot\log\lambda=\rho_{d}(\ma),\label{6}\\
				&\lim_{\lambda\to\infty}D(\alpha,\lambda)\cdot\log\lambda=\rho_{D}(\ma),\ \ \ \  \  \lim_{\lambda\to\infty}\mathcal{T}^{\pm}(\alpha,\lambda)\cdot\log\lambda=\rho_{\mathcal{T}}(\ma).\label{7}
			\end{align}
			
			\item The following ``tail properties" hold: that is for any $\alpha,\beta\in\mathcal{EP}(\ma)$, we have
			\begin{align}\label{8}
				\gamma(\alpha,\lambda)=\gamma(\beta,\lambda),\   d(\alpha,\lambda)=d(\beta,\lambda), \ 
				D(\alpha,\lambda)=D(\beta,\lambda),\  \mathcal{T}^{\pm}(\alpha,\lambda)=\mathcal{T}^{\pm}(\beta,\lambda). 
			\end{align}
		\end{enumerate}
	\end{theorem}
	\begin{remark}
		\begin{enumerate}
			{\rm 
				\item The $C^1$-function $\p_{\ma}(s)$ mentioned above is the topological pressure (see \eqref{def-pressure}) defined by the potential function $s\Psi^{\ma}$ (see \eqref{def-psi}). Additionally, we can establish that the limits in \eqref{2} and \eqref{4} both exist (see Proposition \ref{trans-band}).
				
				\item The derivative of the pressure function $\p_{\ma}$ exhibits profound connections to the interrelations between $\gamma(\alpha,\lambda), d(\alpha,\lambda), D(\alpha,\lambda) $ and $\mathcal{T}^{\pm}(\alpha,\lambda)$. 
				These observations are illustrated in Figure \ref{type-evo}. 
				
				\item (i) and (ii) are new, but (iii) is known (see \cite[Theorem 1.2]{FLW2011}). Herein, we derive the explicit expression and asymptotic formula for the dimension of $\Sigma_{\alpha,\lambda}$. 
				\item 	Damanik et. al. \cite{DGLQ2015} obtained the lower and upper bound for all time-averaged transport exponents for any $\alpha\in\I$, and the part (iv) was presented in \cite[Proposition 4.8(c) and p. 1433]{DGLQ2015}, yet no formal proof is provided therein. The result is indeed anticipated, its demonstration is non-trivial and requires rigorous justification. We give a completed proof and obtain \eqref{4}, which shows that $\mathcal{T}^{\pm}(\alpha,\lambda)$ is an upper bound for the level set of the upper local dimension of the DOS $\mathcal{N}_{\alpha,\lambda}$.
				\item The inequality $d(\alpha,\lambda)<D(\alpha,\lambda)$ establishes a conjecture of Barry Simon, which was made based on an analogy with work of Makarov and Volberg \cite{Ma1998,Vo1993}. It is a widely held belief that if a planar set $E$ is dynamically defined, then the Hausdorff dimension of the harmonic measure determined by $E$ is strictly less than the Hausdorff dimension of $E$ (see the monograph \cite{GM2005}). Our results confirm this conjecture in the present special case.
				
				\item In \cite{Qu2016}, Qu used the asymptotic formulas to obtain \eqref{5}, which inherently precluded him from obtaining conclusions for the $\ma=2$ (in this case, $\rho_{\gamma}(\ma)=\rho_d(\ma)=\rho_D(\ma)$), and a necessitated sufficiently large coupling constant for the inequalities to hold. By contrast, we demonstrate that the inequalities are satisfied if $\lambda\geq240$, and we elucidate previously unclarified results for $\ma=2$.
				
				\item The equalities of \eqref{6} are new. For constant type $\alpha_{\kappa}:=[\kappa,\kappa,\cdots]$, the first equality of \eqref{6} was proved in \cite{Mu2019}; for eventually constant type $\alpha$, the equations of \eqref{6} were contained in \cite[Theorem 1 (iv)]{Qu2016}. In this paper, We generalize their conclusions and adopt a new proof approach.
				
				\item For any irrational $\alpha$, the first equality of \eqref{7}, see \cite{FLW2011,LQW2014}; for constant type $\alpha_{\kappa}$, the second equality of \eqref{7}, see \cite[Theorem 1.4]{DGLQ2015}.  We state them here for comparison. Our method gives a new proof for \eqref{7}.
				
				\item  \eqref{8} shows that the four quantities $\gamma(\alpha,\lambda), d(\alpha,\lambda), D(\alpha,\lambda) $ and $  \mathcal{T}^{\pm}(\alpha,\lambda)$
				only depend on the “tail” of the expansion of $\alpha$ when $\alpha$ is of eventually periodic type.
			}
		\end{enumerate}
	\end{remark}
	\begin{figure}
		\includegraphics[scale=1]{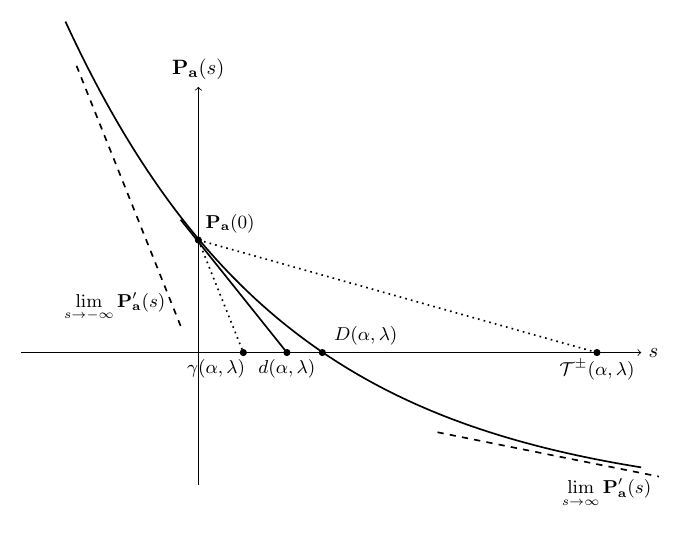}
		\caption{\footnotesize{The pressure function $\p_{\ma}(s)$ is strictly convex (if $\lambda\geq240$) and differentiable on $\mathbb{R}$, admitting a tangent line at each such point. Taking $s\rightarrow-\infty$, $s=0$ and $s\rightarrow\infty$ respectively gives the corresponding tangent lines; translating each tangent line horizontally to pass through the point $(0,\p_{\ma}(0))$, the intersection points of these translated tangent lines with the horizontal axis are exactly the values of the spectral characteristics $\gamma(\alpha,\lambda),d(\alpha,\lambda),\mathcal{T}^{\pm}(\alpha,\lambda) $.}}\label{type-evo}
	\end{figure}
	
	\subsection{Ideas of the proof}\
	
	Let us explain the idea of our proof. By constructing a bi-Lipschitz homeomorphism between the spectrum and the symbolic space (see Proposition \ref{bi-lip-alpha}), we can regard the spectrum as a kind of subshift of finite type, then we successfully transfer the spectral problem to a dynamical problem. We combine the tools from the thermodynamical formalism and multifractal analysis, and derive the desired result. We mainly illustrate how the pressure function relates to spectral characteristics.
	
	Recall that for each $\alpha\in\I$ and $\lambda>4$, Raymond  \cite{Ra1997} constructed a decreasing family $\{\B_n^\alpha:n\ge0\}$ of covers for $\Sigma_{\alpha,\lambda}$ (see Sect. \ref{sec-spectrum-1}).  
	Later in \cite{LQW2014,Qu2016,Qu2018},  a coding symbolic space $\Omega^\alpha$ for $\Sigma_{\alpha,\lambda}$  was given based on Raymond's construction (see Sect. \ref{sec-spectrum-2}). In this way, one can define a natural coding map 
	$\pi^\alpha_\lambda: \Omega^\alpha\to \Sigma_{\alpha,\lambda}$ as
	\begin{equation}
		\pi^{\alpha}_\lambda(x):=\bigcap_{n\ge 0} B_{x|_n}^{\alpha},\quad\text{where\ }B_{x|_n}^{\alpha}\in \B_n^\alpha\text{ is the spectral band}.
	\end{equation}
Through structural analysis of symbolic space $\Omega^{\alpha}$, we reveal a fundamental deficiency: the absence of an amenable dynamical system framework (that is $\Omega^{\alpha}$ is not invariant under the shift map, see Remark \ref{invarint}). Inspired by the works \cite{CQ2023,Qu2016}, for any periodic frequency $\alpha=[\overline{a_1,a_2,\cdots,a_k}]$, we can construct the new symbolic space $\Omega_{\ma}$ as the subshift of finite type. A simple but crucial observation is that there exists a natural bijection $\iota:\Omega_{\ma}\to\Omega^{\check{\alpha}}$, where $\check{\alpha}:=[1,\overline{a_1,a_2,\cdots,a_k}]$.  This means that $\Omega_{\ma}$ can code the spectrum $\Sigma_{\check{\alpha},\lambda}$ by coding map $\pi_{\ma,\lambda}=\pi^{\check{\alpha}}_{\lambda}\circ\iota$ (see Sect. \ref{sec-Sturm-1}). 
	
	Our main aim is to study the spectral characteristics. Thus, we introduce an appropriate metric $d_{\ma}$ on $\Omega_{\ma}$ such that the coding map
	\begin{equation}\label{bi-lip-hom}
		\pi_{\ma,\lambda}:(\Omega_{\ma},d_{\ma})\to(\Sigma_{\check{\alpha},\lambda},|\cdot|)
	\end{equation}
	is a bi-Lipschitz homeomorphism (see Proposition \ref{bi-Lip}).
	For any $\mv$, $\mathbf{w}\in \Omega_{\ma}$, define
	\begin{equation*}
		d_{\ma}(\mv,\mathbf{w}):=|B^{\check{\alpha}}_{w}|,\ \ \ \text{if}\ w=\iota(\mv)\wedge\iota(\mathbf{w}),
	\end{equation*}
	where $\iota(\mv)\wedge\iota(\mathbf{w})$ denotes the common prefix of $\iota(\mv)$ and $\iota(\mathbf{w})$, and $|B^{\check{\alpha}}_w|$ is the length of the spectral band $B^{\check{\alpha}}_w$. 
	With this coding map $\pi_{\ma,\lambda}$, the DOS $\mathcal{N}_{\check{\alpha},\lambda}$ (supported on $\Sigma_{\check{\alpha},\lambda}$) is related to the maximal entropy measure $\mu_{\ma}$ on $\Omega_{\ma}$ (see Proposition \ref{mu-N}), 
	then there exists a constant $C>1$ such that
	\begin{equation}\label{id-1}
		C^{-1}\mathcal{N}_{\check{\alpha},\lambda}\leq 	\mu_{\ma}\circ\pi^{-1}_{\ma,\lambda}\leq C\mathcal{N}_{\check{\alpha},\lambda}.
	\end{equation}
	Roughly speaking, the properties (exact-dimensional property or optimal H\"older exponent) of $\mu_{\ma}$ and $\mathcal{N}_{\check{\alpha},\lambda}$ are the same.
	
	Now we elaborate on the application of thermodynamic formalism within space $\Omega_{\ma}$. Note that the metric $d_{\ma}$ is related to the geometric potential $\Psi^{\ma}=\{\psi^{\ma}_n:n\geq1\}$, where $\psi^{\ma}_n(\mv)=\log|B^{\check{\alpha}}_{\iota(\mv)|_{nk+1}}|$. This kind of metric is also called a weak-Gibbs metric on $\Omega_{\ma}$, which has been used in \cite{Qu2016,Qu2018}. It is natural to define the pressure function
	\begin{equation*}
		\p_{\ma}(s):=\lim\limits_{n\to\infty}\frac{1}{n}\log\sum_{|\mv|=n}\exp(\sup_{x\in[\mv]}s\psi^{\ma}_{n}(x))=\lim\limits_{n\to\infty}\frac{1}{n}\log\sum_{B\in\B^{\check{\alpha}}_{nk+1}}|B|^s.
	\end{equation*}
	Next, we introduce how the derivative of the pressure function $\p_{\ma}$ exhibits profound connections to the interrelations between $\gamma(\alpha,\lambda), d(\alpha,\lambda), D(\alpha,\lambda) $ and $\mathcal{T}^{\pm}(\alpha,\lambda)$. So we prove the crucial result (see Proposition \ref{trans-band}): $\p_{\ma}$ is $C^1$ on $\R$ and the following limits exist
	\begin{equation}\label{id-2}
		\begin{cases}
			\lim\limits_{s\to-\infty}\p'_{\ma}(s)=\lim\limits_{n\to\infty}\frac{1}{n}\inf_{\mv\in\Omega_{\ma}}\psi_{n}^{\ma}(\mv)
=\lim\limits_{n\to\infty}\frac{1}{n}\inf_{B\in\B^{\check{\alpha}}_{nk+1}}\log|B|,\\
			\p'_{\ma}(s)=\lim\limits_{n\to\infty}\int\psi^{\ma}_n\mathrm{d}\mu^{\ma}_{s}
=\lim\limits_{n\to\infty}\frac{1}{n}\log|B(s)|,\quad\text{for some $B(s)\in\B^{\check{\alpha}}_{nk+1}$},\\
			\lim\limits_{s\to\infty}\p'_{\ma}(s)=\lim\limits_{n\to\infty}\frac{1}{n}
\sup_{\mv\in\Omega_{\ma}}\psi_{n}^{\ma}(\mv)=\lim\limits_{n\to\infty}\frac{1}{n}\sup_{B\in\B^{\check{\alpha}}_{nk+1}}\log|B|,
		\end{cases}
	\end{equation}
	where $\mu^{\ma}_{s}$ is the Gibbs measure related to $s\Psi^{\ma}$. By the classical Bowen's formula and \eqref{bi-lip-hom}, we have that $\dim_{H}\Omega_{\ma}=\dim_{H}\Sigma_{\check{\alpha},\lambda}=D(\check{\alpha},\lambda)$ and $D(\check{\alpha},\lambda)=s$ is the zero of $\p_{\ma}(s)=0$. Therefore, there exists a constant $s_{\ma}>0$ such that $$
	D(\check{\alpha},\lambda)=-\frac{\p_{\ma}(0)}{\p'_{\ma}(s_{\ma})}.$$
	By \cite[Proposition 4.8(c) and Proposition 5.1]{DGLQ2015}(it connects the  longest spectral band and the transport exponents), we conclude that
	\begin{equation*}
		\mathcal{T}^{\pm}(\check{\alpha},\lambda)=\lim\limits_{n\to\infty}\frac{n\p_{\ma}(0)}{\sup_{B\in\B^{\check{\alpha}}_{nk+1}}\log|B|}.
	\end{equation*}
	
	Note that the maximal entropy measure $\mu_{\ma}$ on $\Omega_{\ma}$ is a Gibbs measure for potential $\Phi=\{-n\p_{\ma}(0),n\geq1\}$. Moreover, we have that $\mu^{\ma}_0=\mu_{\ma}$. To study the spectrum $\Sigma_{\check{\alpha},\lambda}$ and the DOS $\mathcal{N}_{\check{\alpha},\lambda}$, we only need to analyze two almost additive potentials $\Phi$ and $\Psi^{\ma}$ on $\Omega_{\ma}$. Then we can use some results for the multifractal analysis of quotients of almost additive potentials on subshifts of finite type (see for example \cite{BQ2012,FH2010}). By \eqref{id-1} and \eqref{def-loc-dim}, one can compute the local dimension of $\mathcal{N}_{\check{\alpha},\lambda}$. For  any $x\in\Sigma_{\check{\alpha},\lambda}$, $\mv=\pi_{\ma,\lambda}^{-1}(x)\in\Omega_{\ma}$ and
	\begin{equation}\label{id-3}
		\left\{ \begin{aligned}
			\underline{d}_{\mathcal{N}_{\check{\alpha},\lambda}}(x)&=	\underline{d}_{\mu_{\ma}}(\mv)=\frac{-\p_{\ma}(0)}{\liminf\limits_{n\to\infty}\psi^{\ma}_n(\mv)/n};\\
			\overline{d}_{\mathcal{N}_{\check{\alpha},\lambda}}(x)&=\overline{d}_{\mu_{\ma}}(\mv)=\frac{-\p_{\ma}(0)}{\limsup\limits_{n\to\infty}\psi^{\ma}_n(\mv)/n}.
		\end{aligned}
		\right. 
	\end{equation}
	According to \eqref{id-2} and \eqref{id-3}, we obtain that 
	\begin{equation*}
		\gamma(\check{\alpha},\lambda)=\frac{-\p_{\ma}(0)}{\lim\limits_{s\to-\infty}\p'_{\ma}(s)},\ d(\check{\alpha},\lambda)=\frac{-\p_{\ma}(0)}{\p'(0)},\
		D(\check{\alpha},\lambda)=\frac{-\p_{\ma}(0)}{\p'_{\ma}(s_{\ma})},\  \mathcal{T}^{\pm}(\check{\alpha},\lambda)=\frac{-\p_{\ma}(0)}{\lim\limits_{s\to\infty}\p'_{\ma}(s)}.
	\end{equation*}
	From the ``tail property" of the spectral quantities (see Proposition \ref{geo-lem-tail}), we can relate the spectral characteristics $\gamma(\check{\alpha},\lambda)$ ($d(\check{\alpha},\lambda), D(\check{\alpha},\lambda)$ or $ \mathcal{T}^{\pm}(\check{\alpha},\lambda))$ to that of $\gamma(\alpha,\lambda)$ ($d(\alpha,\lambda), D(\alpha,\lambda)$ or $\mathcal{T}^{\pm}(\alpha,\lambda))$ and conclude the proof.
	
	Finally, we briefly illustrate how to derive the asymptotic behaviors and strict inequalities of spectral characteristics. According to the preceding analysis (see equations \eqref{id-2}), it suffices to verify whether spectral bands can be expressed as certain exponential powers of $\lambda$. For instance, if
	$$
	\lim_{n\to\infty}\frac{1}{n}\inf_{B\in\mathcal{B}^{\check{\alpha}}_{nk+1}}\log|B|\to\frac{\kappa_1}{\log\lambda},
	\quad
	\lim_{n\to\infty}\frac{1}{n}\sup_{B\in\mathcal{B}^{\check{\alpha}}_{nk+1}}\log|B|\to\frac{\kappa_2}{\log\lambda}
	\quad\text{as}\ \lambda\to\infty,
	$$
	then $\kappa_1$ is associated with the H\"older exponent, while $\kappa_2$ corresponds to the transport exponent. Consequently, this result follows directly from Lemma \ref{esti-band-length}. Similarly, to establish the strict inequalities for these spectral characteristics, it suffices to verify that the exponents of $\lambda$ corresponding to distinct spectral bands are not identical.
	
	\medskip
	The rest of the paper is organized as follows. In Sect. \ref{sec-TF}, we recall some fundamental results in thermodynamical formalism. In Sect. \ref{sec-spectrum}, we discuss the structure of the spectrum and its coding. In Sect. \ref{sec-Sturm}, we connect the Sturm Hamiltonian with thermodynamical formalism, and apply the relevant results developed in Sect. \ref{sec-TF} to the spectral analysis of the Sturm Hamiltonian. In Sect. \ref{sec-main}, we prove the Theorem \ref{main-result}. We give the proof of the technical lemma in Sect. \ref{imp-lemma}. In Appendix \ref{exp}, we briefly give a refined estimate of the spectral band for larger coupling constant ($\lambda\geq240$).
	
	{\bf Notations.} In this paper, we use $\lhd$ and $\rhd$ to indicate the beginning and end of the proof of a claim. For two positive sequences $\{a_n:n\in \N\}$ and $\{b_n:n\in \N\}$, the notation $
	a_n\sim b_n$ means that there exists a constant $C>1$ such that $C^{-1}b_n\leq a_n\le C b_n$ for all $n$. Assume $X$ is a metric space and $\mu,\nu$ are two finite Borel measures on $X$. We write $
	\mu\asymp \nu$, if there exists a constant $C>1$ such that $C^{-1}\nu(B)\le\mu(B)\le C\nu(B)$ for any Borel set $B\subset X$. 
	\section{Relevant facts about thermodynamical formalism and multifractal analysis}\label{sec-TF}
	
	Let us begin with some notations and background. We say that $(X,T)$ is a
	topological dynamical system (TDS) if $X$ is a compact metric space and $T:X\to X$ is a
	continuous map. Denote the set of all $T$-invariant probability measures supported on $X$ by $\mathcal{M}(X,T)$. 
	
	Assume $(X,T)$ is a TDS and $\Phi=\{\phi_n: n\geq1\}$ is a family of continuous
	functions from $X$ to $\R$. We call $\Phi$ a potential on $X$. If there exists a constant $%
	C(\Phi)\geq0$ such that 
	\begin{equation*}\label{almost-additive}
		|\phi_{n+m}(x)-\phi_{n}(x)-\phi_{m}(T^nx)|\leq C(\Phi),\ \ \ \forall\ n,m\in\N,
	\end{equation*}
	then we say  that $\Phi$ is {\it almost\ additive} and write $\Phi\in C_{aa}(X,T)$, where $C_{aa}(X,T)$ denotes the set of all the almost additive potentials defined on $X$. 
	
	Given $\Phi\in C_{aa}(X,T)$, if there exists a constant $c>0$ such that $%
	\phi_{n}(x)\leq -cn$ for any $n\geq1$ and $\phi_n$ is decreasing for $n$, then we say that $\Phi$
	is negative and write $\Phi\in C_{aa}^-(X,T)$. In this case, there exist two constants $0<c_1\leq c_2$ such that
	\begin{equation}\label{low-psi-upp}
		-c_2n\leq \phi_n(x)\leq-c_1n,\ \ \ \forall\ n\in\N.
	\end{equation}
	
	Given $\Phi\in C_{aa}(X,T)$, by subadditivity, the limit $\Phi_*(\mu):=\lim\limits_{n\to\infty}\frac{1}{n}\int_X\phi_nd\mu$ exists for every $\mu\in\mathcal{M}(X,T)$. Note that if $\Phi$ is negative, then $\Phi_*(\mu)<0$.
	
	Let $(\Sigma_A,T)$ be a topologically mixing subshift of finite type.
	We say $\Phi\in C_{aa}(\Sigma_A,T)$ has bounded variation property, if there
	exists a constant $D(\Phi)\geq0$ such that 
	\begin{equation*}\label{bd-variation}
		\sup\left\{|\phi_n(x)-\phi_n(y)|: x|_n=y|_n,x,y\in\Sigma_{A}\right\}\leq D(\Phi), \ \ \ \forall\ n\in\N, 
	\end{equation*}
	where $x|_n:=x_{1}x_{2}\cdots x_{n}$ is the $n$-th prefix of $x$. Define
	\begin{equation*}
		\begin{cases}
			\mathcal{F}(\Sigma_A,T):=\{\Phi\in C_{aa}(\Sigma_A,T):\Phi\ \text{has bounded variation}\},\\
			\mathcal{F}^-(\Sigma_A,T):=\{\Phi\in \mathcal{F}(\Sigma_A,T):\Phi\ \text{is negative}\}.
		\end{cases}
	\end{equation*}
	
	Given $\Psi=\{\psi_{n}:n\geq1\}\in\mathcal{F}^-(\Sigma_A,T)$, define 
	a weak-Gibbs metric $d_{\Psi}$ on $(\Sigma_A,T)$ as 
	\begin{equation}\label{weak-gibbs-metric}
		d_{\Psi}(x,y):=\sup_{z\in[x\wedge y]}\exp(\psi_{|x\wedge y|}(z))	\footnote{We adopt the conventions $|\emptyset|=0$ and $\psi_0\equiv0$.},
	\end{equation}
	where $x\wedge y$ denotes the maximal common prefix of $x$ and $y$, $|w|$ denotes the length of the word $w$, and $[x|_n]:=\{y\in\Sigma_{A}:y|_n=x|_n\}$ denotes the cylinder set. This kind of metric is considered in \cite{BQ2012,GP1997,KS2004,Qu2016}. In this section, we always endow $(\Sigma_{A},T)$ with the metric $d_{\Psi}$ without further mention.
	\subsection{Thermodynamical formalism}
	
	\begin{theorem}[\cite{Ba2006,Mu2006}]\label{variation priciple}
		Assume $\Phi\in\mathcal{F}(\Sigma_A,T)$, then the following results hold:
		
		(1) The limit
		\begin{equation*}\label{def-pre}
			P_{top}(\Phi):=\lim\limits_{n\to\infty}\frac{1}{n}\log\sum_{|w|=n}\exp(\sup_{x\in[w]}\phi_n(x))
		\end{equation*}
		exists, and is called the topological pressure of $\Phi$
		. Moreover, the variational principle holds:
		\begin{equation}\label{var-pri}
			P_{top}(\Phi)=\sup\{h_{T}(\mu)+\Phi_*(\mu):\mu\in\mathcal{M}(\Sigma_A,T)\}.
		\end{equation}
		
		(2) There exists an ergodic measure $\mu_{\Phi}\in\mathcal{M}(\Sigma_A,T)$ such that
		\begin{equation*}
			C^{-1}\leq\frac{\mu_{\Phi}([x|_n])}{\exp(-nP_{top}(\Phi)+\phi_n(x))}\leq C,\quad\forall\ x\in\Sigma_A,n\in\N.
		\end{equation*}
		$\mu_{\Phi}$ is called the Gibbs measure related to $\Phi$. Moreover, $\mu_{\Phi}$ is the unique invariant measure which attains the supremum of \eqref{var-pri}.
	\end{theorem}
	If $\mu\in\mathcal{M}(\Sigma_A,T)$ satisfies $P_{top}(\Phi)=h_{T}(\mu)+\Phi_*(\mu)$, then $\mu$ is called an equilibrium state of $\Phi$. Furthermore, this equilibrium state is unique. 
	
	The following corollary is a standard result directly verifiable through the variational principle; its proof is consequently omitted (see for example \cite{Ba1996}).
	\begin{corollary}\label{pressure-continuous}
		Fix $\Psi\in\mathcal{F}^-(\Sigma_{A},T)$. The following results hold:
		
		(1) For any $\Phi\in\mathcal{F}(\Sigma_{A},T)$, the function $s\mapsto P_{top}(\Phi+s\Psi)$ is convex, strictly decreasing on $\R$ and
		\begin{equation*}
			\lim\limits_{s\to-\infty}P_{top}(\Phi+s\Psi)=\infty;\ \ \  \ \ \ \lim\limits_{s\to\infty}P_{top}(\Phi+s\Psi)=-\infty.
		\end{equation*}
		Thus $P_{top}(\Phi+s\Psi)=0$ has a unique solution.
		
		(2) The topological entropy $h_{top}(T)=P_{top}(0\cdot\Psi)$ and the potential $\mathbf{0}=0\cdot\Psi$ admits a Gibbs measure $\mu_{\max}$ with $$\mu_{\max}([\mw|_n])\sim\exp(-nP_{top}(\mathbf{0})),\quad\forall\ x\in\Sigma_{A},n\in\N.$$
		Moreover, $h_{T}(\mu_{\max})=\sup\{h_{T}(\mu):\mu\in\mathcal{M}(\Sigma_A,T)\}=h_{top}(T)$, and $\mu_{\max}$ is called the maximal entropy measure.
		
		(3) Bowen's formula holds: $\dim_H\Sigma_A=s_{\Psi}$, where $s_{\Psi}$ is the zero of $P_{top}(s\Psi)=0$.
	\end{corollary}
	\subsection{Analysis of the pressure function}\
	
	In this subsection, we will show that the pressure function $P_{top}(s\Psi)$ is $C^1$ on $\R$. This property plays a crucial role in the subsequent analysis of this paper.
	
	\begin{proposition}\label{relatived-pressure-diff}
		Assume $\Psi=\{\psi_{n}:n\geq1\}\in\mathcal{F}^-(\Sigma_{A},T)$. Then the following hold:
		
		(1) For any $s\in\R$, $s\Psi\in\mathcal{F}^-(\Sigma_{A},T)$ admits a Gibbs measure $\mu_s$. 
		
		(2) $P(s):=P_{top}(s\Psi)$ is $C^1$ and convex on $\R$, and there exist $0<c_1\leq c_2$ such that 
		\begin{equation}\label{diff-Ly}
			-c_2\leq P^{\prime }(s)=\Psi_*(\mu_s)\leq-c_1,\ \forall\ s\in%
			{\mathbb{R}}.
		\end{equation}
		
		(3) Furthermore, if the potential $\Psi$ satisfies
		\begin{equation}\label{contract}
			\sup_{n\geq1}\sup\{|\psi_{n}(x)-\psi_{n}(y)|:x,y\in\Sigma_{A}\}=\infty,
		\end{equation}
		then $\mu_{s}\neq\mu_{t}$ for any $s\neq t$ and $P$ is strictly convex on $\R$.
	\end{proposition}
	\begin{proof}
		(1) It is seen that $s\Psi\in\mathcal{F}^-(\Sigma_{A},T)$ with $C(s\Psi)=|s|C(\Psi), D(s\Psi)=|s|D(\Psi)$ for any $s\in\R$. Then by Theorem \ref{variation priciple} (2), the potential $s\Psi$ admits a Gibbs measure $\mu_s$. 
		
		(2) We follow partially arguments in \cite[Theorem 2]{BD2009}. Since $\Psi\in\mathcal{F}^-(\Sigma_A,T)$ and by \eqref{low-psi-upp}, there exist two constants $0< c_1\leq c_2$ such that 
		\begin{equation}\label{Psi-bd}
			-c_2\leq\Psi_*(\mu)=\lim\limits_{n\to\infty}\frac{1}{n}\int\psi_n
 \mathrm{d}\mu\leq-c_1,\ \ \ \forall\  \mu\in\mathcal{M}(\Sigma_A,T).
		\end{equation}
		
		Using variational principle (see \eqref{var-pri}) and Theorem \ref{variation priciple} (2), we have 
		\begin{align*}
			\begin{cases}
				P(s)-P(t)\geq h_T(\mu_t)+(s\Psi)_*(\mu_t)-(h_{T}(\mu_t)+\left(t\Psi)_*(\mu_t)\right)=(s-t)\Psi_*(\mu_t),\\
				P(s)-P(t)\leq h_T(\mu_s)+(s\Psi)_*(\mu_s)-(h_{T}(\mu_s)+\left(t\Psi)_*(\mu_s)\right)=(s-t)\Psi_*(\mu_s).
			\end{cases}
		\end{align*}
		This yields the following inequalities 
		\begin{equation}\label{relatived-diff}
			\left\{\begin{aligned}
				\Psi_*(\mu_s)&\leq\frac{P(s)-P(t)}{s-t}\leq\Psi_*(\mu_t),\ \ \text{if }
				\  t>s,\\ 
				\Psi_*(\mu_s)&\geq\frac{P(s)-P(t)}{s-t}\geq\Psi_*(\mu_t),\ \
				\text{if } \  t<s. \end{aligned}\right.
		\end{equation}
		
		\noindent\textbf{Claim:} If $\mu_{t_n}\to\mu$ for
		some sequence $t_n\to s$, then $\mu\equiv\mu_s$. 
		
		\noindent $\lhd $	Note that the entropy function $\mu\mapsto h_{T}(\mu)$ is upper semicontinuous (see for example \cite{Wa1982}) and the function $\mu\mapsto\Psi_*(\mu)$ is continuous (see \cite[equation (36)]{Ba2006}). Then the function 
		$
		\mu\mapsto h_{T}(\mu)+s\Psi_*(\mu)
		$
		is also upper semicontinuous. If $t_n\to s$, by Theorem \ref{variation priciple}, Corollary \ref{pressure-continuous} (1) and \eqref{Psi-bd}, we have 
		\begin{align*}
			h_{T}(\mu)+s\Psi_*(\mu)&\geq\limsup_{n\to\infty}\Big(h_T(\mu_{t_n})+s
			\Psi_*(\mu_{t_n})\Big)\notag \\
			&=\limsup_{n\to\infty}\Big(P(t_n)+(s-t_n)\Psi_*(\mu_{t_n})\Big)\notag\\
			&=P(s)+\limsup_{n\to\infty}(s-t_n)\Psi_*(\mu_{t_n})=P(s).
		\end{align*}
		This combines with variational principle yields the identity 
		\begin{align*}
			h_{T}(\mu)+s\Psi_*(\mu)=P(s).
		\end{align*}
		Hence $\mu$ must be an equilibrium measure of $s\Psi$. By the uniqueness of the equilibrium measure, hence $\mu\equiv\mu_s$, this establishes the statement.
		\hfill $\rhd$
		
		Now we prove that $P$ is $C^1$ and convex on $\R$, and inequality \eqref{diff-Ly} holds. By the Claim as above, we have $\mu_{t}\to\mu_s$ when $t\to s$. By equations \eqref{Psi-bd} and \eqref{relatived-diff}, we conclude that 
		\begin{equation*}
			-c_2\leq P^{\prime }(s)=\lim_{t\to s}\frac{P(s)-P(t)}{s-t}=\Psi_*(\mu_s)\leq-c_1. \label{qpr}
		\end{equation*}
		Thus $P$ is differentiable on ${\mathbb{R}}$ and \eqref{diff-Ly} holds. By Corollary \ref{pressure-continuous} (1), $P(s)=P_{top}(s\Psi)$ is convex on ${\mathbb{R}}$, we conclude that $P$ is $C^1$ on ${\mathbb{R}}$.
		
		(3) At first, we show that $\mu_{s}\neq\mu_{t}$ for any $s\neq t$. We proceed directly by contradiction. Fix $s\in\R$. Assume that there exists a constant $t\neq s$ such that $\mu_{t}=\mu_{s}$. By the definition of Gibbs measure, for any $x\in\Sigma_{A},n\in\N$, we have
		\begin{align*}
			\frac{\exp(t\psi_{n}(x))}{\exp(nP(t))}\sim\mu_{t}([x|_n])=\mu_{s}([x|_n])\sim\frac{\exp(s\psi_{n}(x))}{\exp(nP(s))}.
		\end{align*}
		This implies that for any $x,y\in\Sigma_{A}$ and $n\in\N$,
		\begin{align*}
			\exp((t-s)\psi_{n}(x))\sim\exp(n(P(t)-P(s)))\sim\exp((t-s)\psi_{n}(y)).
		\end{align*}	
		Now we conclude that 
		\begin{equation*}
			\sup_{n\geq1}\sup\{\exp((t-s)(\psi_{n}(x)-\psi_{n}(y))):x,y\in\Sigma_{A}\}\sim1.
		\end{equation*}
		This contradicts with \eqref{contract}. So, we have that $\mu_{t}\neq\mu_{s}$ for all $t\neq s$.
		
		Now we prove that $P$ is strictly convex on ${\mathbb{R}}$. Since $\mu_{t}\neq\mu_{s}$ for any $s\neq t$, then by variational principle and Theorem \ref{variation priciple} (2), we have 
		\begin{align*}
			\begin{cases}
				P(s)-P(t)> h_T(\mu_t)+(s\Psi)_*(\mu_t)-(h_{T}(\mu_t)+\left(t\Psi)_*(\mu_t)\right)=(s-t)\Psi_*(\mu_t),\\
				P(s)-P(t)< h_T(\mu_s)+(s\Psi)_*(\mu_s)-(h_{T}(\mu_s)+\left(t\Psi)_*(\mu_s)\right)=(s-t)\Psi_*(\mu_s).
			\end{cases}
		\end{align*}
		Combining with \eqref{diff-Ly}, this yields the following inequality 
		\begin{equation*}
			P'(s)=\Psi_*(\mu_s)<\frac{P(s)-P(t)}{s-t}<\Psi_*(\mu_t)=P'(t),\quad\forall\ t>s.
		\end{equation*}
		This implies that $P^{\prime }$ is strictly increasing. Consequently, $P$ is strictly convex on ${\mathbb{R}}$. 
	\end{proof}
	By Proposition \ref{relatived-pressure-diff} (2), the following two limits exist: 
	\begin{equation*}\label{lya-infty}
		P^{\prime }(-\infty):=\lim\limits_{s\to-\infty}P^{\prime }(s),\ \ \ \ \
		P^{\prime }(\infty):=\lim\limits_{s\to+\infty}P^{\prime }(s).
	\end{equation*}
	\begin{corollary}\label{psi-pressure}
		Suppose that $\Psi\in\mathcal{F}^-(\Sigma_{A},T)$, then
		\begin{equation}\label{pre-1}
			\left\{ \begin{aligned}
				&	P^{\prime }(-\infty)=\lim_{n\to\infty}\frac{1}{n}\inf_{\mw\in\Sigma_A}\psi_{n}(\mw)=\inf_{\mw\in \Sigma_A}\liminf_{n\to\infty}\frac{1}{n}\psi_{n}(\mw),\\
				&P^{\prime }(\infty)=\lim_{n\to\infty}\frac{1}{n}\sup_{\mw\in\Sigma_A}\psi_{n}(\mw)=\sup_{\mw\in \Sigma_A}\limsup_{n\to\infty}\frac{1}{n}\psi_{n}(\mw). \end{aligned}\right.
		\end{equation}
	\end{corollary}
	
	\begin{proof}
		We only need to prove the first equation of \eqref{pre-1}, and the second one is similar.
		
		\noindent{\bf Claim:}  $\liminf\limits_{n\to\infty}\dfrac{1}{n}\inf\limits_{\mw\in\Sigma_A}\psi_{n}(\mw)=\lim\limits_{n\to\infty}\dfrac{1}{n}\inf\limits_{\mw\in\Sigma_A}\psi_{n}(\mw)=\inf\limits_{\mw\in \Sigma_A}\liminf\limits_{n\to\infty}\dfrac{1}{n}\psi_{n}(\mw).$
		
		\noindent $\lhd $
		Since $\Psi\in\mathcal{F}^-(\Sigma_A,T)$, then for any $n,m\in{\mathbb{N}}$, we have 
		\begin{align*}
			\inf_{\mw\in\Sigma_A}\psi_{n+m}(\mw)\geq\inf_{%
				\mw\in\Sigma_A}\psi_{n}(\mw)+\inf_{\mw%
				\in\Sigma_A}\psi_{m}(\mw)-C(\Psi).
		\end{align*}
		Thus $(C(\Psi)-\inf\limits_{\mw\in\Sigma_A}\psi_{n}(%
		\mw))_{n\geq1}$ form a sub-additive sequence. Then it is well
		known that the limit $	\lim\limits_{n\to\infty}\frac{1}{n}\inf\limits_{\mw\in\Sigma_A}\psi_n(\mw)$ exists, and the first equality of this Claim holds.

		Note that for any $y\in\Sigma_A$ and $n\in\N$, we have that $\inf\limits_{\mw\in\Sigma_A}\psi_n(\mw)\leq\psi_{n}(y)$, so dividing by $n$, and taking the lower limit and infimum for $y\in\Sigma_A$, we get
		\begin{equation*}
			\lim_{n\to\infty}\frac{1}{n}\inf_{\mw\in\Sigma_A}\psi_n(%
			\mw)\leq\inf_{y\in\Sigma_A}\liminf_{n\to\infty}%
			\frac{1}{n}\psi_n(y).
		\end{equation*}
		
		If the second equality of the Claim is not satisfied, the above inequality implies that there exists constant $\varepsilon_0>0$, such that 
		\begin{equation*}
			\liminf_{n\to\infty}\frac{1}{n}\Big(\psi_{n}(y)-\inf_{\mw%
				\in\Sigma_A}\psi_n(\mw)\Big)\geq\varepsilon_0\ \ \ \text{%
				for all}\ y\in\Sigma_A.
		\end{equation*}
		We conclude that there exists a constant $C>0$ such that for all $y
		\in\Sigma_A$ and large $n$, 
		\begin{equation*}
			\psi_{n}(y)-\inf_{\mw\in\Sigma_A}\psi_n(\mw
			)\geq C.
		\end{equation*}
		This clearly constitutes a contradiction.
		\hfill $\rhd$
		
		Since $\Psi$ is negative, then for any $s<0$, we have 
		\begin{equation*}
			s\inf_{\mw\in\Sigma_A}\psi_{n}(\mw)\leq\log\sum_{|w|=n}\exp(%
			\sup_{x\in[w]}s\psi_{n}(x))\leq\log\#\{w:|w|=n\}+s\inf_{\mw\in\Sigma_A}\psi_{n}(\mw).
		\end{equation*}
		Dividing by $n$ and taking the lower limit as $n\to\infty$, we conclude that 
		\begin{equation*}
			s\liminf_{n\to\infty}\frac{1}{n}\inf_{\mw\in\Sigma_A}\psi_{n}(\mw)\leq
			P(s)\leq\liminf\limits_{n\to\infty}\frac{\log\#\{w:|w|=n\}}{n}+s\liminf_{n\to\infty}\frac{1}{n}\inf_{\mw\in\Sigma_A}\psi_{n}(\mw).
		\end{equation*}
		Note that $\liminf\limits_{n\to\infty}\frac{1}{n}\log\#\{w:|w|=n\}<\# A<\infty$. Now by Proposition \ref{relatived-pressure-diff} (2), we have 
		\begin{equation*}
			\liminf_{n\to\infty}\frac{1}{n}\inf_{\mw\in\Sigma_A}\psi_{n}(\mw)=\lim\limits_{s\to-\infty}\frac{P(s)}{s}=\lim\limits_{s\to-%
				\infty}P'(s)=P'(-\infty).
		\end{equation*}
		This combines with Claim, thus showing the first equality of \eqref{pre-1}.
	\end{proof}
	\subsection{Multifractal Analysis}\
	
	Recall that $\mu_{\max}$ is the Gibbs measure related to $\mathbf{0}$, we begin with the following lemma:
	\begin{lemma}\label{mu-phi-loc}
		Fix $\Psi\in\mathcal{F}^{-}(\Sigma_{A},T)$.	For any $\mw\in\Sigma_A$, we have
		\begin{equation}\label{MA}
			\underline{d}_{\mu_{\max}}(\mw) =-\frac{P(0)}{%
				\liminf\limits_{n\to\infty}\frac{1}{n}\psi_{n}(\mw)}\ \ \ \ \text{and}\ \ \ \ \overline{d}_{\mu_{\max}}(\mw) =-\frac{P(0)}{\limsup\limits_{n\to\infty}\frac{1}{n}\psi_{n}(\mw)}.
		\end{equation}
		Consequently, the maximal entropy measure $\mu_{\max}$ is exact-dimensional and 
		\begin{equation}\label{dim-mu-phi}
			\dim_H\mu_{\max}=-\frac{P(0)}{\Psi_*(\mu_{\max})}=-\frac{h_{top}(T)}{\Psi_*(\mu_{\max})}.
		\end{equation}
	\end{lemma}
	
	\begin{proof}
		Recall that the metric is defined by \eqref{weak-gibbs-metric}. Fix $\mw\in\Sigma_A$. For any small $r>0$, let $n=n(x)$ be the unique number satisfies 
		\begin{align*}
			\exp(\psi_{n}(\mw))&<r\leq\exp(\psi_{n-1}(\mw)),
		\end{align*}
		which implies $[\mw|_{n}]\subset B(\mw,r)\subset[\mw|_{n-1}]$. Consequently, 
		\begin{equation*}
			\mu_{\max}([\mw|_{n}])\leq \mu_{\max}(B(\mw%
			,r))\leq\mu_{\max}([\mw|_{n-1}]).
		\end{equation*}
		Combining with Corollary \ref{pressure-continuous} (2) that $\mu_{\max}([\mw|_n])\sim\exp(-nP(0))$, we have
		\begin{equation*}
			\frac{-(n-1)P(0)+C_1}{\psi_{n}(\mw)}\leq\frac{\log\mu_{\max}(B(\mw,r))}{\log r}\leq\frac{-nP(0)+C_2}{%
				\psi_{n-1}(\mw)}.
		\end{equation*}
		Now by taking the
		upper and lower limits, the equations \eqref{MA} hold.
		
		By Corollary \ref{pressure-continuous} (2), $\mu_{\max}$ is a Gibbs measure, and hence it is an ergodic measure. Note that $\Psi$ is almost additive, by Kingman's ergodic theorem, for $\mu_{\max}$ a.e. $x\in\Sigma_A$, 
		\begin{equation*}
			\lim\limits_{n\to\infty}\frac{1}{n}\psi_n(x)=\lim\limits_{n\to\infty}
\frac{1}{n}\int\psi_n\mathrm{d}\mu_{\max}=\Psi_*(\mu_{\max}).
		\end{equation*}
		Together with \eqref{MA}, then $\mu_{\max}$ is exact-dimensional and \eqref{dim-mu-phi} holds.
	\end{proof}
	
	Fix $\Psi\in\mathcal{F}^{-}(\Sigma_{A},T)$.	Let $\phi_n(x)=-nP(0)$ and $\Phi=\{\phi_n:n\geq1\}$. In this case,  we see that the potential $\Phi\in\mathcal{F}(\Sigma_{A},T)$, 
	\begin{equation*}
		\underline{d}_{\mu_{\max}}(\mw)=\liminf\limits_{n\to\infty}\frac{\phi_n(x)}{%
			\psi_{n}(\mw)}\ \ \ \ \text{and}\ \ \ \ \overline{d}_{\mu_{\max}}(\mw)= \limsup\limits_{n\to\infty}\frac{\phi_n(x)}{%
			\psi_{n}(\mw)}.
	\end{equation*}
	Then we can cite some results for the multifractal analysis of quotients of almost additive potentials on subshifts of finite type (see for example \cite{BQ2012,FH2010}).
	
	\begin{proposition}[\text{\cite[Proposition 1 and Theorem 1.1]{BQ2012}}]\label{Qu}
		Write$$
		\Lambda_{\Phi/\Psi}(\beta):=\left\{x\in\Sigma_{A}:\lim\limits_{n\to\infty}\frac{\phi_n(x)}{\psi_n(x)}=\beta\right\}\ \ \text{and}\ \  L_{\Phi/\Psi}:=\left\{\frac{\Phi_*(\mu)}{\Psi_*(\mu)}:\mu\in\mathcal{M}(\Sigma_A,T)\right\}.
		$$
		Then the set
		$\Lambda_{\Phi/\Psi}(\beta)\neq\emptyset\Leftrightarrow
		\beta\in L_{\Phi/\Psi}$. Moreover, if $\beta\in L_{\Phi/\Psi}$, then 
		\begin{equation*}
			\dim_H\Lambda_{\Phi/\Psi}(\beta)=\inf_{q\in{\mathbb{R}}}\mathcal{L}_{\Phi/\Psi}(q,\beta),
		\end{equation*}
		where $\mathcal{L}_{\Phi/\Psi}(q,\beta)$ is the unique solution $
		t=t(q,\beta)$ such that $
		P_{top}(q\Phi+(t-q\beta)\Psi)=0.$
	\end{proposition}	
	\begin{proposition}
		\label{dos-mu-ana} 
		(1) Define $L:=\left\{P(0)/\Psi_*(\mu):\mu\in\mathcal{M}%
		(\Sigma_A,T)\right\}$, then  
		\begin{equation}\label{beta=psi}
			L=\left[-\frac{P(0)}{P'(-\infty)},-\frac{P(0)}{P'(\infty)}\right].
		\end{equation}
		Moreover, we have
		\begin{equation}\label{beta-lower-upper}
			-\frac{P(0)}{P'(-\infty)}=\inf_{x\in\Sigma_{A}}\underline{d}_{\mu_{\max}}(x)\ \ \text{and}\ \ \ -\frac{P(0)}{P'(\infty)}=	\sup_{x\in\Sigma_{A}}\overline{d}_{\mu_{\max}}(x).
		\end{equation}
		Thus the optimal H\"older exponent of $\mu_{\max}$ is $-\frac{P(0)}{P'(-\infty)}$.
		
		(2) Define $\Lambda_{\beta}:=\left\{\mw\in\Sigma_A:d_{\mu_{\max}}(\mw)=\beta\right\}$, then $\Lambda_{\beta}\neq\emptyset$
		if and only if $\beta\in L$. Moreover, for any $\beta\in L$, we have 
		\begin{equation*}
			\dim_H\Lambda_{\beta}=\inf_{q\in{\mathbb{R}}}(\tau(q)+q\beta),
		\end{equation*}
		where $\tau(q)$ is the zero of $P_{top}(q\Phi+\tau\Psi)=0$. 
	\end{proposition}
	
	\begin{proof}
		(1) By \cite[Lemmas A.3(1) and A.4(iii)]{FH2010}, the set $\left\{\Psi_*(\mu):\mu\in\mathcal{M}%
		(\Sigma_A,T)\right\}$ is an interval which equals to $[\beta_*,\beta^*]$, where 
		\begin{equation*}
			\beta_*=\lim\limits_{n\to \infty}\dfrac{1}{n}%
			\inf\limits_{\mw\in\Sigma_A}\psi_{n}(\mw) \ \ \text{and} \ \
			\beta^*=\lim\limits_{n\to \infty}\dfrac{1}{n}%
			\sup\limits_{\mw\in\Sigma_A}\psi_{n}(\mw).
		\end{equation*}
		Now combine with Corollary \ref{psi-pressure}, we see that \eqref{beta=psi} holds. Moreover, equations \eqref{beta-lower-upper} follow directly from Lemma \ref{mu-phi-loc} and Corollary \ref{psi-pressure}.
		
		(2) By statement (1) and
		Proposition \ref{Qu}, we conclude that $\Lambda_{\Phi/\Psi}(\beta)=\Lambda_{%
			\beta},\ L_{\Phi/\Psi}=L$ and $\mathcal{L}_{\Phi/\Psi}(q,\beta)=\tau(q)+q\beta$, then
		$\Lambda_{\beta}\neq\emptyset$ if and only if $\beta\in L$. Moreover, if $%
		\beta\in L$, we have 
		\begin{equation*}
			\dim_H\Lambda_{\beta}=\inf_{q\in{\mathbb{R}}}(\tau(q)+q\beta).
		\end{equation*}
		The proof is complete.
	\end{proof}
	
	\section{Structure of the spectrum of the Sturm Hamiltonian}\label{sec-spectrum}
	
	In this section we discuss the structure of the spectrum and give a symbolic space to code it. We also summarize known results and connections which we will use.
	\subsection{The covering structure of the spectrum}\label{sec-spectrum-1}\
	
	Following \cite{LW2004,Ra1997}, we describe the covering structure of the spectrum $\Sigma_{\alpha,\lambda}$. Recall that the Sturmian potential is given by $v_k=\lambda\chi_{[1-%
		\alpha,1)}(k\alpha+\theta\pmod 1). $ Since $\Sigma_{\alpha,\lambda}$ is
	independent of the phase $\theta$, in the rest of the paper we will take $%
	\theta=0$. 
	
	Assume that $\alpha\in \mathbb{I}$ has continued fraction expansion $[a_1,a_2,\cdots]$. Let $p_n/q_n(n\geq0)$ be the $n$-th partial quotient of $\alpha$, given by 
	\begin{align}\label{q-n}
		\begin{cases}
			p_{-1}=1, p_0=0,\  &p_{n+1}=a_{n+1}p_n+p_{n-1},\ n\geq0, \\
			q_{-1}=0, q_0=1,\ &q_{n+1}=a_{n+1}q_n+q_{n-1},\ n\geq0.
		\end{cases}
	\end{align}
	For any $n\in {\mathbb{N}}$ and $E\in\mathbb{R}$, the \textit{%
		transfer matrix} $M_n(E)$ over $q_n$ sites is defined by 
	\begin{equation*}
		M_n(E):= \left[%
		\begin{array}{cc}
			E-v_{q_n} & -1 \\ 
			1 & 0%
		\end{array}%
		\right] \left[%
		\begin{array}{cc}
			E-v_{q_n-1} & -1 \\ 
			1 & 0%
		\end{array}%
		\right] \cdots \left[%
		\begin{array}{cc}
			E-v_1 & -1 \\ 
			1 & 0%
		\end{array}%
		\right].
	\end{equation*}
	By convention, we define 
	\begin{equation*}
		\begin{array}{l}
			M_{-1}(E):= \left[%
			\begin{array}{cc}
				1 & -\lambda \\ 
				0 & 1%
			\end{array}%
			\right]%
		\end{array}
		\ \ \ \text{ and }\ \ \ 
		\begin{array}{l}
			M_{0}(E):= \left[%
			\begin{array}{cc}
				E & -1 \\ 
				1 & 0%
			\end{array}%
			\right].%
		\end{array}%
	\end{equation*}
	
	\indent For $n\ge0$ and $p\ge-1$, define $$h_{(n,p)}(E):=\mathrm{tr} (
	M_{n-1}(E) M_n^p(E))	\ \ \ \text{ and }\ \ \ 
	\sigma_{(n,p)}:=\{E\in\mathbb{R}:|h_{(n,p)}(E)|\leq2\},
	$$
	where $\mathrm{tr}M$ stands for the trace of the matrix $M$. For
	any $n\ge 0,$ we have
	\begin{equation*}
		\left(\sigma_{(n+2,0)}\cup\sigma_{(n+1,0)}\right)\subset
		\left(\sigma_{(n+1,0)}\cup\sigma_{(n,0)}\right)\ \ \text{and}\
		\		\Sigma_{\alpha,\lambda}=\bigcap_{n\ge0}\left(\sigma_{(n+1,0)}
		\cup\sigma_{(n,0)}\right).
	\end{equation*}
	The set $\sigma_{(n,p)}$ is made of finitely many disjoint intervals.
	Each interval of $\sigma_{(n,p)}$ is called a \emph{band}. Assume $%
	B $ is a band of $\sigma_{(n,p)}$, then $h_{(n,p)}(E)$ is
	monotone on $B$ and $h_{(n,p)}(B)=[-2,2].$ We call $%
	h_{(n,p)}(\cdot)$ the \emph{generating polynomial} of $B$.
	
	Note that the family $\{\sigma_{(n+1,0)}\cup\sigma_{(n,0)}:n\ge 0\}$ forms a decreasing family of coverings of $\Sigma_{\alpha,\lambda}$. However there are some repetitions between $%
	\sigma_{(n+1,0)}\cup\sigma_{(n,0)}$ and $\sigma_{(n+2,0)}\cup\sigma_{(n+1,0)}$. When $\lambda>4$, it is possible to choose a covering of $\Sigma_{\alpha,\lambda}$ elaborately such
	that we can get rid of these repetitions, as we will describe in the follows:
	
	\begin{definition}[\protect\cite{LW2004,Ra1997}]
		For $\lambda>4$ and $n\ge0$, define three types of bands as:
		
		$(n,\mathbf{1})$-type band: a band of $\sigma_{(n,1)}$ contained in
		a band of $\sigma_{(n,0)};$
		
		$(n,\mathbf{2})$-type band: a band of $\sigma_{(n+1,0)}$ contained
		in a band of $\sigma_{(n,-1)};$
		
		$(n,\mathbf{3})$-type band: a band of $\sigma_{(n+1,0)}$ contained
		in a band of $\sigma_{(n,0)}.$
	\end{definition}
	
	All three types of bands actually occur and they are disjoint. We call these
	bands \emph{\ spectral generating bands of order $n$}. For any $n\ge0$,
	define 
	\begin{equation*}
		{\mathcal{B}}_n^\alpha:=\{B: B \text{ is a spectral generating band of
			order } n\}.
	\end{equation*}
	
	The basic covering structure of $\Sigma_{\alpha,\lambda}$ is described in this
	proposition:
	
	\begin{proposition}[\protect\cite{LW2004,Ra1997}]
		\label{basic-struc} Fix $\alpha\in \mathbb{I}$ and $\lambda>4$. We have
		
		(1) For any $n\ge0$, 
		$
		\sigma_{(n+2,0)}\cup\sigma_{(n+1,0)}\subset \bigcup_{B\in{%
				\mathcal{B}}_n^\alpha}B \subset
		\sigma_{(n+1,0)}\cup\sigma_{(n,0)},
		$ thus $\{{\mathcal{B}}_n^\alpha:n\ge0\}$ are nested and 
		\begin{equation*}
			\Sigma_{\alpha,\lambda}=\bigcap_{n\ge0} \bigcup_{B\in{\mathcal{B}}_n^\alpha}B.
		\end{equation*}
		
		(2) Any $(n,\mathbf{1})$-type band contains
		only one band in ${\mathcal{B}}_{n+1}^\alpha$, which is of $(n+1,\mathbf{%
			2})$-type.
		
		(3) Any $(n,\mathbf{2})$-type band contains 
		$2a_{n+1}+1$ bands in ${\mathcal{B}}_{n+1}^\alpha$, $a_{n+1}+1$ of which
		are of $(n+1,\mathbf{1})$-type and $a_{n+1}$ of which are of $(n+1,\mathbf{3}%
		)$-type. Moreover, the $(n+1,\mathbf{1})$-type bands interlace the $(n+1,\mathbf{3})$-type bands.
		
		(4) Any $(n,\mathbf{3})$-type band contains 
		$2a_{n+1}-1$ bands in ${\mathcal{B}}_{n+1}^\alpha$, $a_{n+1}$ of which
		are of $(n+1,\mathbf{1})$-type and $a_{n+1}-1$ of which are of $(n+1,\mathbf{%
			3})$-type. Moreover, the $(n+1,\mathbf{1})$-type bands interlace the $(n+1,%
		\mathbf{3})$-type bands.
	\end{proposition}
	
	Thus $\{{\mathcal{B}}_n^\alpha: n\ge0\}$ form a natural covering of the
	spectrum $\Sigma_{\alpha,\lambda}$ (\cite{LPW2007,LW2005}).
	
	\begin{remark}
		\label{B-n} Note that there are only two spectral generating bands of
		order $0$: one is $\sigma_{(0,1)}=[\lambda-2,\lambda+2]$ with
		generating polynomial $h_{(0,1)}(E)=E-\lambda$ and type $(0,\mathbf{1%
		})$; the other is $\sigma_{(1,0)}=[-2,2]$ with generating
		polynomial $h_{(1,0)}(E)=E$ and type $(0,\mathbf{3})$. Thus 
		\begin{equation}\label{B-0}
			{\mathcal{B}}_0^\alpha=\{[\lambda-2,\lambda+2], [-2,2]\}.
		\end{equation}
	\end{remark}
	
	\subsection{The symbolic space and the coding of $\Sigma_{\alpha,\lambda}$}
	\label{sec-spectrum-2}\ 
	
	In this subsection we describe the coding of the spectrum $\Sigma_{\alpha,%
		\lambda}$ based on \cite{CQ2023,Qu2016,Qu2018}. Here we
	essentially follow \cite{CQ2023}.
	
	\subsubsection{The symbolic space $\Omega^{\protect\alpha}$}
	
	At first, for any $\alpha\in{\mathbb{I}}$, we construct a symbolic space $\Omega^{\alpha}$ to code the spectrum $\Sigma_{\alpha,\lambda}$. 
	
	For each $n\in{\mathbb{N}},$ define an alphabet $\mathscr{A}_n$ as 
	\begin{equation*}\label{alphabet-n}
		\mathscr{A}_n:=\{(\mathbf{1},i)_n:i=1,\cdots,n+1\}\cup \{(\mathbf{2}%
		,1)_n\}\cup \{(\mathbf{3},i)_n:i=1,\cdots,n\}.
	\end{equation*}
	Then $\#\mathscr{A}_n=2n+2.$ Assume $e=(\mathbf{t},i)_n\in \mathscr{A}_n$,
	we call $\mathbf{t}$ the \textit{type} of $e$ and write $\mathbf{t}(e):=\mathbf{t}$.
	
	Fix $m\in {\mathbb{N}}.$ Given ${\mathbf{t}}\in \{\mathbf{1},\mathbf{2},\mathbf{3}\}$ and $\hat e\in 
	\mathscr{A}_{m}$, we call ${\mathbf{t}}\hat e$ \textit{admissible}, denote
	by ${\mathbf{t}}\to \hat e,$ if 
	\begin{align}\label{admissible-T-A}
		({\mathbf{t}},\hat e)\in& \{(\mathbf{1},(\mathbf{2},1)_m)\}\cup  \notag \\
		&\{(\mathbf{2},(\mathbf{1},i)_m): 1\le i\le m+1\}\cup \{(\mathbf{2},(\mathbf{%
			3},i)_m): 1\le i\le m\}\cup \\
		&\{(\mathbf{3},(\mathbf{1},i)_m):1\le i\le m\}\cup \{(\mathbf{3},(\mathbf{3}%
		,i)_m):1\le i\le m-1\}.  \notag
	\end{align}
	For any $e\in{\mathscr A}_n$ and $\hat e\in {\mathscr A}_m$, we call $e\hat
	e $ \textit{admissible}, denote by 
	\begin{equation}\label{admissible-A-A}
		e\to \hat e, \ \ \text{ if }\ \ \mathbf{t}(e)\to \hat e.
	\end{equation}
	
	For any $\alpha=[a_1,a_2,\cdots]\in{\mathbb{I}}$, define the \textit{symbolic
		space} $\Omega^{\alpha}$  as 
	\begin{eqnarray}\label{Omega^alpha}
		\Omega^{\alpha}:=\Big\{x=x_0x_1x_2\cdots \in\{\mathbf{1},\mathbf{3}\}\times\prod_{n=1}^%
		\infty {\mathscr A}_{a_n}: x_n\to x_{n+1}, n\ge0 \Big\}.
	\end{eqnarray}
	\begin{remark}\label{invarint}
		We warn that the space $\Omega^{\alpha}$ is not invariant under the shift map $\sigma(x)=\sigma((x_n)_{n\geq0})=(x_{n+1})_{n\geq0}$, i.e., $\sigma(\Omega^{\alpha})\not\subset\Omega^{\alpha}$. So we must introduce a new symbolic space $\Omega_{\mathbf{a}}$ to handle the spectral problem, see Sect. \ref{sec-Sturm-1}.
	\end{remark}
	For any $n\geq0$, we write $x|_n:=x_0x_1\cdots x_n$ for $x\in\Omega^{\alpha}$ and define 
	\begin{equation*}  \label{def-Omega^alpha_n}
		\Omega^{\alpha}_{n}:=\{x|_n: x\in \Omega^{\alpha}\}\ \ \text{ and }\ \
		\Omega^{\alpha}_{\ast}:=\bigcup_{n\ge 0} \Omega^{\alpha}_{n}.
	\end{equation*}
	
	Given $w=w_0\cdots w_n\in \Omega^{\alpha}_{n}$, 
	define
	\begin{equation}\label{def-tail}
		h_w:=w_0 \ \ \text{and}\ \ \tT(w):=\tT(w_n), 
	\end{equation}
	and we refer to $h_{w}$ as the {\it head} of $w$, and $\tT(w)$
	as the {\it tail type} of $w$, respectively.\footnote{More generally, if $w\in \prod_{j=1}^n \A_{a_j}$, then $h_w=x_1$ and $\tT(w)=\tT(w_n)$; and 
		if $w\in \prod_{j=1}^\infty \A_{a_j}$, then $h_w=w_1$.}

	For pair $({\mathscr A}_n, {\mathscr A}_m),$ we define the \textit{incidence
		matrix} $A_{nm}= (a_{e \hat e})$ of size $(2n+2)\times(2m+2)$ as follows: 
	\begin{equation}\label{A_ij}
		a_{e\hat e}= 
		\begin{cases}
			1, & \text{ if } e\to \hat e, \\ 
			0, & \text{ otherwise}.%
		\end{cases}%
	\end{equation}
	For any $n\in {\mathbb{N}}$, define an
	auxiliary matrix $\hat A_n$ as follows: 
	\begin{equation}\label{hat-A-n}
		\hat A_n= \left[%
		\begin{array}{ccc}
			0 & 1 & 0 \\ 
			n+1 & 0 & n \\ 
			n & 0 & n-1%
		\end{array}%
		\right].
	\end{equation}
	
	\subsubsection{The coding map $\protect\pi^{\protect\alpha}_\protect\lambda$}
	
	\label{sec-coding-map} Given $\lambda>4$ and $\alpha\in{\mathbb{I}}$, now we
	explain that $\Omega^{\alpha}$ is a coding of the spectrum $%
	\Sigma_{\alpha,\lambda}$. 
	
	At first, let
	$B_{\mathbf{1}}^{\alpha}=[\lambda-2,\lambda+2]$ and $B_{\mathbf{3}}^{\alpha}=[-2,2],
	$
	then ${\mathcal{B}}_0^{\alpha}=\{B_w^{\alpha}: w\in \Omega_0^{\alpha}\}$ by \eqref{B-0}.
	
	Assume $B_w^{\alpha}$ is defined for any $w\in\Omega_{n-1}^{\alpha} $ ($n\geq 1$) and
	$
	{\mathcal{B}}_{n-1}^{\alpha}=\{B_w^{\alpha}: w\in
	\Omega_{n-1}^{\alpha}\}.
	$
	Then for given $w\in\Omega_n^{\alpha}$, write 
	$
	w^{\prime }:=w|_{n-1}, w_n:=(\mathbf{t},j)_{a_n},
	$
	and define $B_w^{\alpha}$ to be the unique $j$-th band of $(n,\mathbf{t}%
	)$-type in ${\mathcal{B}}_{n}^{\alpha}$ which is contained in $B_{w^\prime}^{\alpha}$. By Proposition \ref{basic-struc} (2)-(4), $%
	B_w^{\alpha}$ is well-defined for every $w\in \Omega_n^{\alpha}$
	and 
	$
	{\mathcal{B}}_{n}^{\alpha}=\{B_w^{\alpha}: w\in
	\Omega_{n}^{\alpha}\}.
	$
	
	By induction, we can code all the bands in ${\mathcal{B}}_n^{\alpha}$
	by $\Omega_n^{\alpha}.$ Now we can define a natural map $\pi^{\alpha}_\lambda:\Omega^{\alpha}\to
	\Sigma_{\alpha,\lambda}$ as 
	\begin{equation}  \label{pi^a}
		\pi^{\alpha}_\lambda(x):=\bigcap_{n\ge 0} B_{x|_n}^{\alpha}.
	\end{equation}
	By Proposition \ref{basic-struc}, it is seen that $\pi^{\alpha}_\lambda$ is
	a bijection, thus $\Omega^{\alpha}$ is a \textit{coding} of $%
	\Sigma_{\alpha,\lambda}$. 
	
	Indeed, we have the following
	\begin{proposition}\cite[Proposition 5.2]{Qu2018}\label{bi-lip-alpha}
		There exists a weak-Gibbs metric on $\Omega^{\alpha}$, defined as
		\begin{equation*}
			d_{\alpha}(x,y):=|B^{\alpha}_{x\wedge y}|
		\end{equation*}
		such that $\pi^{\alpha}_{\lambda}:(\Omega^{%
			\alpha},d_{\alpha})\to(\Sigma_{\alpha,\lambda},|\cdot|)$ is a bi-Lipschitz homeomorphism. 
	\end{proposition}
	\begin{remark}
		From the above proposition, we can transform the spectral problem of $\Sigma_{\alpha,\lambda}$ to the study on the symbolic space $\Omega^{\alpha}$. However, as indicated in Remark \ref{invarint}, the symbolic space $\Omega^{\alpha}$ is not invariant; therefore, in what follows, we often work on the symbolic space $\Omega_{\mathbf{a}}$, which is a coding of $\Sigma_{\check{\alpha},\lambda},\check{\alpha}=[1,\overline{a_1,a_2,\cdots,a_k}]$, see Proposition \ref{Omega-Sigma}.
	\end{remark}
	%
	\subsection{Useful results for the Sturm Hamiltonian}\
	
	In this subsection, we collect some useful results of the Sturm Hamiltonian for later use.
	
	The following lemma provides estimates for the length of the spectral generating band:
	
	\begin{lemma}[\text{\cite[Lemma 3.7]{LQW2014}}]\label{esti-band-length}
		Assume that $\lambda >20$ and $\alpha =[a_{1},a_{2},\cdots ]\in {\mathbb{%
				I}}$. Write $\tau _{1}:=(\lambda -8)/3$ and $\tau
		_{2}:=2(\lambda +5)$. Then for any $w=w_{0}w_1\cdot \cdot \cdot w_{n}\in \Omega
		_{n}^{\alpha }$, we have 
		\begin{equation*}
			\tau _{2}^{-n}\prod_{i=1}^{n}a_{i}^{-3}\cdot \prod_{1\leq i\leq n;\mathbf{t}(w_{i})=\mathbf{2}}\tau _{2}^{2-a_{i}} \leq |B_{w}^{\alpha }|\leq
			4\tau _{1}^{-n}\Big( \prod_{1\leq i\leq n;\mathbf{t}(w_{i})=\mathbf{2}%
			}\tau _{1}^{2-a_{i}}\Big).
		\end{equation*}
		In particular, we have
		$
		|B_{w}^{\alpha}|\leq2^{2-n}.
		$
	\end{lemma}
	
	The following proposition is \cite[Theorem 5.1]{FLW2011}, see also \cite[Theorem 3.3]{LQW2014}.
	
	\begin{proposition}[Bounded covariation]
		\label{bco-cor} Let $\lambda>20,\alpha=[a_1,a_2,\cdots],\beta=[b_1,b_2,%
		\cdots]$ be irrational with $a_n,b_n$ bounded by $M$. Then, 
		there exists $\eta=\eta(\lambda,M)>1$ such that for any $w,wu\in\Omega_\ast^{\alpha}$
		and $\widetilde{w}, \widetilde{w}u\in \Omega_\ast^{\beta}$,
		\begin{equation*}
			\eta^{-1} \frac{|{B}_{wu}^{\alpha}|}{|{B}_{w}^{\alpha}|}%
			\le \frac{|B_{\widetilde{w}u}^{\beta}|}{|B_{\widetilde{w}}^{\beta}|%
			}\le \eta \frac{|{B}_{wu}^{\alpha}|}{|{B}_{w}^{\alpha}|}.
		\end{equation*}
	\end{proposition}
	
	We remark that in \cite{FLW2011}, only the case $\alpha=\beta$ is considered. However by checking the proof, one can indeed show the stronger result as stated in Proposition \ref{bco-cor}. 
	
	The following proposition connects the longest spectral generating band and the transport exponents:
	\begin{proposition}[\text{\cite[Proposition 4.8(c) and Proposition 5.1]{DGLQ2015}}]\label{trans}
		Let $\lambda>20$. Assume $\alpha=[a_1,a_2,\cdots]\in\I$ is such that $\lim\limits_{n\to\infty}\dfrac{1}{n}\log q_n(\alpha)$ exists and $\limsup\limits_{n\to\infty}\dfrac{1}{n}\sum\limits_{k=1}^{n}
		a_k<\infty$, then 
		\begin{equation*}
			\frac{\lim\limits_{n\to\infty}\frac{1}{n}\log q_n(\alpha)}{%
				-\liminf\limits_{n\to\infty}\frac{1}{n}\log|B^{\alpha}_{n,\max}|}\leq
			\mathcal{T}^-(\alpha,\lambda)\leq\mathcal{T}^{+}(\alpha,\lambda)\leq%
			\frac{\lim\limits_{n\to\infty}\frac{1}{n}\log q_n(\alpha)}{%
				-\limsup\limits_{n\to\infty}\frac{1}{n}\log|B^{\alpha}_{n,\max}|},
		\end{equation*}
		where $|B^{\alpha}_{n,\max}|=\max\{|B^{\alpha}_w|:w\in\Omega^{\alpha}_n\}$ is the longest spectral generating band of order $n$. 
	\end{proposition}
	
	The following lemma is a characterization of the DOS:
	\begin{lemma}[\text{\cite[Proposition 5.5]{CQ2023}}]\label{dos-measure} 
		Let $\lambda>20$ and $\alpha=[a_1,a_2,\cdots]\in\I$ with $a_n$ bounded by $M$. Then there exists a constant $C=C(M)>1$ such that for any $w\in\Omega^{\alpha}_n$, we have 
		\begin{equation*}  \label{def-C-i}
			\frac{C^{-1}}{q_{n}(\alpha)}\leq\mathcal{N}_{\alpha,\lambda}(B^{\alpha}_{w})\leq\frac{C}{q_{n}(\alpha)}.
		\end{equation*}
	\end{lemma}
	The following proposition characterizes the local structure of $\Sigma_{\alpha,\lambda}$ and $\mathcal{N}_{\alpha,\lambda}$:
	\begin{proposition}[\text{\cite[Lemma 5.3 and Proposition 5.7]{CQ2023}}]\label{geo-lem-tail}
		Let $\lambda>20$, $\alpha=[a_1,a_2,\cdots],\beta=[b_1,b_2,%
		\cdots]$ be irrational with $a_n,b_n$ bounded by $M$ such
		that $G^n(\alpha)=G^m(\beta)$ for some $n,m\in {\mathbb{N}}$, where $G$ is the Gauss map defined by $G(x)=1/x-[1/x]$ (where $[\cdot]$ is the integer part). If $l\ge0$
		and $u\in \Omega^{\alpha}_{n+l}, v\in \Omega^{\beta}_{m+l}$ are such that $		\mathbf{t}(u)=\mathbf{t}(v)$, then we have
		
		(1)  There exists a natural bi-Lipschitz map $\tau_{uv}:X_u^{\alpha}\to X_v^{\beta}$, where $X^{\alpha}_w$ is the \textit{basic set} of $\Sigma_{\alpha,\lambda}$, defined as 
		\begin{equation}  \label{def-basic-set}
			X_w^{\alpha}:=\pi^{\alpha}_\lambda([w]^{\alpha}),\quad [w]^{\alpha}:=\{x\in \Omega^{\alpha}: x|_n=w\}.
		\end{equation}
		
		(2) $(\tau_{uv})_*(\mathcal{N}_{\alpha,\lambda}|_{X_u^\alpha})\asymp \mathcal{N}_{\beta,\lambda}|_{X_v^\beta}$. In particular, for any $x\in X_u^\alpha$, we have
		\begin{equation*}\label{equal-loc-dim}
			\underline{d}_{\mathcal{N}_{\alpha,\lambda}}(x)=\underline{d}_{\mathcal{N}_{\beta,\lambda}}(\tau_{uv}(x));\ \ \
			\overline{d}_{\mathcal{N}_{\alpha,\lambda}}(x)=\overline{d}_{\mathcal{N}_{\beta,\lambda}}(\tau_{uv}(x)).
		\end{equation*}
		Consequently, $\mathcal{N}_{\alpha,\lambda}$ is exact-dimensional if and only if $\mathcal{N}_{\beta,\lambda}$ is exact-dimensional.
	\end{proposition}
	We remark that Lemma \ref{dos-measure} and Proposition \ref{geo-lem-tail} in \cite{CQ2023} were established for $\lambda\geq24$, our verification demonstrates that these results remain valid for $\lambda>20$ if $\alpha$ is bounded-type.
	
	\section{The Sturm Hamiltonian and thermodynamical formalism}\label{sec-Sturm}
	In this section, based on the structure of the spectrum, we connect the Sturm Hamiltonian with thermodynamical formalism. At first, we define the symbolic space $\Omega_{\ma}$  as a counterpart to the spectrum. Next, by using the spectral generating bands, we define the potential and pressure function. Finally, we define an appropriate weak-Gibbs metric and prove that the DOS is exact-dimensional.
	
	Throughout the remainder of this paper, when not explicitly mentioned, we set 
	$\lambda>20,k\in\N,\mathbf{a}=a_1a_2\cdots a_k\in\N^k,\check{\alpha}=[1,\overline{a_1,a_2,\cdots,a_k}]$ as a standing convention.
	\subsection{The topologically mixing symbolic space $(\Omega_{\ma},T_{\mathbf{a}})$}\label{sec-Sturm-1}\ 
	
	We will construct the symbolic space $(\Omega_{\ma},T_{\ma})$ as the subshift of finite type with topologically mixing property. 
	\subsubsection{The symbolic space $(\Omega_{\ma},T_{\ma})$}
	Define the alphabet $\mathcal{A}_{\mathbf{a}}$ as 
	\begin{equation*}
		\mathcal{A}_{\mathbf{a}}:=\Big\{e_{1}\cdots e_{k}:e_i\in\mathscr{A}_{a_i}, e_{i}\rightarrow e_{i+1},1\leq i<k\Big\}.
	\end{equation*}
	The alphabet set $\mathcal{A}_{\ma}$ captures all admissible words of order $k$ contained in $ {\mathscr A}_{a_1}\times\cdots\times{\mathscr A}_{a_k}$.	
	
	For any $\mv,\mathbf{w}\in \mathcal{A}_{\ma}$, define
	the \textit{incidence matrix} $A_{\ma}=(a_{\mv\mathbf{w}})$ of size $\#\mathcal{A}_{\ma}\times \#\mathcal{A}_{\ma}$
	as
	\begin{equation*}\label{def-A}
		a_{\mv\mathbf{w}}:=%
		\begin{cases}
			1, & \text{ if }\mathbf{t}(\mv)\rightarrow h_{\mathbf{w}}, \\ 
			0, & \text{ otherwise},
		\end{cases}%
	\end{equation*}
	where $\tT(\mv)$ is the tail type of $\mv$ and $h_{\mathbf{w}}$ is the head of $\mathbf{w}$ (see \eqref{def-tail}).
	By equations \eqref{A_ij} and \eqref{hat-A-n}, we define a related
	matrix $\hat{A}_{\ma}:=\hat{A}_{a_k}\cdots\hat{A}_{a_2}\hat{A}_{a_1}$.
	
	Now we construct the symbolic space $\Omega _{\mathbf{a}}$ indexed by $\ma$ as: 
	\begin{align*}
		\Omega _{\mathbf{a}}:&=\Big\{\mathbf{v}=\mathbf{v}_{1}\mathbf{v}_{2}\cdots \in 
		{\mathcal A}_{\mathbf{a}}^{{\mathbb{N}}}:\mathbf{v}_{i}\in \mathcal{A}_{\mathbf{a}},\  \mathbf{v}_{i}\rightarrow \mathbf{v}%
		_{i+1},\ i\in\N\Big\}
	\end{align*}
	\begin{remark}\label{a-Omega-a}
		Intuitively, the set $\Omega_{\ma}$ consists of infinite paths that have elements of set $\mathscr{A}_{a_1}$ as their first letters and are constructed according to the same relations (see \eqref{admissible-T-A} and \eqref{admissible-A-A}); therefore, the symbolic space $\Omega_{\ma}$ has the following representation:
		\begin{equation*}
			\Omega_{\ma}=\Big\{x=x_1x_2\cdots \in\Big({\mathscr A}_{a_1}\times\cdots\times{\mathscr A}_{a_k}\Big)\times\Big({\mathscr A}_{a_1}
			\times\cdots\times{\mathscr A}_{a_k}\Big)\cdots: x_n\to x_{n+1}, n\in\N \Big\}.
		\end{equation*}
	\end{remark}
	
	For any $n\in {\mathbb{N}}$, define the set of \textit{admissible words of
		order $n$} of $\Omega _{\mathbf{a}}$ as 
	\begin{equation*}
		\Omega _{\mathbf{a},n}:=\left\{\mathbf{v}_{1}\mathbf{v}
		_{2}\cdots \mathbf{v}_{n}\in {\mathcal A}_{\mathbf{a}}^n:\  \mathbf{v}_{i}\rightarrow \mathbf{v}%
		_{i+1},\ 1\leq i< n \right\}. 
	\end{equation*}%
	For any $\mathbf{w}\in \Omega _{\mathbf{a},n}$, we define the cylinder $[\mathbf{w}]_{\ma}$ as
	\begin{equation}\label{cylinder-a}
		\lbrack \mathbf{w}]_{\mathbf{a}}:=\{\mv\in \Omega _{\mathbf{a}}:\mv|_{n}=\mathbf{w}\},
	\end{equation}
	where $\mv|_{n}:=\mathbf{v}_{1}\mathbf{v}_{2}\cdots \mathbf{v}_{n}$ is the $n$-th prefix of $\mv$.
	
	Define the shift map $T_{\ma}:\Omega_{\ma}\to\Omega_{\ma}$ as $T_{\ma}(\mv)=T_{\ma}((\mv_n)_{n\in\N})=(\mv_{n+1})_{n\in\N}$.
	Then $(\Omega_{\ma},T_{\ma})$ is the subshift of finite type with alphabet $\mathcal{A}_{\mathbf{a}}$ and incidence matrix  $A_{\mathbf{a}}$. 
	\subsubsection{$\Omega_{\ma}$ is a coding of the spectrum $\Sigma_{\check{\alpha},\lambda}$}
	Note that $\Omega_{\ma}$ and $\Omega^{\alpha}$ are formally similar yet distinct, where $\alpha=[\overline{a_1,a_2,\cdots,a_k}]$. A simple but crucial observation is that, there exists a natural bijection between $\Omega_{\ma}$ and $\Omega^{\check{\alpha}}$, where $\check{\alpha}=[1,\overline{a_1,a_2,\cdots,a_k}]$. 
	
	Define a map $\iota :\Omega _{%
		\mathbf{a}}\rightarrow \{\mathbf{1},\mathbf{3}\}\times \mathscr{A}_{1}\times %
	\mathcal{A}_{\mathbf{a}}^{{\mathbb{N}}}$ as 
	\begin{equation}\label{def-iota0}
		\iota(\mathbf{v}):=\varpi^{\mv}\mv,\ \ \ \text{where}\ \varpi^{\mv}:=\begin{cases}
			\mathbf{1}(\mathbf{2},1)_{1},\  & \mbox{if}\ \ \mathbf{t}(h_{\mv})=\mathbf{1}\ \text{or}\ \mathbf{3}, \\ 
			\mathbf{3}(\mathbf{1},1)_{1},\  & \mbox{if}\ \ \mathbf{t}(h_{\mv})=\mathbf{2}.%
		\end{cases}%
	\end{equation}
	Then, $\iota$ maps $\Omega_{\ma}$ to $\Omega^{\check{\alpha}}$  and satisfies the following property:
	
	\begin{proposition}\label{Omega-Sigma} 
		The map $\iota:\Omega_{\ma}\to\Omega^{\check{\alpha}}$ is
		a bijection. As a consequence, the natural map  $\pi_{\ma,\lambda}:\Omega_{\ma}\to\Sigma_{\check{\alpha},\lambda}$ is a coding map, where $\pi_{\ma,\lambda}=\pi^{\check{%
				\alpha}}_{\lambda}\circ\iota$.
	\end{proposition}
	
	\begin{proof}
		Note that the symbolic space $\Omega^{\check{\alpha}}$ can be expressed as: 
		\begin{eqnarray*}
			\Omega^{\check{\alpha}}=\Big\{\mv_{-1}\mathbf{v}_0\mathbf{v}_1\cdots: \mathbf{v}_{-1}\in \{\mathbf{1},\mathbf{3}\},\
			\mv_0\in\mathscr{A}_{1},\ \mathbf{v}_{i}\in{%
				\mathcal A}_{\mathbf{a}},\ i\geq 1;\ \mathbf{v}_i\to \mathbf{v}%
			_{i+1}, i\ge-1\Big\}.
		\end{eqnarray*}
		By equations \eqref{def-iota0}, \eqref{admissible-A-A} and \eqref{admissible-T-A}, we have that $\iota(\Omega_{\mathbf{a}})\subset \Omega^{\check{\alpha}}$. Utilizing \eqref{def-iota0} and %
		\eqref{admissible-T-A} once more, one can observe that 
		\begin{equation*}
			\left\{\mv_{-1}\mv_{0}:\mv=\mv_{-1}\mv_0\mv_1\mv_2\cdots\in\Omega^{%
				\check{\alpha}}\right\}=\left\{\mathbf{1}(\mathbf{2},1)_1,\mathbf{3}(\mathbf{%
				1},1)_1\right\}.
		\end{equation*}
		So for any $\mv=\mv_{-1}\mv_0\mv_1\cdots\in \Omega^{\check{%
				\alpha}}$, we have $\mv^*:=\mv_1\mv_2\cdots\in \Omega_{%
			\mathbf{a}}$ and $\iota(\mv^*)=\mv$ by \eqref{def-iota0}. Hence $\iota(\Omega_{%
			\mathbf{a}})=\Omega^{\check{\alpha}}$. Since $\iota$ is injective, we
		conclude that $\iota: \Omega _{\mathbf{a}}\to \Omega^{\check{\alpha}}$ is
		bijective.
		
		By \eqref{pi^a} and $\iota:\Omega _{\mathbf{a}}\to\Omega^{\check{\alpha}}$ is
		a bijection, it is seen that $\Omega_{\ma}$ is a coding of $\Sigma_{%
			\check{\alpha},\lambda}$.
	\end{proof}
	\begin{remark}
		Later we will construct a compatible metric $d_{\ma}$ on $\Omega_{\ma}$ such that
		$\pi_{\ma,\lambda}:(\Omega_{\ma},d_{\ma})\to(\Sigma_{\check{\alpha},\lambda},|\cdot|)$ is a bi-Lipschitz homeomorphism, see Proposition \ref{bi-Lip}.
	\end{remark}
	
	\subsubsection{$(\Omega_{\ma},T_{\ma})$ is topologically mixing}
	Recall that a nonnegative square matrix $B$ is called primitive if there exists some $m\in\N$ such that all the entries of $B^m$ are positive. 
	
	The following lemma implies that indeed $(\Omega_{\ma},T_{\ma})$ is topologically mixing.
	
	\begin{lemma}\label{primitive}
		$\hat{A}_{\ma},A_{\ma}$ are primitive and have the same Perron-Frobenius eigenvalue $E_{\ma}$. Moreover, we have $q_{kn}(\alpha)\sim E_{\ma}^n$ for all $\alpha\in\mathcal{P}(\ma)$, 
	\end{lemma}
	\begin{proof}
		It is straightforward to check that all the entries of $\hat{A}_{\ma}^5$ are positive,
		thus $\hat{A}_{\ma}$ is primitive. 
		
		Note that there exist matrices $$
		P=\left(\begin{array}{ccc}
			1&0&1\\
			-1&1&0\\ 
			-1&1&-1
		\end{array}\right)\ \ \text{and}\ \ Q_m=\left(%
		\begin{array}{cc}
			m & 1 \\ 
			1 & 0%
		\end{array}%
		\right)$$
		such that 
		\begin{equation*}
			\hat{A}_{\ma}=P\left(\begin{array}{cc}
				(-1)^k&0\\
				0&B_{\ma} 
			\end{array}\right)P^{-1},\ \  \text{where}\ B_{\mathbf{a}}=Q_{a_k}\cdots Q_{a_{2}}Q_{a_1}.
		\end{equation*}
		It follows that 
		$$\text{det}(\lambda I_3-\hat{A}_{\ma})=(\lambda-(-1)^k)\text{det}(\lambda I_2-B_{\ma}).$$
		Thus $\hat{A}_{\ma}$ and $B_{\ma}$ have the same Perron-Frobenius eigenvalue $E_{\ma}$. On the other hand, if we consider the graph related to the incidence matrix $A_{\ma}$, then it is easy to show that the graph is aperiodic. Consequently $A_{\ma}$ is primitive. Let $N=\#\mathscr{A}_{\ma}$, by direct computation,
		$$\text{det}(\lambda I_N-A_{\ma})=\lambda^{N-3}\text{det}(\lambda I_3-\hat{A}_{\ma}).$$
		Thus the Perron-Frobenius eigenvalue of $A_{\ma}$ is also $E_{\ma}$.
		
		Fix $\alpha\in\mathcal{P}(\ma)$ and by \eqref{q-n}, for any $n\in{\mathbb{N}%
		}$, we have 
		\begin{equation*}
			\left( 
			\begin{array}{c}
				q_{kn}(\alpha) \\ 
				q_{kn-1}(\alpha)
			\end{array}
			\right)=\left(\left( 
			\begin{array}{cc}
				a_k & 1 \\ 
				1 & 0%
			\end{array}
			\right)\cdots\left( 
			\begin{array}{cc}
				a_1 & 1 \\ 
				1 & 0%
			\end{array}
			\right)\right)^n\left( 
			\begin{array}{c}
				1 \\ 
				0%
			\end{array}
			\right)=B_{\ma}^n\left( 
			\begin{array}{c}
				1 \\ 
				0%
			\end{array}
			\right).
		\end{equation*}
		From this it is easy to show that there exist two constants $c_{\mathbf{a}%
		},d_{\mathbf{a}}$ such that 
		\begin{equation*}
			q_{kn}(\alpha)=c_{\mathbf{a}}E_{\mathbf{a}}^n+d_{\mathbf{a}}(-E_{\mathbf{a}%
			})^{-n}\sim E_{\mathbf{a}}^{n}. \label{qE}
		\end{equation*}
		The proof is complete.
	\end{proof}

	\subsection{The geometric potential and weak-Gibbs metric on $\Omega_{\mathbf{a}}$}\label{sec-weak-metric}\
	
	Next we define the geometric potential $\Psi^{\ma}$ which captures the exponential rate of the length of the generating bands and can be viewed as Lyapunov exponent function. We demonstrate that $\Psi^{\ma}$ can admit a weak-Gibbs metric $d_{\ma}$ such that $\pi_{\ma,\lambda}:(\Omega_{\mathbf{a}%
	},d_{\ma})\to(\Sigma_{\check{\alpha},\lambda},|\cdot|)$ is a bi-Lipschitz
	homeomorphism.
	\subsubsection{The geometric potential $\Psi^{\ma}$}
	For any $n\in\N$, we define $\psi_n^{\ma}:\Omega_{\ma}\to\R$ as 
	\begin{equation}\label{def-psi}
		\psi^{\ma}_{n}(\mv):=\log|B^{\check{\alpha}}_{\varpi^{\mv}\mv|_n}|=%
		\begin{cases}
			\log|B^{\check{\alpha}}_{\mathbf{1}(\mathbf{2},1)_1\mv|_n}|, \  & %
			\mbox{if}\ \ \mathbf{t}(h_{\mv})=\mathbf{1}\ \text{or}\ \mathbf{3}, \\ 
			\log|B^{\check{\alpha}}_{\mathbf{3}(\mathbf{1},1)_1\mv|_n}|, \  & %
			\mbox{if}\ \ \mathbf{t}(h_{\mv})=\mathbf{2}.
		\end{cases}
	\end{equation}
	
	Write $\Psi^{\ma}=\{\psi_n^{\ma}:n\geq1\}$. We have the following analog of \cite[Lemma 7]{Qu2016}:
	\begin{lemma}\label{Psi-add} 
		$\Psi^{\ma}\in\mathcal{F}^-(\Omega _{\mathbf{a}},T_{\ma})$ with $D(\Psi^{\ma})=0$. Moreover, for any $\mathbf{v}\in\Omega _{\mathbf{a}} $, we have 
		\begin{equation}
			S_nf(\mv)\log\tau _{2}-3n\sum_{i=1}^{k}\log
			a_{i}-\log \tau _{2}\leq \psi^{\ma}_{n}(\mv)\leq 
			S_nf(\mv)\log\tau_{1}+\log 4,  \label{lem31}
		\end{equation}%
		where $\tau_{1}=(\lambda -8)/3,\ \tau
		_{2}=2(\lambda +5)$, $S_nf(\mv):=\sum_{i=0}^{n-1}f(T_{\ma}^{i}\mv)$ and 
		\begin{equation}\label{ff}
			f(\mv):=-k+\sum\limits_{j=1}^{k}(2-a_{j})\chi _{\{\mathbf{2}\}}(%
			\mathbf{t}(v_{j})) \ \ \text{if}\ \mv |_{1}=v_{1}v_{2}\cdots v_{k}\in\mathscr{A}_{\ma}.
		\end{equation}
	\end{lemma}
	
	\begin{proof}
		By \eqref{def-psi}, it is seen that $\Psi^{\ma}$ has bounded variation property with $D(\Psi^{\ma})=0$.
		
		For any $\mv\in\Omega_{\ma},n\in\N$, we write 
		$\mathbf{w}=T_{\mathbf{a}}^n(\mv)$. 
		Then for any $m\in\N$, we have
		\begin{equation*}
			\begin{cases}
				\varpi^{\mv}\mv|_{n}=\varpi^{\mv}\mv_{1}\mv_{2}\cdots\mv_{n}=\varpi^{\mv}v_1v_2\cdots v_{nk-1}v_{nk}\in\Omega^{\check{\alpha}}_{nk+1},\\
				\varpi^{\mathbf{w}}\mathbf{w}|_{m}=\varpi^{\mathbf{w}}\mv_{n+1}\cdots\mv_{n+m}=\varpi^{\mathbf{w}}v_{nk+1}\cdots v_{(n+m)k-1}v_{(n+m)k}\in\Omega^{\check{\alpha}}_{mk+1},\\
				\varpi^{\mv}\mv|_{n+m}=\varpi^{\mv}v_1v_2\cdots v_{nk}v_{nk+1}\cdots v_{(n+m)k-1}v_{(n+m)k}\in\Omega^{\check{\alpha}}_{(n+m)k+1}.
			\end{cases}
		\end{equation*}
		This combines with Proposition \ref{bco-cor},  we have (since $|B^{\check{\alpha}}_{\varpi^{\mathbf{w}}}|\in\{|B^{\check{\alpha}}_{\mathbf{1}(\mathbf{2},1)_{1}}|,|B^{\check{\alpha}}_{\mathbf{3}(\mathbf{1},1)_{1}}|\}$)
		$$ \frac{|B^{\check{\alpha}}_{\varpi^{\mv}\mv|_{n+m}}|}
		{|B^{\check{\alpha}}_{\varpi^{\mv}\mv|_n}|}\sim
		\frac{|B_{\varpi^{\mathbf{w}}\mathbf{w}|_m}^{\check{\alpha}}|}{|B_{\varpi^{\mathbf{w}}}^{\check{\alpha}}|}\sim|B_{\varpi^{\mathbf{w}}\mathbf{w}|_m}^{\check{\alpha}}|,$$
		where the constants related to ``$\sim$" depend on $\lambda$. Taking logarithms on both sides of the above equation and applying \eqref{def-psi}, we obtain that $\Psi^{\ma}$ is almost-additive. Also, since $\psi^{\ma}_n(\mv)\leq-nk\log2$ by Lemma \ref{esti-band-length}, and hence we conclude that $\Psi^{\ma}\in \mathcal F^-(\Omega_{\ma},T_{\ma}).$
		
		Fix any $\mv\in\Omega_{\ma}$ and $n\in\N$. By Proposition \ref{Omega-Sigma} and \eqref{def-iota0}, we see that 
		$$v=\iota(\mv)|_{nk+1}=\varpi^{\mv}v_1v_2\cdots v_{nk-1}v_{nk}\in \Omega^{\check{\alpha}}_{nk+1}.$$ 
		Since $\check{\alpha}=[1,\overline{a_1,a_2,\cdots,a_k}]$, then by Lemma \ref{esti-band-length} and \eqref{def-iota0}, we have
		\begin{equation}
			\tau _{2}^{-nk-1}\prod_{i=1}^{k}a_{i}^{-3n}\cdot \Big(\prod_{\substack{1\leq i\leq kn;\\
					{\mathbf{t}}(v_{i})=\mathbf{2}}}\tau_{2}^{2-a_{i}}\Big)\leq |B_{v}^{\check{\alpha}}|\leq 4\tau _{1}^{-nk}\Big(
			\prod_{\substack{1\leq i\leq kn;\\
					{\mathbf{t}}(v_{i})=\mathbf{2}}}\tau _{1}^{2-a_{i}}\Big). \label{31-1}
		\end{equation}%
		It follows from (\ref{31-1}) that 
		\begin{align*}
			& -\log\tau_2+\Big( -nk+\sum\limits_{i=0}^{n-1}\sum\limits_{j=1}^{k}(2-a_{j})\chi _{\{
				\mathbf{2}\}}(\mathbf{t}(v_{ki+j}))\Big) \log \tau
			_{2}-3n\sum_{i=1}^{k}\log a_{i}\leq \psi^{\ma}_{n}(\mv) \\
			& \leq \Big( -nk+\sum\limits_{i=0}^{n-1}\sum\limits_{j=1}^{k}(2-a_{j})\chi _{\{
				\mathbf{2}\}}(\mathbf{t}(v_{ki+j}))\Big) \log \tau _{1}+\log 4,
		\end{align*}%
		thus showing (\ref{lem31}).
	\end{proof}
	\subsubsection{The weak-Gibbs metric and bi-Lipschitz homeomorphism}
	Since $\Psi^{\ma}\in\mathcal{F}^-(\Omega_{\ma},T_{\ma})$, then we can define a weak-Gibbs metric $d_{\ma}$ on $\Omega_{\ma}$ according to \eqref{weak-gibbs-metric} as follows. For any $\mv,\tilde{\mv}\in\Omega_{\ma}$, define 
	\begin{equation*}
		d_{\ma}(\mv,\tilde{\mv}):=\sup_{\mathbf{w}\in[\mv\wedge \tilde{\mv}]_{\ma}}\exp(\psi^{\ma}_{|\mv\wedge\tilde{\mv}|}(\mathbf{w})). \label{wd}
	\end{equation*}
	
	Recall that $\pi_{\ma,\lambda}:
	\Omega_{\ma}\to\Sigma_{\check{\alpha},\lambda}$ is a coding map. Moreover, we have the following:
	\begin{proposition}\label{bi-Lip} 
		$\pi_{\ma,\lambda}:(\Omega_{\mathbf{a}%
		},d_{\ma})\to(\Sigma_{\check{\alpha},\lambda},|\cdot|)$ is a bi-Lipschitz
		homeomorphism.
	\end{proposition}	
	\begin{proof}
		Given $\mv,\tilde{\mv}\in\Omega_{\ma}$
		and assume $\mv|_n=\tilde{\mv}|_n$ with $\mv_{n+1}\neq%
		\tilde{\mv}_{n+1}$ for some $n\in\N$. In this case, by Proposition \ref{Omega-Sigma} and \eqref{def-iota0}, we have that $nk+1\leq|\iota(\mv)\wedge\iota(\tilde{\mv})|<(n+1)k+1$, and hence
		$$d_{\ma}(\mv,\tilde{\mathbf{v
		}})=|B^{\check{\alpha}}_{v}|,\quad\text{where}\ v=\iota(\mv)|_{nk+1}\in\Omega^{\check{\alpha}}_{nk+1}.$$
		If we write $\iota(\mv)\wedge\iota(\tilde{\mv})=vw$ for some $w=e_1e_2\cdots e_s\in(\prod_{i=1}^{s}\mathscr{A}_{a_i})\cup\{\emptyset\}$\footnote{If $w=\emptyset$, then $\iota(\mv)\wedge\iota(\tilde{\mv})=v$.} with $0<s<k$. By Proposition \ref{bco-cor}, then 
		\begin{equation*}
			\eta^{-1} \frac{|{B}_{\mathbf{t}w}^{\alpha}|}{|{B}_{\mathbf{t}}^{\alpha}|}%
			\le \frac{|B_{vw}^{\check{\alpha}}|}{|B_{v}^{\check{\alpha}}|%
			}\le \eta \frac{|{B}_{\mathbf{t}w}^{\alpha}|}{|{B}_{\mathbf{t}}^{\alpha}|},\ \ \text{where}\ \mathbf{t}=\mathbf{t}(v).
		\end{equation*}
		Together with Lemma \ref{esti-band-length}, we see that $(M=\max\{a_1,a_2,\cdots,a_k\})$
		\begin{equation}\label{bi-lip}
			|B^{\check{\alpha}}_v|\geq|B_{vw}^{\check{\alpha}}|\geq\eta^{-1}|B_{v}^{\check{\alpha}}|\frac{\min\{|B^{\alpha}_{\mathbf{t}w}|:\mathbf{t}\to w,\ \mathbf{t}\in\{\mathbf{1},\mathbf{2},\mathbf{3}\}\}}{\max\{|B^{\alpha}_{\mathbf{t}}|:\mathbf{t}\in\{\mathbf{1},\mathbf{2},\mathbf{3}\}\}}\geq\frac{|B_{v}^{\check{\alpha}}|}{4\eta(\tau_2^MM^3)^k}.
		\end{equation}
		
		By Proposition \ref{Omega-Sigma} and \eqref{bi-lip}, the map $%
		\iota:\Omega_{\ma}\to\Omega^{\check{\alpha}}$ is a bijection and 
		\begin{equation*}
			d_{\ma}(\mv,\tilde{\mv})=|B^{\check{\alpha}}_{v}|\sim|B^{\check{\alpha}}_{vw}|=d_{\check{\alpha}}(\iota(\mv),\iota(\tilde{\mv})).
		\end{equation*}
		This implies that $\iota:(\Omega_{\ma},d_{\ma})\to(\Omega^{\check{%
				\alpha}},d_{\check{\alpha}})$ is a bi-Lipschitz homeomorphism. By Proposition \ref{bi-lip-alpha}, $\pi_{\lambda}^{\check{\alpha}}:(\Omega^{\check{\alpha}},d_{\check{\alpha}})\to(\Sigma_{\check{%
				\alpha},\lambda},|\cdot|)$ is a bi-Lipschitz homeomorphism. Thus $%
		\pi_{\ma,\lambda}=\pi^{\check{\alpha}}_{\lambda}\circ\iota$ is also a
		bi-Lipschitz homeomorphism.
	\end{proof}
	
	By leveraging the bi-Lipschitz homeomorphism $\pi_{\ma,\lambda}$, it can be established that $\mathcal{N}_{\check{\alpha},\lambda}$ is exact-dimensional.
	\begin{proposition}\label{mu-N}
		Let $\mu_{\ma}$ be the Gibbs measure related to $\mathbf{0}$. Then the following hold:
		
		(1) $\mu_{\mathbf{a}
		}\circ\pi_{\ma,\lambda}^{-1}\asymp\mathcal{N}_{\check{\alpha},\lambda}$, and hence we have
		\begin{equation*}
			\underline{d}_{\mu_{\ma}}(\mv)=\underline{d}_{\mathcal{N}_{\check{\alpha},\lambda}}(\pi_{\ma,\lambda}(\mv))\ \ \ \ \text{and} \ \ \ \ \overline{d}_{\mu_{\ma}}(\mv)=\overline{d}_{\mathcal{N}_{\check{\alpha},\lambda}}(\pi_{\ma,\lambda}(\mv)).
		\end{equation*}	
		
		(2) The DOS $\mathcal{N}_{\check{\alpha},\lambda}$ is exact-dimensional and 
		\begin{equation*}\label{loc-dim}
			\dim_H\mathcal{N}_{\check{\alpha},\lambda}=-\frac{h_{top}(T_{\ma})}{\Psi^{\ma}_*(\mu_{\ma})}=-\frac{\log E_{\ma}}{\Psi^{\ma}_*(\mu_{\ma})}. 
		\end{equation*}
	\end{proposition}
	\begin{proof}
		(1)	By Lemma \ref{primitive}, $(\Omega_{\ma},T_{\ma})$ is a topologically mixing subshift and the incidence matrix $A_{\ma}$ has Perron-Frobenius eigenvalue $E_{\ma}$, therefore $h_{top}(T_{\ma})=\log E_{\ma}$ (see for example \cite{Wa1982}). By Corollary \ref{pressure-continuous} (2), $\mu_{\ma}$ is the maximal entropy measure such that 
		\begin{equation}\label{htop-E-a}
			\mu_{\ma}([\mv|_n]_{\ma})\sim \exp(-nh_{top}(T_{\ma}))\sim E_{\ma}^{-n},\ \ \ \forall\ \mv\in\Omega_{\ma},\ n\in\N.
		\end{equation}
		By Lemmas \ref{dos-measure} and \ref{primitive}, for any $w\in\Omega^{\check{\alpha}}_{nk+1}$, we see that
		\begin{equation}\label{nu=N}
			\mathcal{N}_{\check{\alpha},\lambda}(X_{w}^{\check{\alpha}})=\mathcal{N}_{\check{\alpha},\lambda}(B^{\check{\alpha}}_{w})\sim q_{nk+1}^{-1}(\check{\alpha})\sim E_{\ma}^{-n}.
		\end{equation}
		Write $\nu_{\mathbf{a}}:=\mu_{\mathbf{a}%
		}\circ\pi_{\ma,\lambda}^{-1}$. Assume $\mv\in\Omega_{\mathbf{a}}$, then $v=\iota(\mv)\in\Omega^{\check{\alpha}}$ by Proposition \ref{Omega-Sigma}. Using  \eqref{cylinder-a} and \eqref{def-iota0}, for any $n\in\N$, then $w:=v|_{nk+1}\in\Omega^{\check{\alpha}}_{nk+1}$ and  $\pi_{\ma,\lambda}([\mv|_n]_{\mathbf{a}})=X^{\check{\alpha}%
		}_{w}$. Combining \eqref{htop-E-a} and \eqref{nu=N}, 
		\begin{equation*}
			\nu_{\mathbf{a}}(X_{\varpi^{\mv}\mv|_n}^{\check{\alpha}})=\mu_{\mathbf{a}%
			}(\pi_{\ma,\lambda}^{-1}(X_{w}^{\check{\alpha}}))=\mu_{%
				\mathbf{a}}([\mv|_n]_{\mathbf{a}})\sim E_{\mathbf{a}}^{-n}\sim\mathcal{N}_{\check{%
					\alpha},\lambda}(X_{w}^{\check{\alpha}}).
		\end{equation*}
		Since $\{X_{\varpi^{\mv}\mv|_n}^{\check{\alpha}}:\mv\in\Omega_{\ma,n},n\in\N\}$ generates the Borel $\sigma$-algebra of $\Sigma_{\check{\alpha},\lambda}$, then $\nu_{\ma}\asymp\mathcal{N}_{\check{\alpha},\lambda}$. Now by \eqref{def-loc-dim} and Proposition \ref{bi-Lip}, we conclude that
		\begin{equation*}
			\underline{d}_{\mu_{\ma}}(\mv)=\underline{d}_{\mathcal{N}_{\check{\alpha},\lambda}}(\pi_{\ma,\lambda}(\mv))\ \ \ \ \text{and} \ \ \ \ \overline{d}_{\mu_{\ma}}(\mv)=\overline{d}_{\mathcal{N}_{\check{\alpha},\lambda}}(\pi_{\ma,\lambda}(\mv)).
		\end{equation*}	
		
		(2)	By Proposition \ref{bi-Lip} and $\mu_{\mathbf{a}}\circ\pi_{\ma,\lambda}^{-1}\asymp\mathcal{N}_{\check{\alpha},\lambda}$, then $\mathcal{N}_{\check{\alpha},\lambda}$ is exact-dimensional. Moreover, by Lemma \ref{mu-phi-loc}, we obtian 
		\begin{equation*}
			\dim_H\mathcal{N}_{\check{\alpha},\lambda}=\dim_H\mu_{\ma}=-\frac{h_{top}(T_{\ma})}{\Psi^{\ma}_*(\mu_{\ma})}=-\frac{\log E_{\ma}}{\Psi^{\ma}_*(\mu_{\ma})}.
		\end{equation*}	
		Now we finish the proof of this proposition.
	\end{proof}
	\subsection{The pressure function $\p_{\ma}$}\
	
	By Lemma \ref{Psi-add}, $\Psi^{\ma}\in\mathcal{F}^{-}(\Omega_{\ma},T_{\ma})$. For any $s\in\R$, define the pressure $\p_{\ma}(s)$ of $s\Psi^{\mathbf{a}}$ as
	\begin{equation}\label{def-pressure}
		\p_{\ma}(s):=\lim\limits_{n\to\infty}\frac{1}{n}\log\sum_{|\mv|=n}\exp(\sup_{x\in[\mv]_{\ma}}s\psi^{\mathbf{a}}_n(x)).
	\end{equation}
	
	\begin{proposition}\label{trans-band} 
		For any $s\in\R$, the potential $s\Psi^{\ma}$ admits a Gibbs measure $\mu^{\ma}_{s}$ with $\mu^{\ma}_{0}=\mu_{\ma}$. The pressure $\p_{\ma}$ is $C^1$ on $\R$. Moreover, $\p'_{\ma}(s)=\Psi^{\ma}_*(\mu^{\ma}_{s})$ for all $s\in\R$ and 
		\begin{equation*}
			\begin{cases}
				\p'_{\ma}(-\infty):=\lim\limits_{s\to-\infty}\p'_{\ma}(s)=k\cdot\lim\limits_{n\to\infty}\frac{1}{n}\log|B^{\check{\alpha}}_{n,\min}|,\\ \p'_{\ma}(\infty):=\lim\limits_{s\to\infty}\p'_{\ma}(s)=k\cdot\lim\limits_{n\to\infty}\frac{1}{n}\log|B^{\check{\alpha}}_{n,\max}|,
			\end{cases}
		\end{equation*}
		where $|B^{\alpha}_{n,\min}|=\min\{|B^{\alpha}_w|:w\in\Omega^{\alpha}_n\}$ is the shortest spectral generating band of order $n$.
	\end{proposition}
	\begin{proof}
		For any $s\in\R$, since $s\Psi^{\ma}\in\mathcal{F}^-(\Omega_{\ma},T_{\ma})$, then by Theorem \ref{variation priciple}, $s\Psi^{\ma}$ admits a Gibbs measure $\mu^{\ma}_{s}$. If $s=0$, for any $\mv\in\Omega_{\ma}$ and $n\in\N$, by Corolalry \ref{pressure-continuous} (2), we have
		\begin{equation*}
			\mu^{\ma}_{0}([\mv|_n]_{\ma})\sim\exp(-n\p_{\ma}(0))\sim\exp(-nh_{top}(T_{\ma})).
		\end{equation*}
		Combining with \eqref{htop-E-a}, we have $\mu^{\ma}_{0}\asymp\mu_{\ma}$. Since $\mu^{\ma}_{0},\mu_{\ma}$ are ergodic measures, then $\mu^{\ma}_{0}=\mu_{\ma}$.
		
		By Proposition \ref{relatived-pressure-diff} (2), $\p_{\ma}$ is $C^1$ and convex on $\R$ and $\p'_{\ma}(s)=\Psi^{\ma}_*(\mu^{\ma}_{s})$ for all $s\in\R$. Moreover, the limits $\lim\limits_{s\to-\infty}\p'_{\ma}(s)$ and $ \lim\limits_{s\to\infty}\p'_{\ma}(s)$ both exist by Corollary \ref{psi-pressure}.
		
		By Proposition \ref{Omega-Sigma}, we have that $\iota(\Omega_{\ma,m})=\Omega^{\check{\alpha}}_{mk+1}$. Then by \eqref{def-psi}, for any $m\in\N$, 
		\begin{equation}\label{band-1}
			\begin{cases} 
				\exp(\inf\limits_{\mv\in\Omega_{\ma}}\psi^{\ma}_m(\mv))=|B^{\check{\alpha}}_{mk+1,\min}|=\inf\limits_{w\in\Omega^{\check{\alpha}}_{mk+1}}|B^{\check{\alpha}}_w|,\\ \exp(\sup\limits_{\mv\in\Omega_{\ma}}\psi^{\ma}_m(\mv))=|B^{\check{\alpha}}_{mk+1,\max}|=\sup\limits\limits_{w\in\Omega^{\check{\alpha}}_{mk+1}}|B^{\check{\alpha}}_w|.
			\end{cases}
		\end{equation}
		For any large $n$, assume $m=m(n)$ is such that $(m-1)k+1\leq n< mk+1$, then 
		\begin{equation*}
			\left\{ \begin{aligned}
				&|B^{\check{\alpha}}_{mk+1,\min}|\leq|B^{\check{\alpha}}_{n,\min}|\leq|B^{\check{\alpha}}_{(m-1)k+1,\min}|,\\
				&|B^{\check{\alpha}}_{mk+1,\max}|\leq|B^{\check{\alpha}}_{n,\max}|\leq|B^{\check{\alpha}}_{(m-1)k+1,\max}|.
			\end{aligned}\right.
		\end{equation*}
		Combining \eqref{band-1} and Corollary \ref{psi-pressure}, we obtain the results.
	\end{proof}
	Now we can obtain the following asymptotic properties:
	\begin{corollary}\label{asy-lemma} 
		For any $k\in\N,\mathbf{a}\in \N^k$ and $s\in\R$, we have 
		\begin{equation}\label{asy}
			\lim\limits_{\lambda \rightarrow \infty }\frac{\p'_{\ma}(s)}{\log
				\lambda }=\int f\mathrm{d}\mu^{\ma}_{s},\ \ \lim\limits_{\lambda \rightarrow \infty }\frac{\p'_{\ma}(-\infty)}{\log \lambda }=
			\underline{\mathbf{F}_{\ma}},\ \  \ \lim\limits_{\lambda \rightarrow \infty }\frac{\p'_{\ma}(\infty)}{\log \lambda }=\overline{\mathbf{F}_{\ma}}.
		\end{equation}
		where $f$ is defined in \eqref{ff} and  \begin{equation}  \label{def-F}
			\underline{\mathbf{F}_{\ma}}=\lim_{n\to\infty}\frac{1}{n}\inf_{\mv\in\Omega_{%
					\mathbf{a}}}S_nf(\mv);\ \ \ \ \ \ \ \overline{\mathbf{F}_{\ma}}=\lim_{n\to\infty}%
			\frac{1}{n}\sup_{\mv\in\Omega_{\ma}}S_nf(\mv).
		\end{equation}
	\end{corollary}
	
	\begin{proof}
		Note that $\mu^{\ma}_s$ is an invariant measure for any $s\in\R$, then by Lemma \ref{Psi-add}, we have 
		\begin{align*}
			&\log\tau_2\lim\limits_{n\to\infty}\frac{1}{n}\int S_nf\mathrm{d}\mu^{\ma}_{s}-3\sum_{i=1}^{k}\log a_i \leq\lim\limits_{n\to\infty}\frac{1}{n}\int\psi^{\ma}_n\mathrm{d}\mu^{\ma}_{s}\leq\log%
			\tau_1\lim\limits_{n\to\infty}\frac{1}{n}\int S_nf\mathrm{d}\mu^{\ma}_{s},
		\end{align*}
		where $\tau_{1}=(\lambda -8)/3$ and $\tau
		_{2}=2(\lambda +5)$. Combining this with Proposition \ref{trans-band}, we obtain 
		\begin{equation*}
			\lim\limits_{\lambda \rightarrow \infty }\frac{\p'_{\ma}(s)}{\log
				\lambda }=\lim\limits_{\lambda\to\infty}\frac{\Psi^{\ma}_*(\mu^{\ma}_{s})}{\log\lambda}=\int
			f\mathrm{d}\mu^{\ma}_{s}.
		\end{equation*}	
		
		Note that $S_{n+m}f(\mv)=S_{n}f(\mv)+S_{m}f(T^n\mv)$, then
		\begin{align*}
			\inf_{\mv\in\Omega_{\ma}}S_{n+m}f(\mv)\geq\inf_{\mv\in\Omega_{\ma}}S_nf(\mv)+\inf_{\mv \in\Omega_{\ma}}S_{m}f(\mv), \\
			\sup_{\mv\in\Omega_{\ma}}S_{n+m}f(\mv)\leq\sup_{\mv\in\Omega_{\ma}}S_nf(\mv)+\sup_{\mv\in\Omega_{\mathbf{%
						a}}}S_{m}f(\mv),
		\end{align*}
		which implies that the limits in \eqref{def-F} exist. Since $\inf\limits_{\mv \in\Omega_{\ma}}\psi^{\ma}_n(\mv)=|B^{\check{\alpha}}_{nk+1,\min}|$, using Lemma \ref{Psi-add} again, we see that 
		\begin{align*}
			\log\tau_2\lim_{n\to\infty}\frac{1}{n}\inf_{\mv\in\Omega_{%
					\mathbf{a}}}S_nf(\mv)-3\sum_{i=1}^{k}\log a_i
			\leq\lim_{n\to\infty}\frac{1}{n}|B^{\check{\alpha}}_{nk+1,\min}|\leq\log\tau_1\lim_{n\to\infty}\frac{1%
			}{n}\inf_{\mv\in\Omega_{\ma}}S_nf(\mv).
		\end{align*}
		Together with Proposition \ref{trans-band} and \eqref{def-F}, we conclude that 
		\begin{equation*}
			\lim\limits_{\lambda\to\infty}\frac{\p'_{\ma}(-\infty)}{\log\lambda}=\underline{\mathbf{F}_{\ma}}.
		\end{equation*}
		
		A	similar proof shows that the last equality in \eqref{asy} holds. Now the result follows.
	\end{proof}
	\begin{remark}\label{asy-constant}
		For any $\ma=a_1a_2\in\N^2$, one can obtain that
		\begin{equation*}
			\underline{\mathbf{F}_{\ma}}=\begin{cases}
				-4/3,&\text{if}\ \max\{a_1,a_2\}=1,\\
				-\max\{a_1,a_2\},&\text{if}\ \max\{a_1,a_2\}>1,
			\end{cases}\ \ 
			\overline{\mathbf{F}_{\ma}}=\begin{cases}
				-1,&\text{if}\ \min\{a_1,a_2\}=1,\\
				-2, &\text{if}\ \min\{a_1,a_2\}>1.
			\end{cases}
		\end{equation*}
	\end{remark}
	\section{Proofs of Theorems \ref{main-result}}\label{sec-main}
	
	In this section, let $k\in\N,\ma=a_1a_2\cdots a_k\in\N^k$ and write $
	\check{\alpha}=[1,\overline{a_1,a_2,\cdots,a_k}].$
	Recall that $(\Omega_{\ma},T_{\ma})$ is a topologically mixing subshift of finite type and $h_{top}(T_{\ma})=\log E_{\ma}$ is the topological entropy of $(\Omega_{\ma},T_{\ma})$. We define the potential $\Psi^{\ma}=\{\psi^{\ma}_n:n\geq1\}$ according to \eqref{def-psi}. Let $\mu^{\ma}_s$ be the Gibbs measure related to the potential $s\Psi^{\ma}$ on $\Omega_{\ma}$. Let $\p_{\ma}(s)$ be the pressure function of the potential $s\Psi^{\ma}$ (see \eqref{def-pressure}).
	\subsection{Proof of Theorem \ref{main-result} (vii)}\
	
	Fix $\lambda>20$. Let $\alpha,\beta\in\mathcal{EP}(\mathbf{a})$ with continued fraction expansions 
	$$\alpha=[b_1,\cdots,b_m,\overline{a_1,\cdots,a_k}],\quad
	\beta=[c_1,\cdots,c_n,\overline{a_1,\cdots,a_k}].$$ Recall that $X^{\alpha}_{w}$ is the basic set of $\Sigma_{\alpha,\lambda}$ (see \eqref{def-basic-set}), then 
	\begin{equation}\label{T-1}
		\Sigma_{\alpha,\lambda}=\bigsqcup_{u\in\Omega^{\alpha}_{m+5}}X_{u}^{\alpha} \ \ \text{and}\ \ \
		\Sigma_{\beta,\lambda}=\bigsqcup_{v\in\Omega^{\beta}_{n+5}}X_{v}^{\beta}.
	\end{equation}
	Moreover, by \eqref{admissible-T-A} and \eqref{admissible-A-A}, we have
	\begin{equation}\label{T-2}
		\{\mathbf{t}(u):u\in\Omega^{\alpha}_{m+5}\}= \{\mathbf{t}(v):v\in\Omega^{\beta}_{n+5}\}=\{\mathbf{1},\mathbf{2},\mathbf{3}\}.
	\end{equation}
	We shall prove the inequalities of Theorem \ref{main-result} (vii) one by one.
	
	By equations \eqref{T-1}, \eqref{T-2} and Proposition \ref{geo-lem-tail}, we conclude that
	\begin{align*}
		\inf\{\underline{d}_{\mathcal{N}_{\alpha,\lambda}}(x):x\in\Sigma_{\alpha,\lambda}\}&=\inf_{u\in\Omega^{\alpha}_{m+5}}\inf\{\underline{d}_{\mathcal{N}_{\alpha,\lambda}}(x):x\in X^{\alpha}_u\}\\
		&=\inf_{v\in\Omega^{\beta}_{n+5}}\inf\{\underline{d}_{\mathcal{N}_{\beta,\lambda}}(x):x\in X^{\beta}_v\}=\inf\{\underline{d}_{\mathcal{N}_{\beta,\lambda}}(x):x\in\Sigma_{\beta,\lambda}\}.
	\end{align*}
	Now by \eqref{def-gamma}, we see that $\gamma(\alpha,\lambda)=\gamma(\beta,\lambda)$.
	
	Similarly, by equations \eqref{T-1}, \eqref{T-2} and Proposition \ref{geo-lem-tail}, we have
	\begin{align*}
		\dim_{H}\mathcal{N}_{\alpha,\lambda}&=\sup\{s: \underline{d}_{\mathcal{N}_{\alpha,\lambda}}(x)\ge s \text{ for  } \mathcal{N}_{\alpha,\lambda} \text{ a.e. }x\in\Sigma_{\alpha,\lambda}\}\\
		&=\inf_{u\in\Omega^{\alpha}_{m+5}}\sup\{s: \underline{d}_{\mathcal{N}_{\alpha,\lambda}}(x)\ge s \text{ for  } \mathcal{N}_{\alpha,\lambda} \text{ a.e. }x\in X^{\alpha}_{u}\}\\
		&=\inf_{v\in\Omega^{\beta}_{n+5}}\sup\{s: \underline{d}_{\mathcal{N}_{\beta,\lambda}}(x)\ge s \text{ for  } \mathcal{N}_{\beta,\lambda} \text{ a.e. }x\in X^{\beta}_{v}\}\\
		&=\sup\{s: \underline{d}_{\mathcal{N}_{\beta,\lambda}}(x)\ge s \text{ for  } \mathcal{N}_{\beta,\lambda} \text{ a.e. }x\in\Sigma_{\beta,\lambda}\}=\dim_{H}\mathcal{N}_{\beta,\lambda}.
	\end{align*}
	It follows from \eqref{dim-meas} that $d(\alpha,\lambda)=d(\beta,\lambda)$.
	
	By equations \eqref{T-1}, \eqref{T-2} and Proposition \ref{geo-lem-tail} (1), we get 
	\begin{align*}
		\dim_{H}\Sigma_{\alpha,\lambda}=\sup_{u\in\Omega^{\alpha}_{m+5}}\dim_HX^{\alpha}_{u}=\sup_{v\in\Omega^{\beta}_{n+5}}\dim_HX^{\beta}_{v}=\dim_{H}\Sigma_{\beta,\lambda}.
	\end{align*}
	This implies that $D(\alpha,\lambda)=D(\beta,\lambda)$.
	
	By the proof of Lemma \ref{primitive}, it is straightforward to check that
	\begin{equation}\label{5-0}
		\lim\limits_{n\to\infty}\frac{1}{n}\log q_n(\alpha)=\lim\limits_{n\to\infty}\frac{1}{n}\log q_n(\beta)=\frac{\log E_{\ma}}{k}.
	\end{equation}
	For any $l\in\N$, by equation \eqref{T-2}, Proposition \ref{bco-cor} and Lemma \ref{esti-band-length}, then
	\begin{align*}
		|B^{\alpha}_{m+5+l,\max}|&=\sup\{|B^{\alpha}_{uw}|:u\in\Omega^{\alpha}_{m+5},uw\in\Omega^{\alpha}_{m+5+l}\}\\
		&\sim\sup\{|B^{\beta}_{vw}|:v\in\Omega^{\beta}_{n+5},vw\in\Omega^{\beta}_{n+5+l}\}=|B^{\beta}_{n+5+l,\max}|\\
		&\sim\sup\{|B^{\beta}_{\tilde{u}w}|:\tilde{u}\in\Omega^{\check{\alpha}}_{1+5},\tilde{u}w\in\Omega^{\check{\alpha}}_{1+5+l}\}=|B^{\check{\alpha}}_{1+5+l,\max}|,
	\end{align*}
	where the constants related to ``$\sim$" are independent of $l$. Thus, we have the following:
	\begin{equation*}
		|B^{\alpha}_{m+5+l,\max}|\sim |B^{\beta}_{n+5+l,\max}|\sim |B^{\check{\alpha}}_{1+5+l,\max}|.
	\end{equation*}
	By Proposition \ref{trans-band}, the limit $\lim\limits_{l\to\infty}\frac{1}{l}\log|B^{\check{\alpha}}_{1+5+l,\max}|$ exists and equals to $\p'_{\ma}(\infty)/k$. Therefore, combining \eqref{5-0} and Proposition \ref{trans}, we conclude that 
	\begin{equation}\label{trans-1}
		\mathcal{T}^{\pm}(\alpha,\lambda)=\mathcal{T}^{\pm}(\beta,\lambda)=\mathcal{T}^{\pm}(\check{\alpha},\lambda)=-\frac{\log E_{\ma}}{\p'_{\ma}(\infty)}.
	\end{equation}
	\subsection{Proof of Theorem \ref{main-result} (i)-(iv)}\
	
	Fix $\lambda>20$. Note that $\mu_{\ma}=\mu^{\ma}_0$ is the maximal entropy measure such that $h_{\mu_{\ma}}(T_{\ma})=h_{top}(T_{\ma})=\p_{\ma}(0)$. By Proposition \ref{dos-mu-ana} (1), we have
	\begin{equation}\label{pressure-measure}
		\inf_{\mv\in\Omega_{\ma}}\underline{d}_{\mu_{\ma}}(\mv)=-\frac{\p_{\ma}(0)}{\p'_{\ma}(-\infty)}\ \ \text{and}\ \ \ \sup_{\mv\in\Omega_{\ma}}\overline{d}_{\mu_{\ma}}(\mv)=-\frac{\p_{\ma}(0)}{\p_{\ma}'(\infty)}.
	\end{equation}
	\begin{enumerate}[label=(\roman*)]
		\item By Proposition \ref{mu-N} (2), the DOS $\mathcal{N}_{\check{\alpha},\lambda}$ is exact-dimensional. By Proposition \ref{geo-lem-tail} (2) and Theorem \ref{main-result} (vii), we see that $\mathcal{N}_{\alpha,\lambda}$ is also exact-dimensional and 
		\begin{equation*}
			d(\alpha,\lambda)=d(\check{\alpha},\lambda)=-\frac{h_{top}(T_{\ma})}{\Psi^{\ma}_*(\mu_{\ma})}
=-\frac{\p_{\ma}(0)}{\Psi^{\ma}_*(\mu^{\ma}_{0})}.
		\end{equation*}
		Proposition \ref{trans-band} tells us that $\p'_{\ma}(0)=\Psi^{\ma}_*(\mu^{\ma}_0)$. Then the result follows. 
		
		\item By Proposition \ref{bi-Lip}, $\pi_{\ma,\lambda}:(\Omega_{\mathbf{a}%
		},d_{\ma})\to(\Sigma_{\check{\alpha},\lambda},|\cdot|)$ is a bi-Lipschitz
		homeomorphism. According to \eqref{pressure-measure} and Proposition \ref{mu-N} (1), we conclude that
		\begin{equation*}
			-\frac{\p_{\ma}(0)}{\p'_{\ma}(-\infty)}=\inf_{\mv\in\Omega_{\ma}}\underline{d}_{\mu_{\ma}}(\mv)=\inf_{\mv\in\Omega_{\ma}}\underline{d}_{\mathcal{N}_{\check{\alpha},\lambda}}(\pi_{\ma,\lambda}(\mv))=\inf\{\underline{d}_{\mathcal{N}_{\check{\alpha},\lambda}}(x):x\in\Sigma_{\check{\alpha},\lambda}\}=\gamma(\check{\alpha},\lambda).
		\end{equation*}
		We finish the proof of part (ii) by using the fact that  $\gamma(\check{\alpha},\lambda)=\gamma(\alpha,\lambda)$.
		\item  Since the map $\pi_{\ma,\lambda}:(\Omega_{\mathbf{a}%
		},d_{\ma})\to(\Sigma_{\check{\alpha},\lambda},|\cdot|)$ is a bi-Lipschitz
		homeomorphism, then by Corollary \ref{pressure-continuous} (3), $\dim_H\Omega_{\ma}=\dim_H\Sigma_{\check{\alpha},\lambda}=D(\check{\alpha},\lambda),
		$
		where $D(\check{\alpha},\lambda)$ is the zero of $\p_{\ma}(s)=0$. By Proposition \ref{trans-band}, there exists $s_{\ma}\in(0,D(\check{\alpha},%
		\lambda))$ such that  
		\begin{equation*} \label{d-D-2}
			D(\check{\alpha},\lambda)=-\frac{\p_{\ma}(0)-\p_{\ma}(D(\check{\alpha},\lambda))} {\p_{\ma}^{\prime }(s_{\ma})}=-\frac{\p_{\ma}(0)}{\p^{\prime}_{\ma}(s_{\ma})},
		\end{equation*}
		thus showing Theorem \ref{main-result} (iii) by noting that $D(\check{\alpha},\lambda)=D(\alpha,\lambda)$.
		
		\item The result follows by \eqref{trans-1}, \eqref{pressure-measure} and the fact $\log E_{\ma}=\p_{\ma}(0)$.
		
	\end{enumerate}
	\subsection{Proof of Theorem \ref{main-result} (v) and (vi)}\
	
	By Theorem \ref{main-result} (i)-(iv), we see that
	\begin{equation*}
		\gamma(\alpha,\lambda)=-\frac{\p_{\ma}(0)}{\p'_{\ma}(-\infty)},\ d(\alpha,\lambda)=-\frac{\p_{\ma}(0)}{\p'_{\ma}(0)},\ D(\alpha,\lambda)=-\frac{\p_{\ma}(0)}{\p'_{\ma}(s_{\ma})},\	\mathcal{T}^{\pm}(\alpha,\lambda)=-\frac{\p_{\ma}(0)}{\p'_{\ma}(\infty)}.
	\end{equation*}
	
	To establish strict inequalities for these spectral characteristics, we require the following:
	\begin{lemma}\label{band-band}
		Let $\lambda\geq240$, then for any $\alpha=[a_1,a_2,\cdots]\in\mathbb{I}$, we have 
		\begin{equation*}
			\sup_{n\geq1}\sup\{\log(|B^{\alpha}_{w}|/B^{\alpha}_{v}|):w,v\in\Omega^{\alpha}_{n}\}=\infty.
		\end{equation*}
	\end{lemma}
	The proof of lemma \ref{band-band} is quite technical, and will be given in Sect. \ref{imp-lemma}.
	\begin{enumerate}[label=(\roman*),start=5]
		\item Since $\ma=a_1a_2\cdots a_k$, then by Proposition \ref{Omega-Sigma}, for any $n\in\N$, we have $$\Omega_{\ma,n}=\Omega^{\check{\alpha}}_{nk+1},\quad\check{\alpha}=[1,\overline{a_1,a_2,\cdots,a_k}].$$ Now by Lemma \ref{band-band} and \eqref{def-psi}, we obtain that 
		\begin{equation*}
			\sup_{n\geq1}\sup\{|\psi^{\ma}_{n}(x)-\psi^{\ma}_{n}(y)|:x,y\in\Omega_{\ma}\}=\infty.
		\end{equation*}
		This combines with Proposition \ref{relatived-pressure-diff} (3), we conclude that $\p_{\ma}$ is strictly convex on $\R$ and $\p'_{\ma}$ is strictly increasing on $\R$. Note that $\p_{\ma}(0)>0$, we know that
		\begin{equation*}
			\p'_{\ma}(-\infty)<\p'_{\ma}(0)<\p'_{\ma}(s_{\ma})<\p'_{\ma}(\infty)<0,
		\end{equation*}
		which implies Theorem \ref{main-result} (v).
		
		\item 	By Corollary \ref{asy-lemma}, we have
		\begin{align*}
			&\lim_{\lambda\to\infty}\gamma(\alpha,\lambda)\cdot\log\lambda=-\frac{\p_{\ma}(0)}{\underline{\mathbf{F}_{\ma}}},\ \ \ \ \ \ \lim_{\lambda\to\infty}d(\alpha,\lambda)\cdot\log\lambda=-\frac{\p_{\ma}(0)}{\int f\mathrm{d}\mu^{\ma}_0},\\
			&\lim_{\lambda\to\infty}D(\alpha,\lambda)\cdot\log\lambda=-\frac{\p_{\ma}(0)}{\int f\mathrm{d}\mu^{\ma}_{s_{\ma}}}, \ \   \lim_{\lambda\to\infty}\mathcal{T}^{\pm}(\alpha,\lambda)\cdot\log\lambda=-\frac{\p_{\ma}(0)}{\overline{\mathbf{F}_{\ma}}}.
		\end{align*}
		This completes the proof of Theorem \ref{main-result} (vi).
	\end{enumerate}
	
	\section{Proof of Lemma \ref{band-band}}\label{imp-lemma}
	
	In this section, we prove Lemma \ref{band-band}. Roughly speaking, it suffices to construct two distinct words $w,v\in\Omega^{\alpha}_{n}$ such that the ratio of their corresponding spectral band lengths $|B^{\alpha}_{w}|/|B^{\alpha}_{v}|$ can be made arbitrarily large (if $n$ is large).
	
	The word $w\in\Omega^{\alpha}_{n}$ corresponds to relatively long spectral bands $B^{\alpha}_w$, and its construction essentially follows the approach given in \cite[Proposition 3.4]{DGLQ2015}. The only difference is that we employ more refined estimates, with the requirement $\lambda\geq240$.
	
	The	word $v=v_0v_1\cdots v_{n}\in\Omega^{\alpha}_{n}$ is associated with the relatively short spectral band $B^{\alpha}_v$, and its construction is rather straightforward. As for word $v_i=(\mathbf{t}_i,l_i)_{a_i}$, we set $\mathbf{t}_i=2$ whenever possible if $a_i\ge 2$; while if $a_i=1$, we choose $\mathbf{t}_i\ne 2$ as far as possible.
	
	The following lemma constitutes an improvement upon Lemma \ref{esti-band-length} for $\lambda\geq240$:
	\begin{lemma}\label{estimate-band}
		Assume that $\lambda\geq240$ and $\alpha=[a_1,a_2,\cdots]\in\mathbb{I}$. For any $w=w_0w_1\cdots w_n\in\Omega^{\alpha}_n$ with $w_i=(\mathbf{t}_i,l_i)_{a_i}$, then we have
		\begin{align*}
			\frac{			4}{t_2^n}\prod_{\mathbf{t}_i=\mathbf{2}}\frac{(2t_2)^{2-a_i}}{2}\prod_{\mathbf{t}_i\neq\mathbf{2}}\frac{\sin^2\frac{l_i\pi}{p(\mathbf{t}_{i-1}\mathbf{t}_i)+1}}{(p(\mathbf{t}_{i-1}\mathbf{t}_i)+1)}\leq|B^{\alpha}_w|\leq\frac{4}{t_1^n}\prod_{\mathbf{t}_i=\mathbf{2}}\frac{(2t_1)^{2-a_i}}{2}\prod_{\mathbf{t}_i\neq\mathbf{2}}\frac{\sin^2\frac{l_i\pi}{p(\mathbf{t}_{i-1}\mathbf{t}_i)+1}}{(p(\mathbf{t}_{i-1}\mathbf{t}_i)+1)},
		\end{align*}
		where $t_1=12(\lambda-8)/25$, $t_2=13(\lambda+8)/25$ and ($\mathbf{t}_0=\mathbf{t}(w_0)$)
		\begin{equation}
			p(\mathbf{t}_{i-1}\mathbf{t}_i)=\begin{cases}
				1,\quad&\text{if}\ \mathbf{t}_{i-1}\mathbf{t}_i=\mathbf{1}\mathbf{2},\\
				a_{i}+1,\quad&\text{if}\ \mathbf{t}_{i-1}\mathbf{t}_i=\mathbf{2}\mathbf{1},\\
				a_{i},\quad&\text{if}\ \mathbf{t}_{i-1}\mathbf{t}_i=\mathbf{2}\mathbf{3},\mathbf{3}\mathbf{1},\\
				a_{i}-1,\quad&\text{if}\ \mathbf{t}_{i-1}\mathbf{t}_i=\mathbf{3}\mathbf{3}.
			\end{cases} \label{def-p}
		\end{equation}
	\end{lemma}
	\begin{remark}
		The proof of this lemma follows exactly that of \cite[Propositions 3.2 and 3.3]{FLW2011}, except that we take the coupling constant $\lambda\geq240$, see also the proof of \cite[Lemma 3.7]{LQW2014}. Roughly speaking, when the $\lambda$ is sufficiently large, the polynomial generating the spectral band behaves very much like a straight line near the origin, then the coefficients of $\lambda$ in the parameters $t_1,t_2$ are both close to $1/2$. For the sake of completeness, we provide a proof sketch in Appendix \ref{exp}.
	\end{remark}
	
	\noindent{\bf Proof of Lemma \ref{band-band}.} 
	Fix $\lambda\geq240$. For any $\alpha=[a_1,a_2,\cdots]\in\I$, we discuss two cases:
	
	\medskip
	\noindent{\bf Case 1:} $a_i\equiv1$ for all $i\geq1$. In this case, fix large $n\in\N$, we construct two admissible words
	\begin{equation*}
		w=\mathbf{1}((\mathbf{2},1)_{1}(\mathbf{1},1)_1)^{3n},\ v=\mathbf{3}((\mathbf{1},1)_{1}(\mathbf{2},1)_1(\mathbf{3},1)_1)^{2n}\in\Omega^{\alpha}_{6n}.
	\end{equation*}
	Then by Lemma \ref{estimate-band}, we have 
	\begin{equation*}
		|B^{\alpha}_{w}|\geq\frac{4}{t_2^{6n}}t_2^{3n}\left(\frac{1}{4}\right)^{3n},\quad\ |B^{\alpha}_{v}|\leq\frac{4}{t_1^{6n}}\left(\frac{1}{2}\right)^{2n}
t_1^{2n}\left(\frac{1}{2}\right)^{2n}.
	\end{equation*}
	From this, we conclude that (since $t_1^4>4t_2^3$ if $\lambda\geq240$)
	\begin{equation*}
		\frac{|B^{\alpha}_{w}|}{|B^{\alpha}_{v}|}\geq
\left(\frac{t_1^4}{4t_2^3}\right)^{n}\to\infty\quad(n\to\infty).
	\end{equation*}
	
	\medskip
	\noindent{\bf Case 2:} $a_i>1$ for some $i\geq1$. In this case, for any $n\in\N$, we consider the integer sequence $a_1a_2\cdots a_n$. It may be divided by $1$'s into
	several segments that do not contain $1$'s,	i.e., we can write
	\begin{equation}
		a_1a_2\cdots a_n=A_11^{m_1}A_2\cdots A_s1^{m_s}A_{s+1}, \label{es7}
	\end{equation}
	where for any $i=1,\cdots, s+1$, $A_i=a_ma_{m+1}\cdots a_{m+l}$ for some $m>0,l\geq0,a_j>1$ for any $m\leq j\leq m+l$, and $m_i\geq1$. Note that if $a_1=1$, then $A_1=\emptyset$. 
	
	Now we can choose large $n\in\N$ such that $a_k>1$ for some $k<n$. Write $|A_i|=n_i$, in this case, we have that
	\begin{equation*}
		n=\sum_{i=1}^{s}m_i+\sum_{i=1}^{s+1}n_{i}\quad\text{and}\quad \sum_{i=1}^{s+1}n_{i}>0.
	\end{equation*}  
	Define $f(k):=\csc^2\frac{[k/2]}{k}\pi$. We will construct two admissible words $w,v\in\Omega^{\alpha}_{n}$ such that
	\begin{align}
		|B^{\alpha}_w|&\geq\frac{1}{t_2^n}\left(\frac{t_2}{4}\right)^{\sum_{i=1}^{s}[(m_i+1)/2]}
		\prod_{a_{i}\geq2}\frac{1}{a_if(a_i)},\label{low}\\
		|B^{\alpha}_v|&\leq\frac{4\sqrt3}{t_1^n}\left(\frac{t_1}{4}\right)^{\sum_{i=1}^{s}[(m_i+1)/3]}\left(\frac{1}{\sqrt{3}}\right)^{\sum_{i=1}^{s+1}n_i}
		\prod_{a_{i}\geq2}\frac{1}{a_if(a_i)}.\label{upp}
	\end{align}
	We show first that \eqref{low} and \eqref{upp} imply Lemma \ref{band-band}. If $\sum_{i=1}^{s}m_i=0$, then $s=0$ and $\sum_{i=1}^{s+1}n_i=n_1=n$. From this, we conclude that (note that $3^{1/6}t_1>6t_1/5>t_2$ if $\lambda\geq240$)
	\begin{align*}
		\frac{|B^{\alpha}_w|}{|B^{\alpha}_v|}&\geq \frac{(\sqrt{3})^{\sum_{i=1}^{s+1}n_i}}{4\sqrt3}\left(\frac{t_1}{t_2}\right)^n
		\left(\frac{t_2}{4}\right)^{\sum_{i=1}^{s}\left([\frac{m_i+1}{2}]-[\frac{m_i+1}{3}]\right)}
		\geq\frac{1}{4\sqrt3}\left(\frac{\sqrt{3}t_1}{t_2}\right)^n\to\infty\ (n\to\infty).
	\end{align*}
	If $\sum_{i=1}^{s}m_i>0$, we define $p:=\#\{1\leq i\leq s:m_{i}=2\}$, then we have
	\begin{equation*}
		n=\sum_{i=1}^{s}m_i+\sum_{i=1}^{s+1}n_{i}\geq\sum_{m_i=2}m_i+\sum_{m_{i}\neq2}m_i+s\geq 2p+(s-p)+s=2s+p\geq 3p.
	\end{equation*}
	So we conclude that (since $\lambda\geq240$, we have $t_2\geq4\cdot3^3=108$ and $3^{1/6}t_1>t_2$)
	\begin{align*}
		\frac{|B^{\alpha}_w|}{|B^{\alpha}_v|}&\geq\frac{1}{4\sqrt3}\left(\frac{t_1}{t_2}\right)^n\left(\frac{t_2}{4}\right)^{\sum_{i=1}^{s}\left([\frac{m_i+1}{2}]-[\frac{m_i+1}{3}]\right)}(\sqrt{3})^{\sum_{i=1}^{s+1}n_i}\\
		&=\frac{1}{4\sqrt3}\left(\frac{t_1}{t_2}\right)^n\left(\frac{t_2}{4}\right)^{\sum_{m_{i}\neq2}\left([\frac{m_i+1}{2}]-[\frac{m_i+1}{3}]\right)}(\sqrt{3})^{n-\sum_{m_i=2}m_i-\sum_{m_i\neq2}m_i}\\
		&\geq\frac{1}{4\sqrt3}\left(\frac{t_1}{t_2}\right)^n\left(\frac{t_2}{4}\right)^{\sum_{m_{i}\neq2}\frac{m_i}{6}}(\sqrt{3})^{n-2\cdot\frac{n}{3}-\sum_{m_i\neq2}m_i}\\
		&\geq\frac{1}{4\sqrt3}\left(\frac{t_1}{t_2}\right)^n\left(\frac{t_2}{4\cdot3^{3}}\right)^{\frac{1}{6}\sum_{m_{i}\neq2}m_i}3^{\frac{n}{6}}
		\geq\frac{1}{4\sqrt3}\left(\frac{3^{1/6}t_1}{t_2}\right)^n\to\infty\ (n\to\infty).
	\end{align*}
	
	Next, we construct admissible words $w,v\in\Omega^{\alpha}_{n}$ and show that equations \eqref{low} and \eqref{upp} hold. To show that equations \eqref{low} and \eqref{upp}, we need the following estimates, the proof of which is elementary
	and will be omitted.
	
	For any $k\geq2$, we have 
	\begin{align}
		&\frac{\sin^2\frac{[(k+1)/2]}{k+1}\pi}{k+1}\geq\frac{1}{2kf(k)},\quad \frac{\sin^2\frac{[(k+2)/2]}{k+2}\pi}{k+2}\geq\frac{1}{2kf(k)},\label{imp-1}\\
		&\frac{\sin^2\frac{\pi}{k+1}}{k+1}\leq\frac{1}{2kf(k)},\quad \frac{\sin^2\frac{\pi}{k+2}}{k+2}\leq\frac{1}{3kf(k)},\quad\frac{1}{2\cdot(2t_1)^{k-2}}\leq\frac{1}{kf(k)}.\label{imp-2}
	\end{align}
	
	The construction of the admissible words $w,v$ is carried out in two steps.
	
	\noindent{\bf Step 1:} We construct the admissible word $w\in\Omega^{\alpha}_{n}$ such that \eqref{low} holds. 
	
	Write $\tau_j=\sum_{i=1}^{j}(n_i+m_i).$ At first we define $\mathbf{t}_i,i =1,2,\cdots,\tau_s$ by induction. For any $j=1,2,\cdots,\tau_1$, we discuss two cases:
	
	{\rm Case 1:} $n_1=0$. Define
	\begin{equation*}
		\mathbf{t}_1\cdots \mathbf{t}_{m_1}=\begin{cases}
			\mathbf{2}(\mathbf{1}\mathbf{2})^{(m_1-1)/2},&\text{if $m_1$ is odd},\\
			(\mathbf{1}\mathbf{2})^{m_1/2},&\text{if $m_1$ is even}.
		\end{cases}
	\end{equation*}
	
	{\rm Case 2:} $n_1>0$. Define
	\begin{equation*}
		\mathbf{t}_1\cdots \mathbf{t}_{n_1}\cdot\mathbf{t}_{n_1+1}\cdots\mathbf{t}_{\tau_1}=\begin{cases}
			\mathbf{3}^{n_{1}-1}\mathbf{1}\cdot\mathbf{2}(\mathbf{1}\mathbf{2})^{(m_1-1)/2},&\text{if $m_1$ is odd},\\
			\mathbf{3}^{n_1}\cdot(\mathbf{1}\mathbf{2})^{m_1/2},&\text{if $m_1$ is even}.
		\end{cases}
	\end{equation*}
	
	Assume $\mathbf{t}_j$ is already defined for $j\leq\tau_{i-1}$. Now define $\mathbf{t}_j$, $j=\tau_{i-1}+1,\cdots,\tau_i$ as follows:
	\begin{equation*}
		\mathbf{t}_{\tau_{i-1}+1}\cdots \mathbf{t}_{\tau_{i-1}+n_i}\cdots \mathbf{t}_{\tau_i}=\begin{cases}
			\mathbf{3}^{n_{i}-1}\mathbf{1}\cdot\mathbf{2}(\mathbf{1}\mathbf{2})^{(m_i-1)/2},&\text{if $m_i$ is odd},\\
			\mathbf{3}^{n_{i}}\cdot(\mathbf{1}\mathbf{2})^{m_i/2},&\text{if $m_i$ is even}.
		\end{cases}
	\end{equation*}
	Thus by induction we have defined $\mathbf{t}_i$ for $i=1,2,\cdots,\tau_s$. Finally, if $n_{s+1}>0$, define 
	\begin{equation*}
		\mathbf{t}_{\tau_s+1}\cdots\mathbf{t}_{\tau_s+n_{s+1}}=
		\mathbf{3}^{n_{i}}.
	\end{equation*}
	
	By the admissible rules (\eqref{admissible-T-A} and \eqref{admissible-A-A}), we define the admissible word $w\in\Omega^{\alpha}_{n}$ as
	\begin{equation*}
		w=\begin{cases}
			\mathbf{1}w_1\cdots w_{n-1}w_{n},\quad&\text{if}\ n_1=0\text{\ and}\ m_1\text{ is odd},\\
			\mathbf{3}w_1\cdots w_{n-1}w_{n},\quad&\text{otherwise}.
		\end{cases}
	\end{equation*} 
	where $w_i=(\mathbf{t}_i,[(p(\mathbf{t}_{i-1}\mathbf{t}_i)+1)/2])_{a_i}$.
	
	By Lemma \ref{estimate-band}, 
	\begin{equation}\label{bw}
		|B_w^\alpha|\ge \frac{4}{t_2^n}\prod_{\mathbf{t}_i=\mathbf{2}}\frac{(2t_2)^{2-a_i}}{2}
		\prod_{\mathbf{t}_i\neq\mathbf{2}}
		\frac{\sin^2\frac{[(p(\mathbf{t}_{i-1}\mathbf{t}_i)+1)/2]\pi}{p(\mathbf{t}_{i-1}\mathbf{t}_i)+1}}
		{p(\mathbf{t}_{i-1}\mathbf{t}_i)+1}.
	\end{equation}
	
	We now estimate the above factors according to the construction of $w$. 
	
	For each $i=1,2,...,s+1$, and for every index $\tau_{i-1}+1\leq j\leq \tau_i$ (we adopt $\tau_{0}=0$ and $\tau_{s+1}=\tau_{s}+n_{s+1}$ by convention) with $a_j\ge2$, by construction of $w$, we have $\mathbf{t}_j\neq\mathbf{2}$. At each such position, \eqref{bw} contributes
	\begin{equation*}
		\frac{\sin^2\frac{[(p(\mathbf{t}_{j-1}\mathbf{t}_j)+1)/2]\pi}{p(\mathbf{t}_{j-1}\mathbf{t}_j)+1}}
		{p(\mathbf{t}_{j-1}\mathbf{t}_j)+1}
		\begin{cases}
			=\frac1{a_jf(a_j)},\quad&\text{if}\ \mathbf{t}_{j-1}\mathbf{t}_j=\mathbf{3}\mathbf{3},\\
			\ge \frac1{2a_jf(a_j)} \ \ (\text{by \eqref{imp-1}}),\quad&\text{if}\ \mathbf{t}_{j-1}\mathbf{t}_j=\mathbf{2}\mathbf{3},\ \mathbf{3}\mathbf{1},\ \mathbf{2}\mathbf{1}.
		\end{cases}
	\end{equation*}
	Furthermore, the second case is possible only at $\tau_{i-1}+1$ and $\tau_{i-1}+n_i$ by construction of $w$. Hence, for theses positions, \eqref{bw} contributes at least
	\begin{equation}\label{es1}
		\frac{1}{4}\prod_{j=\tau_{i-1}+1}^{\tau_{i-1}+n_i}\frac1{a_jf(a_j)}.
	\end{equation}
	
	For the indices with $a_j=1$, the contribution occurs
	when $\mathbf{t}_{j-1}\mathbf{t}_j=\mathbf{1}\mathbf{2}$ and $\mathbf{t}_{j-1}\mathbf{t}_j=\mathbf{2}\mathbf{1}$. By construction, the number of such indices in $j=\tau_{i-1}+1,\cdots,\tau_i$ equals
	\begin{align*}
		\#\{\tau_{i-1}+1\leq j\leq \tau_i:\mathbf{t}(w_{j-1})\mathbf{t}(w_j)=\mathbf{1}\mathbf{2},a_{j}=1\}
		=\left[\frac{m_i+1}{2}\right] \ \text{and} \\
		\#\{\tau_{i-1}+1\leq j\leq \tau_i:\mathbf{t}(w_{j-1})\mathbf{t}(w_j)=\mathbf{2}\mathbf{1},a_{j}=1\}
		=\left[\frac{m_i+1}{2}\right]-1.
	\end{align*}
	For these positions, \eqref{bw} contributes
	\begin{equation}\label{es2}
		t_2^{\left[\frac{m_i+1}{2}\right]}\left(\frac{1}{4}\right)^{\left[\frac{m_i+1}{2}\right]-1}
		=4\left(\frac{t_2}{4}\right)^{\left[\frac{m_i+1}{2}\right]}
	\end{equation}
	by noting that $(2t_2)^{2-a_j}/2=t_2$ with $a_{j}=1$ and 
	\begin{equation*}
		\frac{\sin^2\frac{[(p(\mathbf{t}_{j-1} \mathbf{t}_j)+1)/2]\pi}{p(\mathbf{t}_{j-1} \mathbf{t}_j)+1}}{p(\mathbf{t}_{j-1} \mathbf{t}_j)+1}
		=  \frac{\sin^2\frac{\pi}{3}}{3}=\frac{1}{4}\ \  \text{with } a_{j}=1 \ \text{and }\mathbf{t}_{j-1} \mathbf{t}_j=\mathbf{21}. 
	\end{equation*}
	
	Multiplying all these local estimates in \eqref{es1} for $i=1,2,...,s+1$ and those in \eqref{es2} for $i=1,2,...,s$, we obtain
	\begin{equation*}
		|B^\alpha_w|\ge\frac{1}{t_2^n}\Big(\frac{t_2}{4}\Big)^{\sum_{i=1}^{s}\left[\frac{m_i+1}{2}\right]}
		\prod_{a_i\ge2}\frac1{a_if(a_i)} .
	\end{equation*}
	
	\noindent{\bf Step 2:} We construct the admissible word $v\in\Omega^{\alpha}_{n}$ such that \eqref{upp} holds. 
	
	At first we define $\hat{\mathbf{t}}_i,i =1,2,\cdots,\tau_s$ by induction. For any $i=1,2,\cdots,\tau_1$, we discuss three cases:
	
	{\rm Case 1:} $n_1=0$. Define
	\begin{equation*}
		\hat{\mathbf{t}}_1\cdots \hat{\mathbf{t}}_{m_1}=\begin{cases}
			(\mathbf{1}\mathbf{2}\mathbf{1})(\mathbf{2}\mathbf{3}\mathbf{1})^{(m_1-3)/3},\ &\text{if}\ m_1=0(\text{mod}\ 3),\\
			\mathbf{1}(\mathbf{2}\mathbf{3}\mathbf{1})^{(m_1-1)/3},\ &\text{if}\ m_1=1(\text{mod}\ 3),\\
			\mathbf{2}\mathbf{1}(\mathbf{2}\mathbf{3}\mathbf{1})^{(m_1-2)/3},\ &\text{if}\ m_1=2(\text{mod}\ 3).
		\end{cases}
	\end{equation*}
	
	{\rm Case 2:} $n_1>0$ is odd. Define
	\begin{equation*}
		\hat{\mathbf{t}}_1\cdots \hat{\mathbf{t}}_{n_1}\cdot\hat{\mathbf{t}}_{n_1+1}\cdots\hat{\mathbf{t}}_{\tau_1}=\begin{cases}
			(\mathbf{2}\mathbf{1})^{(n_1-1)/2}\mathbf{2}\cdot(\mathbf{1}\mathbf{2}\mathbf{1})(\mathbf{2}\mathbf{3}\mathbf{1})^{(m_1-3)/3},\ &\text{if}\ m_1=0(\text{mod}\ 3),\\
			(\mathbf{2}\mathbf{1})^{(n_1-1)/2}\mathbf{2}\cdot\mathbf{1}(\mathbf{2}\mathbf{3}\mathbf{1})^{(m_1-1)/3},\ &\text{if}\ m_1=1(\text{mod}\ 3),\\
			(\mathbf{2}\mathbf{1})^{(n_1-1)/2}\mathbf{2}\cdot\mathbf{3}\mathbf{1}(\mathbf{2}\mathbf{3}\mathbf{1})^{(m_1-2)/3},\ &\text{if}\ m_1=2(\text{mod}\ 3).
		\end{cases}
	\end{equation*}
	
	{\rm Case 3:} $n_1>0$ is even. Define
	\begin{equation*}
		\hat{\mathbf{t}}_1\cdots \hat{\mathbf{t}}_{n_1}\cdot\hat{\mathbf{t}}_{n_1+1}\cdots\hat{\mathbf{t}}_{\tau_1}=\begin{cases}
			(\mathbf{2}\mathbf{1})^{(n_1-2)/2}	\mathbf{2}\mathbf{1}\cdot(\mathbf{2}\mathbf{3}\mathbf{1})^{m_1/3},\quad &\text{if}\ m_1=0(\text{mod}\ 3),\\
			(\mathbf{2}\mathbf{1})^{(n_1-2)/2}	\mathbf{2}\mathbf{3}\cdot\mathbf{1}(\mathbf{2}\mathbf{3}\mathbf{1})^{(m_1-1)/3},\quad &\text{if}\ m_1=1(\text{mod}\ 3),\\
			(\mathbf{2}\mathbf{1})^{(n_1-2)/2}\mathbf{2}\mathbf{1}	\cdot\mathbf{2}\mathbf{1}(\mathbf{2}\mathbf{3}\mathbf{1})^{(m_1-2)/3},\quad &\text{if}\ m_1=2(\text{mod}\ 3).
		\end{cases}
	\end{equation*}
	
	Assume $\hat{\mathbf{t}}_j$ is already defined for $j\leq\tau_{i-1}$. Now define $\hat{\mathbf{t}}_j$, $j=\tau_{i-1}+1,\cdots,\tau_i$ as follows.
	
	{\rm Case 1:} $n_i$ is odd. Define
	\begin{equation*}\label{type-1}
		\hat{\mathbf{t}}_{\tau_{i-1}+1}\cdots \hat{\mathbf{t}}_{\tau_{i-1}+n_i}\cdots \hat{\mathbf{t}}_{\tau_i}=\begin{cases}
			(\mathbf{2}\mathbf{1})^{(n_i-1)/2}\mathbf{2}\cdot(\mathbf{1}\mathbf{2}\mathbf{1})(\mathbf{2}\mathbf{3}\mathbf{1})^{(m_i-3)/3},\ &\text{if}\ m_i=0(\text{mod}\ 3),\\
			(\mathbf{2}\mathbf{1})^{(n_i-1)/2}\mathbf{2}\cdot\mathbf{1}(\mathbf{2}\mathbf{3}\mathbf{1})^{(m_i-1)/3},\ &\text{if}\ m_i=1(\text{mod}\ 3),\\
			(\mathbf{2}\mathbf{1})^{(n_i-1)/2}\mathbf{2}\cdot\mathbf{3}\mathbf{1}(\mathbf{2}\mathbf{3}\mathbf{1})^{(m_i-2)/3},\ &\text{if}\ m_i=2(\text{mod}\ 3).
		\end{cases}
	\end{equation*}
	
	{\rm Case 2:} $n_i$ is even. Define
	\begin{equation*}\label{type-2}
		\hat{\mathbf{t}}_{\tau_{i-1}+1}\cdots \hat{\mathbf{t}}_{\tau_{i-1}+n_i}\cdots \hat{\mathbf{t}}_{\tau_i}=\begin{cases}
			(\mathbf{2}\mathbf{1})^{(n_i-2)/2}\mathbf{2}\mathbf{1}\cdot(\mathbf{2}\mathbf{3}\mathbf{1})^{m_i/3},\ &\text{if}\ m_i=0(\text{mod}\ 3),\\
			(\mathbf{2}\mathbf{1})^{(n_i-2)/2}\mathbf{2}\mathbf{3}\cdot\mathbf{1}(\mathbf{2}\mathbf{3}\mathbf{1})^{(m_i-1)/3},\ &\text{if}\ m_i=1(\text{mod}\ 3),\\
			(\mathbf{2}\mathbf{1})^{(n_i-2)/2}\mathbf{2}\mathbf{1}\cdot\mathbf{2}\mathbf{1}(\mathbf{2}\mathbf{3}\mathbf{1})^{(m_i-2)/3},\ &\text{if}\ m_i=2(\text{mod}\ 3).
		\end{cases}
	\end{equation*}
	Thus by induction we have defined $\hat{\mathbf{t}}_i$ for $i=1,2,\cdots,\tau_s$. Finally, if $n_{s+1}>0$, define 
	\begin{equation*}
		\hat{\mathbf{t}}_{\tau_s+1}\cdots\hat{\mathbf{t}}_{\tau_s+n_{s+1}}=\begin{cases}
			(\mathbf{2}\mathbf{1})^{(n_{s+1}-1)/2}\mathbf{2},\quad &\text{if}\ n_{s+1}\text{ is odd},\\
			(\mathbf{2}\mathbf{1})^{n_{s+1}/2},\quad &\text{if}\ n_{s+1}\text{ is even}.
		\end{cases}
	\end{equation*}
	
	By the admissible rules (\eqref{admissible-T-A} and \eqref{admissible-A-A}), we define the admissible word $v\in\Omega^{\alpha}_{n}$ as
	\begin{equation*}
		v=\begin{cases}
			\mathbf{3}v_1\cdots v_{n-1}v_{n},\quad&\text{if $n_1=0$ and $m_1 \neq 2$(mod 3)},\\
			\mathbf{1}v_1\cdots v_{n-1}v_{n},\quad&\text{otherwise},
		\end{cases}
	\end{equation*} 
	where $v_i=(\hat{\mathbf{t}}_i,1)_{a_i}, 1\leq i\leq n$. 

	For the word $v$, we apply the upper estimate in Lemma \ref{estimate-band}. Thus
	\begin{equation}\label{bv}
		|B_v^\alpha|\le\frac{4}{t_1^n}
		\prod_{\hat{\mathbf{t}}_i=2}\frac{(2t_1)^{2-a_i}}{2}
		\prod_{\hat{\mathbf{t}}_i\ne2}
		\frac{\sin^2\frac{\pi}{p(\hat{\mathbf{t}}_{i-1}\hat{\mathbf{t}}_i)+1}}
		{p(\hat{\mathbf{t}}_{i-1}\hat{\mathbf{t}}_i)+1}.
	\end{equation}

	We first estimate the contribution of the blocks $A_i$ in \eqref{es7}.  On such a block one has
	$a_j\ge 2$.  According to the construction of $\hat{\mathbf{t}}_j$, the symbols in $A_i$ are
	arranged as alternating patterns of the form $\mathbf{21}$, possibly with one terminal symbol. Thus the number of indices in $A_i$ for which $\hat{\mathbf{t}}_j\ne 2$ is at least
	$[ n_i/2]$, while the remaining indices have $\hat{\mathbf{t}}_j=2$.
	
	If $\hat{\mathbf{t}}_j=2$, then the corresponding factor in \eqref{bv} is
	\begin{equation*}
		\frac{(2t_1)^{2-a_j}}{2}\le \frac{1}{a_j f(a_j)}
	\end{equation*}
	by the last inequality in \eqref{imp-2}. 
	If $\hat{\mathbf{t}}_j\ne 2$, then $\hat{\mathbf{t}}_{j-1} \hat{\mathbf{t}}_j=\mathbf{2}\mathbf{1}$ or $\mathbf{2}\mathbf{3}$, and hence the first two inequalities in \eqref{imp-2} give
	\begin{equation*}
		\frac{\sin^2\frac{\pi}{p(\hat{\mathbf{t}}_{j-1} \hat{\mathbf{t}}_j)+1}}
		{p(\hat{\mathbf{t}}_{j-1} \hat{\mathbf{t}}_j)+1}\leq
		\begin{cases}
			\frac1{3a_jf(a_j)},\quad&\text{if}\ \hat{\mathbf{t}}_{j-1} \hat{\mathbf{t}}_j=\mathbf{2}\mathbf{1},\\
			\frac1{2a_jf(a_j)} ,\quad&\text{if}\ \hat{\mathbf{t}}_{j-1} \hat{\mathbf{t}}_j=\mathbf{2}\mathbf{3}.
		\end{cases}
	\end{equation*}
	Furthermore, the second case is possible only at $\tau_{i-1}+n_i$ when $n_i>0$ is even and $m_i=1\pmod 3$ by construction of $v$. Therefore, we get
	\begin{equation}
		\text{contribution of }A_i \le 
		\left(\frac1{3}\right)^{[\frac{n_i}{2}]}\left(1+\frac{1_{\{2|n_i,\ m_i=1\pmod3\}}}{2}\right)
		\prod_{a_j\in A_i}\frac{1}{a_j f(a_j)} . \label{es3}
	\end{equation}
	
	It remains to estimate the block $1^{m_i}$. 
	
	When $n_i>0$ is odd, by construction, $1^{m_i}$ is arranged as
	\begin{equation}
		(\mathbf{121})(\mathbf{231})^{(m_i-3)/3},\qquad
		\mathbf{1}(\mathbf{231})^{(m_i-1)/3},\qquad
		\mathbf{31}(\mathbf{231})^{(m_i-2)/3}, \label{odd}
	\end{equation}
	according to $m_i= 0,1,2\pmod 3$. For each pattern $\mathbf{231}$, the
	corresponding factors in \eqref{bv} are bounded by
	\[
	\frac{t_1}{4}.
	\]
	by noting that $(2t_1)^{2-a_j}/2=t_1$ for $a_{j}=1$ and 
	\begin{equation*}
		\frac{\sin^2\frac{\pi}{p(\hat{\mathbf{t}}_{j-1} \hat{\mathbf{t}}_j)+1}}
		{p(\hat{\mathbf{t}}_{j-1} \hat{\mathbf{t}}_j)+1}=\frac{1}{2} \ \text{with } a_{j}=1 \ \text{and }\hat{\mathbf{t}}_{j-1} \hat{\mathbf{t}}_j=\mathbf{23},\ \mathbf{31}.
	\end{equation*}
	Similarly, the corresponding factors in \eqref{bv} of the initial pattern $\mathbf{121}$ (in \eqref{odd}) are bounded by $\frac{t_1}{16}$ since
	\begin{equation*}
		\frac{\sin^2\frac{\pi}{p(\hat{\mathbf{t}}_{j-1} \hat{\mathbf{t}}_j)+1}}
		{p(\hat{\mathbf{t}}_{j-1} \hat{\mathbf{t}}_j)+1}=\frac{1}{4} \ \ \  \text{with } a_{j}=1 \ \text{and }\hat{\mathbf{t}}_{j-1} \hat{\mathbf{t}}_j=\mathbf{21}.
	\end{equation*}
	And, for the initial patterns $\mathbf{1}$, $\mathbf{31}$ in \eqref{odd}, the corresponding factors in \eqref{bv} are upper bounded by $1/4$.
	Therefore, the contribution of $1^{m_i}$ in \eqref{bv} is bounded by
	\begin{equation}
		\max\left\{\frac{t_1}{16}\left(\frac{t_1}{4}\right)^{\frac{m_i-3}{3}},
		\frac{1}{4}\left(\frac{t_1}{4}\right)^{\frac{m_i-1}{3}},
		\frac{1}{4}\left(\frac{t_1}{4}\right)^{\frac{m_i-2}{3}}\right\}
		\leq\frac{1}{4}\left(\frac{t_1}{4}\right)^{\left[\frac{m_i+1}{3}\right]} . \label{es4}
	\end{equation}
	
	When $n_i$ is even, by construction, $1^{m_i}$ is arranged as
	\begin{equation*}
		(\mathbf{231})^{m_i/3},\qquad
		\mathbf{1}(\mathbf{231})^{(m_i-1)/3},\qquad
		\mathbf{21}(\mathbf{231})^{(m_i-2)/3}
	\end{equation*}
	according to $m_i=0,1,2\pmod3$.	Similarly, the contribution of $1^{m_i}$ in \eqref{bv} is bounded by
	\begin{equation}
		\max\left\{\left(\frac{t_1}{4}\right)^{\frac{m_i}{3}},
		\frac{1}{2}\left(\frac{t_1}{4}\right)^{\frac{m_i-1}{3}},
		\frac{t_1}{4}\left(\frac{t_1}{4}\right)^{\frac{m_i-2}{3}}\right\}
		\leq\left(1-\frac{1_{\{m_i=1\pmod3\}}}{2}\right)
		\left(\frac{t_1}{4}\right)^{\left[\frac{m_i+1}{3}\right]} . \label{es5}
	\end{equation}
	
	Combining \eqref{es3}, \eqref{es4} and \eqref{es5}, we know that the contributions of $A_i$ and $1^{m_i}$ in \eqref{bv} are bounded by 
	\begin{equation}\label{es6}		\left(\frac1{3}\right)^{\frac{n_i}{2}}\left(\frac{t_1}{4}\right)^{\left[\frac{m_i+1}{3}\right]}
\prod_{a_j\in A_i}\frac{1}{a_j f(a_j)},
	\end{equation}
	and it is easy to check that \eqref{es6} holds for the case $n_1=0$.
	
	Multiplying \eqref{es6} for $i=1,\ldots,s$ and \eqref{es3} for $i=s+1$, we obtain
	\begin{equation*}
		|B_v^\alpha|\le \frac{4\sqrt3}{t_1^n} \left(\frac{t_1}{4}\right)^{\sum_{i=1}^s
			\left[\frac{m_i+1}{3}\right]}
		\left(\frac1{\sqrt3}\right)^{\sum_{i=1}^{s+1}n_i} 
		\prod_{a_j\ge 2}\frac{1}{a_j f(a_j)},
	\end{equation*}
	which proves \eqref{upp}.

	\smallskip
	\noindent{\bf Acknowledgement}.
	The authors extend their sincere gratitude to Yanhui Qu for his meticulous review of the preliminary manuscript and insightful discussions that significantly contributed to the refinement of this work. 

	\appendix
	\section{A brief proof of Lemma \ref{estimate-band}}\label{exp}
	In this Appendix, we mainly explain the origin of Lemma \ref{estimate-band}. We follow the method in the proof of \cite[Lemma 3.7]{LQW2014} and explain the necessary modifications. For the specific details of that proposition, the reader is referred to \cite[Propositions 3.1, 3.2 and 3.3]{FLW2011}.
	
	We write $B_w:=B^{\alpha}_{w}$ and $\hat{B}_w:=\hat{B}^{\alpha}_{w}$ for simplicity, and we omit the frequency $\alpha$ when no confusion arises.
	
	Given $w\in\Omega^{\alpha}_n$, write $w=w_1\cdots w_n$ and $w|_m=w_1\cdots w_m$
	for $m=1,\cdots,n.$ Write $B_m=B^{\alpha}_{w|_m}.$ Then
	for any $k\le n$
	$$
	B_n\subseteq B_{n-1}\subseteq\cdots\subseteq B_{k}
	$$
	is  a sequence of spectral generating bands from order $n$ to $k$. We
	call the sequence $(B_i)_{i=k}^n$ an {\em initial ladder}, and the
	bands $B_i(k\le i\le n)$ are called  {\it initial rungs}. Now we are going
	to modify the initial ladder by the following way: for any $i(k\le
	i\le n-1)$,
	\begin{itemize}
		\item if   $B_i$ is of $(i,\mathbf{1})$-type with $a_{i+1}=1$:  delete the rung
		$B_{i+1}$ (in this case $B_{i+1}$ must be $(i+1,\mathbf{2})$-type and $B_{i+1}=B_i$);
		
		\item if $B_i$ is of $(i,\mathbf{1})$-type with $a_{i+1}=2$: change
		nothing;
		
		\item if $B_i$ is of $(i,\mathbf{1})$-type with $a_{i+1}>2$: add rungs
		$(B_{(i,p)})_{p=2}^{a_{i+1}-1}$ between $B_i$ and $B_{i+1}$ :
		$$B_{i+1}=B_{(i,a_{i+1})}\subset B_{(i,a_{i+1}-1)}\subset\cdots\subset
		B_{(i,2)}\subset B_{(i,1)}=B_i;$$
		where $B_{(i,p)}$ is the unique band in $\sigma_{(i,p)}$ which is included in $B_i$.
		\item if $B_i$ is of $(i,\mathbf{2})$ or $(i,\mathbf{3})$-type: change nothing.
	\end{itemize}
	By this way we get  a unique modified ladder which we relabel as
	$$B_n=\hat{B}_m\subset\cdots\subset\hat{B}_1\subset \hat{B}_{0}=B_{k}.$$
	We call $(\hat{B}_i)_{i=0}^m$ the {\em modified ladder}, and we
	denote the corresponding generating polynomials by
	$(\hat{h}_i)_{i=0}^m$. For $i=0,\cdots,m-1$ define
	\begin{equation}\label{pvalue}
		(p_i,l_i)=
		\begin{cases}
			(p(\mathbf{t}_{j-1}\mathbf{t}_j),l), &
			\begin{array}{l}
				\text{ if }  (\hat B_{i},\hat B_{i+1})=(B_{j-1},B_{j})\text{ for some } j \text{ and } w_j=(\mathbf{t}_j,l),
			\end{array} \\
			(1,1), & \text{  otherwise},
		\end{cases}
	\end{equation}
where $p$ is defined in \eqref{def-p}.	We call $(p_i)_{i=0}^{m-1}$ and $(l_i)_{i=0}^{m-1}$
	the {\it type sequence } and {\it index sequence } of the modified ladder.
	
	Let $p\ge1$, $1\le l\le p$. Define
	$$I_{p,l}:=\left\{2\cos\frac{l+c}{p+1}\pi\ :\ |c|\le\frac{1}{150}\mbox{ and }
	\ \left|S_{p+1}\left(2\cos\frac{l+c}{p+1}\pi\right)\right|\le\frac{1}{54}
	\right\},$$
	where $S_{p}(x)$ is the Chebyshev polynomial defined by
	\begin{align*}
		S_0(x)\equiv0,\quad S_1(x)\equiv1,\quad S_{p+1}(x)=xS_{p}(x)-S_{p-1}(x),\ p\geq1.
	\end{align*}
By induction, we know that 
\begin{equation}\label{e.sp}
 S_p(2\cos\theta)=\frac{\sin (p\theta)}{\sin\theta},\ \ \theta\in(0,\pi).
\end{equation}
	\begin{proposition}\label{index}
		Assume $\lambda\ge240$. Let $(\hat{B}_i)_{i=0}^m$ be a modified ladder,
		$(\hat{h}_i)_{i=0}^{m}$ the corresponding generating polynomials,
		and $(p_i)_{i=0}^{m-1}, (l_i)_{i=0}^{m-1}$ be the type sequence and index sequence respectively. Then for any
		$0\le i<m$,
		$$\hat{h}_i(\hat{B}_{i+1})\subset I_{p_i,l_i}.$$
	\end{proposition}
	\begin{proof}
		The proof of this proposition can be directly referred to \cite[Proposition 3.1]{FLW2011}, provided that the corresponding required that $\lambda\geq240$. We now briefly explain this point.
		
		Since $\lambda\geq240$, then by the proof of \cite[Proposition 3.1]{FLW2011}, we know that
		\begin{equation*}
			|S_{p_i+1}(\hat{h}_{i}(x))|\leq\frac{4}{\lambda-4}\leq\frac{1}{54},
		\end{equation*}
		and for any $x\in\hat{B}_{i+1}$, there exist an integer $l$($1\leq l\leq p_i$) and a unique $c$ with $|c|<1$ such that $\hat{h}_i(x)=2\cos((l+c)/(p_i+1))\pi$. Then, we have from \eqref{e.sp} that
		\begin{equation*}
			|S_{p_i+1}(\hat{h}_{i}(x))|=\frac{|\sin c\pi|}{\sin((l+c)/(p_i+1))\pi}\leq\frac{1}{54}.
		\end{equation*}
		This implies that $|\sin c\pi|\leq1/54$, and we have $|c|\leq1/150$ or $|c\pm1|\leq1/150.$ Applying the argument of \cite[Proposition 3.1]{FLW2011} yields $|c|\leq 1/150$, and the statement follows.
	\end{proof}
	
	We collect some useful estimations of Chebyshev polynomials on the interval $I_{p,l}$,
	which is essentially the  \cite[Proposition 7]{LW2004}, see also \cite[Proposition 3.2]{FLW2011}.
	\begin{proposition}\label{keyLW}
		Fix $p\ge1$, $1\le l\le p$. For any $t\in I_{p,l}$,
		\begin{align}\label{new-est}
			&	|S_{p+1}(t)|\le\frac{1}{4},\quad |S_p(t)|\le \frac{5}{4},\quad |S'_p(t)|\le2|S'_{p+1}(t)|,\notag\\
			&		\frac{12(p+1)}{25}\csc^2\frac{l\pi}{p+1}
			\le |S'_{p+1}(t)|\le \frac{13(p+1)}{25}\csc^2\frac{l\pi}{p+1}.
		\end{align}
	\end{proposition}
	\begin{proof}
		We only prove \eqref{new-est}; the remaining formulas can be found in \cite[Proposition 3.2]{FLW2011}.
		
		By the definition of $I_{p,l}$, $0< |c|<1/150$. Then we have 
		\begin{equation*}
			\left|(p+1)\tan\frac{l+c}{p+1}\pi\right|\geq(1-|c|)\pi=\frac{149\pi}{150};\quad \frac{149}{150}\frac{l\pi}{p+1}\leq\frac{l+c}{p+1}\pi\leq\frac{151}{150}\frac{l\pi}{p+1}.
		\end{equation*}
		Note that $\sin x/x$ is decreasing on $(0,\pi)$, then 
		\begin{equation*}
			\frac{150^2}{151^2}\csc^2\frac{l\pi}{p+1}\leq\csc^2\frac{l+c}{p+1}\leq\frac{150^2}{149^2}\csc^2\frac{l\pi}{p+1}.
		\end{equation*}
		
By \eqref{e.sp}, we have
		\begin{align*}
			S'_{p+1}\left(2\cos\frac{l+c}{p+1}\pi\right)&=\frac{(-1)^{l+1}(p+1)\cos c\pi}{2\sin^2((l+c)/(p+1))\pi}\left[1-\frac{\tan c\pi}{(p+1)\tan((l+c)/(p+1))\pi}\right]\\
			&=\frac{(-1)^{l+1}(p+1)}{2\sin^2((l+c)/(p+1))\pi}\left[\cos c\pi-\frac{\sin c\pi}{(p+1)\tan((l+c)/(p+1))\pi}\right].
		\end{align*}
		This combines with $|\sin c\pi|\leq1/54$, now we get that
		\begin{align*}
			\left|S'_{p+1}(2\cos\frac{l+c}{p+1}\pi)\right|&\leq\frac{(p+1)\csc^2\frac{l+c}{p+1}\pi}{2}\left(|\cos c\pi|+\frac{|\sin c\pi|}{|(p+1)\tan((l+c)/(p+1))\pi|}\right)\\
			&\leq\frac{p+1}{2}\csc^2\frac{l\pi}{p+1}\cdot\frac{150^2}{149^2}\left(1+\frac{1/54}{149\pi/150}\right)\\
			&\leq\frac{13(p+1)}{25}\csc^2\frac{l\pi}{p+1},
		\end{align*}
		and
		\begin{align*}
			\left|S'_{p+1}(2\cos\frac{l+c}{p+1}\pi)\right|&\geq\frac{(p+1)\csc^2\frac{l+c}{p+1}\pi}{2}\left(|\cos c\pi|-\frac{|\sin c\pi|}{|(p+1)\tan((l+c)/(p+1))\pi|}\right)\\
			&\geq\frac{p+1}{2}\csc^2\frac{l\pi}{p+1}\cdot\frac{150^2}{151^2}\left(\cos\frac{\pi}{150}-\frac{1/54}{149\pi/150}\right)\\
			&\geq\frac{12(p+1)}{25}\csc^2\frac{l\pi}{p+1}.
		\end{align*}
		Now the result follows.
	\end{proof}
	Next, we directly use the above proposition and obtain the following conclusion; for details, see the proof of \cite[Proposition 3.3]{FLW2011}. We remark that, compared with \cite[Proposition 3.3]{FLW2011}, we only modify \eqref{new-est}, so the coefficients of $\lambda$ in the following proposition are modified accordingly.
	
	\begin{proposition}\label{lm-2}
		Assume $\lambda\ge240$. Let $(\hat{B}_i)_{i=0}^m$ be a modified ladder,
		$(\hat{h}_i)_{i=0}^m$, $(p_i)_{i=0}^{m-1}$ and $(l_i)_{i=0}^{m-1}$
		be the corresponding generating polynomials, type sequence and index sequence.
		For any $0\le i< m$, $x\in \hat{B}_{i+1}$, we have,
		$$\frac{12(\lambda-8)}{25}(p_i+1)\csc^2\frac{l_i\pi}{p_i+1}
		\le\left|\frac{\hat{h}'_{i+1}(x)}{\hat{h}'_{i}(x)}\right|
		\le \frac{13(\lambda+8)}{25}(p_i+1)\csc^2\frac{l_i\pi}{p_i+1}.$$
	\end{proposition}
	
	\smallskip
	
	\noindent {\bf Proof of Lemma \ref{estimate-band}.}\
	Given $w\in \Omega^{\alpha}_n.$  Consider the initial ladder $(B_i)_{i=0}^n$ with $B_0$ the unique band in $\mathcal{B}^{\alpha}_0$ containing $B_w$ and $B_n=B_w.$
	Let $(\hat{B}_i)_{i=0}^m$ be the related  modified ladder and
	$(\hat{h}_i)_{i=0}^m$, $(p_i)_{i=0}^{m-1}$ and $(l_i)_{i=0}^{m-1}$
	be the corresponding generating polynomials, type sequence and index
	sequence.
	Since $\hat h_m(\hat B_m)=[-2,2]$, there exists $x_0\in\hat B_m$
	such that $|\hat h_m^\prime(x_0)| |\hat B_m|=4.$
	Also, note that
	$|\hat{h}_{0}^\prime|\equiv 1$,
	then by Proposition \ref{lm-2}, the definition of  modified ladder and \eqref{pvalue}, we have 
	\begin{align*}
		|B_w|&=|\hat{B}_m|= 4\frac{|\hat h_0^\prime(x_0)|}{|\hat h_m^\prime(x_0)|}\le 4
		\prod_{i=0}^{m-1}\frac{\sin^2\frac{l_i\pi}{p_i+1}}{(p_i+1)t_1}\ \text{(where $t_1=12(\lambda-8)/25$)}\\
&\le4\prod_{\mathbf{t}_j=\mathbf{2}}\frac{1}{(2t_1)^{a_j-1}}\cdot \prod_{\mathbf{t}_{j}\neq\mathbf{2}}
		\frac{\sin^2\frac{l_j\pi}{p_j+1}}{(p(\mathbf{t}_{j-1}\mathbf{t}_j)+1)t_1}\\
		&\leq\frac{4}{t_1^n}\prod_{\mathbf{t}_j=\mathbf{2}}\frac{(2t_1)^{2-a_j}}{2}\prod_{\mathbf{t}_j\neq\mathbf{2}}\frac{\sin^2\frac{l_j\pi}{p(\mathbf{t}_{j-1}\mathbf{t}_j)+1}}{(p(\mathbf{t}_{j-1}\mathbf{t}_j)+1)}.
	\end{align*}
Let $t_2=13(\lambda+8)/25$.	Similarly, we have
	\begin{align*}
		|B_w|&\ge  4
		\prod_{i=0}^{m-1}\frac{\sin^2\frac{l_i\pi}{p_i+1}}{(p_i+1)t_2}\ge4\prod_{\mathbf{t}_{j}\neq\mathbf{2}}\frac{1}{(2t_2)^{a_j-1}}\cdot \prod_{\mathbf{t}_j\neq\mathbf{2}}\frac{\sin^2\frac{l_j\pi}{p(\mathbf{t}_{j-1}\mathbf{t}_j)+1}}{(p(\mathbf{t}_{j-1}\mathbf{t}_j)+1)t_2}\\
		&\ge\frac{			4}{t_2^n}\prod_{\mathbf{t}_j=\mathbf{2}}\frac{(2t_2)^{2-a_j}}{2}\prod_{\mathbf{t}_j\neq\mathbf{2}}\frac{\sin^2\frac{l_j\pi}{p(\mathbf{t}_{j-1}\mathbf{t}_j)+1}}{(p(\mathbf{t}_{j-1}\mathbf{t}_j)+1)}.
	\end{align*}
	Now we complete the proof of this lemma.
	\hfill $\Box$

	\smallskip
	\noindent{\bf Declarations}
	
	\textbf{Funding:} J. Cao was partially supported by Nankai Zhide Foundation. Z. Yu was supported by the Natural Science Foundation of Hunan Province, China (No.2025JJ60039).
	
	\textbf{Conflict of interest:} The authors declare that they have no conflict of interest.
	
	\textbf{Data availability statement:} Data sharing is not applicable to this article as no datasets were generated or analysed during the current study.
	

\end{document}